\documentclass[11pt,reqno]{amsart}
\usepackage{amsmath}
\usepackage{amsfonts}
\usepackage{amssymb}
\usepackage{amsthm}
\usepackage{amsaddr}
\usepackage{graphicx}
\usepackage{empheq}
\usepackage{indentfirst}
\usepackage{cite}
\usepackage{mathrsfs}
\usepackage{cases}
\usepackage{graphics}
\usepackage{xcolor}
\usepackage{bm}
\usepackage{times}
\usepackage{hyperref}
\hypersetup{colorlinks=true, linkcolor=blue, citecolor=purple, urlcolor=blue}

\usepackage[left=2.4cm, right=2.4cm, top=2cm, bottom=2cm]{geometry}

\newtheorem{theorem}{\textbf{Theorem}}[section]
\newtheorem{lemma}{\textbf{Lemma}}[section]
\newtheorem{proposition}{\textbf{Proposition}}[section]
\newtheorem{corollary}{\textbf{Corollary}}[section]
\newtheorem{remark}{\textbf{Remark}}[section]
\newtheorem{definition}{\textbf{Definition}}[section]
\allowdisplaybreaks[4] %
\def\be{\begin{equation}}
\def\ee{\end{equation}}
\def\bt{\begin{theorem}}
\def\et{\end{theorem}}
\def\bl{\begin{lemma}}
\def\el{\end{lemma}}
\def\br{\begin{remark}}
\def\er{\end{remark}}
\def\bp{\begin{proposition}}
\def\ep{\end{proposition}}
\def\bc{\begin{corollary}}
\def\ec{\end{corollary}}
\def\bd{\begin{definition}}
\def\ed{\end{definition}}
\def\d{{\rm d}}
\def\l{\langle}
\def\r{\rangle}

\def\H{\bm{H}}

\def\L{\bm{L}}
\def\ddt{\frac{\d}{\d t}}

\def\n{{\bm{n}}}

\def\div{\mathrm{div}}
\def\J{\mathbf{J}}

\newcommand{\habil}[1]{}

\newcommand{\uloc}{\operatorname{uloc}}

\newcommand{\R}{\ensuremath{\mathbb{R}}}

\def\uu{\bm{u}}

\def\vv{\bm{v}}

\def\yy{\bm{y}}

\def \au {\rm}
\def \ti {\it}
\def \jou {\rm}
\def \bk {\it}
\def \no#1#2#3 {{\bf #1} (#3), #2.}
\def \eds#1#2#3 {#1, #2, #3.}

\makeatletter
\def\@settitle{\begin{center}%
  \baselineskip14\p@\relax
    \huge
  \@title
  \end{center}%
}
\makeatother

\begin{document}

\title[Navier--Stokes/Cahn--Hilliard system for two-phase flows with chemotaxis]{
\emph{On a Thermodynamically Consistent Diffuse-Interface Model for Incompressible Two-Phase Flows \\ with Chemotaxis and Mass Transport
}}

\author[A. Giorgini, J.-N. He \& H. Wu]{
Andrea Giorgini$^\dagger$,\ \ \
Jingning He$^\ddagger$,\ \ \
Hao Wu$^\ast$
}

\address{$^\dagger$Dipartimento di Matematica \\
Politecnico di Milano \\
Milano 20133, Italy}
\email{andrea.giorgini@polimi.it}

\address{$^\ddagger$School of Mathematics \\
Hangzhou Normal University \\
Hangzhou 311121, P. R. China}
\email{hejingning@hznu.edu.cn}

\address{$^\ast$School of Mathematical Sciences\\
Fudan University\\
Shanghai 200433, P. R. China}
\email{haowufd@fudan.edu.cn}

\begin{abstract}
We investigate a hydrodynamic system of Navier--Stokes/Cahn--Hilliard type, which describes the motion of a two-phase flow of two incompressible fluids with unmatched densities coupled with a soluble chemical species. Derived from Onsager's variational principle, this thermodynamically consistent diffuse-interface model incorporates both the chemotaxis effects induced by the chemical species and the mass transport processes within the mixture. For the two-dimensional initial-boundary value problem, we establish the existence of global finite energy solutions and global weak solutions, using a suitable approximation scheme combined with compactness methods. Next, by carefully analyzing three decoupled subsystems and employing a bootstrap argument, we prove the existence and uniqueness of a global strong solution for sufficiently regular initial data, as well as the propagation of regularity for global weak solutions. In particular, we show that the density of the chemical substance stays bounded for all time if its initial datum is bounded. This implies a significant distinction from the classical Keller--Segel system: diffusion driven by the chemical potential gradient can prevent the formation of concentration singularities.
\smallskip

\smallskip

\end{abstract}
%
%
%
\date{\today}
\maketitle
\tableofcontents

\section{Introduction}
\setcounter{equation}{0}
We consider a thermodynamically consistent diffuse interface model that describes the dynamics of a two-phase flow of two viscous incompressible Newtonian fluids with unmatched densities coupled with a soluble chemical species. The resulting hydrodynamic system reads as follows
\begin{equation}
 \label{NSCHc}
\begin{cases}
\partial_t ( \rho(\varphi)\vv) + \div \big( \vv \otimes \big( \rho(\varphi) \vv + \J ) \big) - \div \big( 2\nu(\varphi) D \vv \big) + \nabla P
= \mu \nabla \varphi + w \nabla \sigma,\\
\div \, \vv=0,\\
\partial_t \varphi +\vv\cdot \nabla \varphi = \div\big(m(\varphi)\nabla \mu\big),\\
\mu= -\varepsilon \Delta \varphi+ \dfrac{1}{\varepsilon} \Psi'(\varphi) + \beta'(\varphi) \sigma,\\
\partial_t \sigma + \vv \cdot \nabla \sigma -\div \left( \sigma \nabla w\right)=0,\\
w = \ln \sigma + \beta(\varphi),
\end{cases}
\end{equation}
in $\Omega \times (0,\infty)$, where $\Omega$ is a bounded domain in $\mathbb{R}^2$ with a sufficiently smooth boundary $\partial \Omega$. The system \eqref{NSCHc} is completed with the following boundary and initial conditions
\begin{equation}
\label{NSCHc-bic}
\begin{cases}
\vv=\mathbf{0}, \quad \partial_\n \varphi
= m(\varphi)\partial_\n \mu
= \sigma \partial_\n w=0 \quad &\text{on }  \partial \Omega \times (0,\infty),\\
\vv|_{t=0}=\vv_0, \quad \varphi|_{t=0}=\varphi_0, \quad \sigma|_{t=0}=\sigma_0 \quad &\text{in } \Omega.
\end{cases}
\end{equation}
Here, $\n$ is the unit outward normal vector on $\partial \Omega$, and $\partial_\n$ denotes the outer normal derivative on $\partial \Omega$.
The system \eqref{NSCHc} was originally introduced by Abels, Garcke and Gr\"{u}n \cite{AGG2012}, derived from mass balance laws and Onsager's variational principle \cite{Ons31}. It presents a diffuse interface description of an isothermal mixed flow with two incompressible immiscible constituents, where $\varepsilon>0$ is a (small) parameter related to the ``thickness'' of the partially mixing interfacial region. The model under investigation pertains to scenarios in which the chemical species is soluble and dilute within the solvent, thereby preventing the formation of a third phase in the binary fluid mixture. On the other hand, it incorporates effects due to the transport of the chemical species across free interfaces between the two fluids,
while neglecting its influence on interfacial surface tension.

\subsection{Description of the model }
The coupled system \eqref{NSCHc} is characterized by the following state variables: the volume-averaged velocity $\vv\colon \Omega\times [0,\infty)\to \R^2$, the pressure of the fluid mixture $P\colon \Omega \times [0,\infty)\to \R$, the difference between the volume fractions of the two fluid components $\varphi\colon \Omega \times [0,\infty)\to [-1,1]$, and the density of the chemical species $\sigma\colon \Omega \times [0,\infty)\to [0,\infty)$. The scalar functions $\mu, w: \Omega \times [0,\infty)\to \mathbb{R}$ are the so-called chemical potentials with respect to the state variables $\varphi$ and $\sigma$, respectively.

In the equations for the linear momentum, $D \vv$ stands for the symmetrized gradient of $\vv$, that is, $D \vv=\frac12 (\nabla \vv +(\nabla \vv)^T)$. The average density $\rho$ and the average viscosity $\nu$ of the fluid mixture are defined as
\begin{equation}
\label{meanr-v}
\rho(\varphi)= \widetilde{\rho}_1 \frac{1-\varphi}{2}+ \widetilde{\rho}_2 \frac{1+\varphi}{2},
\qquad
\nu(\varphi)=\widetilde{\nu}_1 \frac{1-\varphi}{2}+ \widetilde{\nu}_2 \frac{1+\varphi}{2},
\end{equation}
where $\widetilde{\rho}_1$, $\widetilde{\rho}_2$ and $\widetilde{\nu}_1$, $\widetilde{\nu}_2$ denote the positive homogeneous density and viscosity parameters of the two fluid components, respectively.
In the current model, the density flux consists of two parts: $\rho(\varphi)\bm{v}$ which describes the transport by the fluid velocity, and the relative flux $\J$ given by
\begin{equation}
\label{Jrhonu}
\J= -\rho'(\varphi) m(\varphi) \nabla \mu= -\frac{\widetilde{\rho}_2-\widetilde{\rho}_1}{2} m(\varphi) \nabla \mu.
\end{equation}
This additional flux term accounts for diffusion of the components relative to the mean velocity in the case of unmatched densities \cite{AGG2012}. It is crucial to ensure the thermodynamic consistency of the system and simply vanishes when $\widetilde{\rho}_1=\widetilde{\rho}_2$. The interfacial force term $\mu \nabla \varphi + w \nabla \sigma$ can be equivalently written as
$$
\mu \nabla \varphi + w \nabla \sigma = -\varepsilon \mathrm{div} (\nabla \varphi \otimes \nabla \varphi) + \nabla \left(\frac{\varepsilon}{2}|\nabla \varphi|^2+ \frac{1}{\varepsilon}\Psi(\varphi)+ \sigma \ln \sigma -\sigma +\beta(\varphi) \sigma \right) \! ,
$$
where the first term on the right-hand side accounts for the capillary force due to surface tension, and the second term in the gradient can be absorbed into the pressure.

The evolution of the phase-field variable $\varphi$ is governed by a convective Cahn--Hilliard equation, in which diffusion of fluid components is taken into account. The scalar function $m\colon [-1,1]\to [0,\infty)$ is a mobility coefficient that measures the diffusion strength and may depend on $\varphi$ in general. In the chemical potential $\mu$, the nonlinear function $\Psi$ denotes the homogeneous free energy density of the mixing. A physically relevant example is the Flory--Huggins potential \cite{CH}:
\begin{equation}
\label{Log}
\Psi_{\mathrm{FH}}(r)
=\frac{\theta}{2}\big[ (1+r)\ln(1+r)+(1-r)\ln(1-r)\big]-\frac{\theta_0}{2} r^2, \quad r \in (-1,1),
\end{equation}
where $\theta, \theta_0$ are positive parameters. When $\theta< \theta_0$, $\Psi_{\mathrm{FH}}$ presents a double-well structure, which is important for the phase separation phenomenon. On the other hand, the singular nature of $\Psi_{\mathrm{FH}}$ and its derivatives near $\pm 1$ ensures that the phase-field variable $\varphi$ stays in the physically reasonable interval $[-1,1]$. The function $\beta$ characterizes the interaction between the fluid mixture and the chemical species. An example of $\beta$ was given in \cite{AGG2012}, which reaches for $\varphi\leq -1$ or $\varphi\geq 1$ the values $\beta_1$ or $\beta_2$ (parameters appearing in Henry's jump condition at the interface separating the two phases), respectively.

Finally, the density function $\sigma$ satisfies a mass balance law that results in a convection-diffusion equation. For simplicity, we have set the mobility coefficient in the $\sigma$-equation to be one, i.e., a positive constant.

In this study, we take classical boundary conditions as in \eqref{NSCHc-bic}: the fluid velocity field satisfies a no-slip boundary condition $\vv=\mathbf{0}$, the homogeneous Neumann type boundary conditions $m(\varphi)\partial_\n \mu=\sigma \partial_\n w=0$ indicate that there is no flux of the fluid components as well as the chemical species through the boundary, and $\partial_\n \varphi=0$ describes a ``contact angle'' of $\pi/2$ of the diffused interface and the boundary of the domain. In this setting, the initial boundary value problem \eqref{NSCHc}--\eqref{NSCHc-bic} satisfies two fundamental properties, that is, mass conservation and energy dissipation, which are the basis for subsequent analysis:

(1) \emph{Mass conservation}. Integrating \eqref{NSCHc}$_3$ and 
\eqref{NSCHc}$_5$ over $\Omega$, after integration by parts, we obtain
\begin{align}
\frac{\d}{\d t} \int_\Omega \varphi\,\d x =0,\quad \frac{\d}{\d t} \int_\Omega \sigma\,\d x =0,\quad \forall\, t>0.
\end{align}

(2) \emph{Energy dissipation}. Define the total energy
\begin{align}
\mathcal{E}(\bm{v},\varphi,\sigma)
&=\int_{\Omega}\Big( \underbrace{\frac{1}{2}\rho(\varphi)|\boldsymbol{v}|^{2}}_{\text{kinetic energy}}
+ \underbrace{\frac{\varepsilon}{2}|\nabla \varphi|^{2}
+ \frac{1}{\varepsilon}\Psi(\varphi)}_{\text{mixing energy}}
+ \underbrace{\sigma(\ln \sigma-1)}_{\text{entropy}}
+ \underbrace{\beta(\varphi) \sigma}_{\text{interaction}}\Big)\, \mathrm{d} x.
\label{total-Energy-0}
\end{align}
A direction calculation (see Section \ref{EXIST-WEAK} for related details) leads to the following formal energy identity
\begin{align}
\frac{\d}{\d t}\mathcal{E}(\bm{v},\varphi,\sigma)
+
\int_{\Omega} \left(2 \nu(\varphi)|D \boldsymbol{v}|^{2}
+m(\varphi)|\nabla \mu|^{2} +  \sigma|\nabla w|^2 \right)\, \mathrm{d} x=0,\quad \forall\, t>0.
\label{energy-id-0}
\end{align}

\subsection{Related literature}
When the interaction with chemical species is neglected, the system \eqref{NSCHc} reduces to the Navier--Stokes/Cahn--Hilliard system for incompressible two-phase flows with unmatched densities derived in \cite{AGG2012}. This model constitutes a thermodynamically consistent extension of the well-known ``Model H'' \cite{GPV96,HH77}, which is restricted to incompressible two-phase flows with ``matched'' densities, i.e., the density of the mixture is constant ($\widetilde{\rho}_1=\widetilde{\rho}_2$ in \eqref{meanr-v})
Many efforts in the literature have been dedicated to generalizing ``Model H'' to the case of unmatched densities, which is important in applications, see, for instance, \cite{Aki14,Bo02,DSS2007,LT98,Shen2013,SSBZ2017}. Comparison between existing models can be found in a recent work \cite{ten2023}, in which a unified framework was proposed for Navier--Stokes/Cahn--Hilliard models with unmatched densities. Unlike the quasi-incompressible models \cite{LT98,Shen2013,Aki14,SSBZ2017} -- which employ a mass-averaged velocity that is not divergence-free -- the model introduced in \cite{AGG2012} adopts the volume-averaged velocity as in \cite{Bo02,DSS2007}, thereby satisfying the incompressibility condition. In particular, this model is thermodynamically consistent as it satisfies an energy-dissipation law with a modified kinetic energy in terms of the volume-averaged velocity (cf. \cite{Bo02,DSS2007}). Furthermore, by treating the fluid mixture as a single fluid governed by the linear momentum conservation law (neglecting the momentum due to relative motions of the fluids \cite{GPV96}), the hydrodynamic system formulated with respect to the volume-averaged velocity remains frame-invariant (see \cite[Remark 2.2]{AGG2012}). For recent progress on the mathematical analysis of the model proposed in \cite{AGG2012} and its variants, we refer to \cite{ADG2013,ADG2013-2,AGG2023,AGP2024,Frigeri2016,Frigeri2021,GalGW2019, Gior2021,G2021,GK23} and the references therein.

The interaction between (fluid) mixtures and chemical species has drawn considerable attention in recent research, particularly in the context of tumor growth modeling \cite{CL10,GLSS,Hawk2011,Oden10}. Transport mechanisms such as \textit{chemotaxis} and \textit{mass transport} play an important role in associated complex dynamics. To see this, we first focus on the interplay between $\varphi$ and $\sigma$, neglecting their coupling to the macroscopic fluid flow. In \cite{GLSS}, the authors considered the free energy given by
$$
\mathcal{E}_{\text{free}}(\varphi, \sigma)= \int_{\Omega}\Big( \frac{\varepsilon}{2}|\nabla \varphi|^{2}
+ \frac{1}{\varepsilon}\Psi(\varphi)
+ \frac{1}{2}|\sigma|^2 + \chi (1-\varphi)\sigma \Big)\, \mathrm{d} x,
$$
which accounts for the phase separation process of the binary mixture (tumor and healthy tissues), the diffusion of a chemical species (e.g., a nutrient or a drug) and the interactions between them. The constant $\chi$ can be regarded as a parameter for a certain transport mechanism, e.g., the magnitude of chemotaxis sensitivities. Based on the mass balance law and the variational principle, they derived the following system (stated in the simplest form without fluid coupling):
\begin{equation}
\begin{cases}
\partial_t \varphi = \div(m(\varphi)\nabla \mu) +S(\varphi,\sigma),\\
\mu= -\varepsilon \Delta \varphi+ \dfrac{1}{\varepsilon} \Psi'(\varphi) -\chi \sigma,\\
\partial_t \sigma =\div(\nabla \sigma -\chi \nabla \varphi) +R(\varphi, \sigma),
\end{cases}
\label{CH-CH}
\end{equation}
where $S=S(\varphi, \sigma)$, $R=R(\varphi, \sigma)$ denote certain mass source terms.
In \eqref{CH-CH}, the mass fluxes are given by
\begin{align*}
&\bm{q}_\varphi:= -m(\varphi)\nabla \mu= -m(\varphi)\nabla( - \varepsilon\Delta \varphi + \varepsilon^{-1}\Psi'(\varphi)-\chi \sigma),
\quad
\bm{q}_\sigma:=-\nabla(\sigma -\chi \varphi).
\end{align*}
The term $\chi m(\varphi)\nabla \sigma$ in $\bm{q}_\varphi$ represents the chemotactic response of the binary mixture to the chemical species (i.e., chemotaxis), while the term $\chi \nabla \varphi$ in $\bm{q}_\sigma$ drives the chemical species via the concentration gradient of the mixture (i.e., mass transport). The latter is often referred to as ``active transport'' in the biological sense when $\chi\geq 0$, see \cite{GLSS} for detailed explanations. Analysis on the system \eqref{CH-CH} subject to different types of boundary conditions can be found in \cite{GLam17a,GLam17b}. When fluid effects are taken into account, we refer the reader to \cite{CGSS23,EG19jde,GLSS,FLRS18,KS22} for the coupling with Darcy's/Brinkman's equations, and to \cite{H,HW,HW24,LW2018} for the coupling with the Navier--Stokes system.

The reaction-diffusion equation for $\sigma$ in \eqref{CH-CH} has a cross-diffusion structure that depends linearly on $\varphi$. However, since the term $\chi\Delta \varphi$ has no sign property, the solution $\sigma$ may not satisfy a minimum principle. Even if the initial datum $\sigma_0$ is nonnegative, one cannot guarantee that $\sigma$ is nonnegative for $t>0$, in conflict with the physical interpretation of $\sigma$ as a concentration. This limitation motivated the authors of \cite{RSS2023} to propose an alternative evolution equation for $\sigma$:
\begin{align}
\partial _t\sigma = \div(\nabla  \sigma - \chi   \sigma\nabla \varphi) + R(\varphi,\sigma),\label{sigma-new}
\end{align}
which ensures the non-negativity of $\sigma$ with a properly designed source term $R$. From the perspective of the free energy, it corresponds to replacing the quadratic energy density $\frac12|\sigma|^2$ in $\mathcal{E}_{\text{free}}$ with the logarithmic entropy $\sigma\ln \sigma -\sigma$ (exactly as in \eqref{total-Energy-0}). This modification more naturally reflects the mass transfer
process of chemical species driven by the gradient of mixture distribution $\nabla \varphi$, since the mass flux $\chi\sigma\nabla \varphi$ is now proportional to the concentration of chemical species $\sigma$. Together with the Cahn--Hilliard equation for $\varphi$, we arrive at a new coupled system
\begin{equation}
\left\{ \begin{array}{l}
\partial_t \varphi = \div(m(\varphi)\nabla \mu) +S(\varphi,\sigma),
\\
\mu= -\varepsilon \Delta \varphi+ \dfrac{1}{\varepsilon} \Psi'(\varphi) -\chi \sigma,
\\
\partial _t\sigma = \div(\nabla  \sigma - \chi   \sigma\nabla \varphi) + R(\varphi,\sigma).
 \end{array} \right.
 \label{CH-KS}
\end{equation}
The nonlinear cross-diffusion structure in \eqref{sigma-new} is closely related to the well-known Keller--Segel system for chemotaxis (see, e.g., \cite{BBTW,Hor2004,KS}):
\begin{equation}
\left\{ \begin{array}{l}
\partial_t u=\mathrm{div}(\gamma (v)\nabla u-u\chi (v)\nabla v),\\
\tau \partial_t v=\Delta v + u -v.
\end{array} \right.
\label{KS-a}
\end{equation}
Replacing the second-order reaction-diffusion equation for $v$ in \eqref{KS-a} with the fourth-order Cahn--Hilliard equation for $\varphi$ in \eqref{CH-KS} leads to new mathematical challenges in the analysis. In \cite{RSS2023}, the authors established the existence of global weak solutions to \eqref{CH-KS} in two and three dimensions, assuming that $R(\varphi,\sigma)$ has a logistic growth with respect to $\sigma$. Further regularity properties were obtained under stronger assumptions on the structural data. In addition, the existence of the global attractor was recently achieved in \cite{SS2025}.
In this framework, analogously to the Keller--Segel system, the logistic term can penalize large values of the concentration, facilitating the derivation of global-in-time estimates.
Extensions to multi-species tumor growth models with chemotaxis and angiogenesis are presented in \cite{AS24,CGSS25}, which, however, adopt a different coupling structure, see \cite[Remark 1.1]{AS24}.
In terms of the fluid interaction, we refer to \cite{Sch24} for the analysis of a coupled system involving a Brinkman-type equation for the macroscopic velocity field. There, the author replaced $\sigma$ in the mass flux $\chi\sigma\nabla \varphi$ with a ``degenerate'' chemotactic sensitivity function like $\frac{\sigma}{1+\sigma^{q-1}}$ with $q\in (1,2]$, that is, with a slower growth for large values of $\sigma$ (cf. \cite{HorW2005,WWX18} for the Keller--Segel system). This modification served as an alternative regularization, thereby enabling a proof for the existence of global weak solutions in both two and three dimensions without invoking logistic degradation.
In the recent work \cite{GHW1}, the authors considered the coupling with a Navier--Stokes system for the fluid velocity field. Assuming matched densities for the binary fluids and the presence of logistic degradation, they established the existence of global weak solutions in two and three dimensions under general structural assumptions. Moreover, under stronger conditions on the coefficients and data, they proved the regularity and uniqueness of weak solutions in two dimensions. It is worth mentioning that in all the aforementioned works, global weak solutions are obtained under certain regularization of the system, either through logistic degradation or a degenerate chemotactic sensitivity function.

\subsection{Summary of results and proof strategies}
The main goal of this study is to analyze the initial boundary value problem \eqref{NSCHc}--\eqref{NSCHc-bic} in a general setting, that is, with unmatched densities, non-constant viscosity $\nu$ and mobility $m$, nonlinear interaction function $\beta$, but without any regularization due to logistic degradation or degenerate chemotactic sensitivity. Here, we focus on the case where the spatial dimension is two. The analysis in three dimensions is more involved and will be conducted in a forthcoming study.

We now present a summary of our main results:
\begin{itemize}
\item[(1)] When the initial energy $\mathcal{E}(\bm{v}_0,\varphi_0,\sigma_0)$ is finite (cf. \eqref{total-Energy-0}), problem \eqref{NSCHc}--\eqref{NSCHc-bic} admits a global finite energy solution $(\bm{v},\varphi, \mu, \sigma)$ that satisfies an energy inequality (cf. \eqref{energy-id-0}).
If in addition, $\sigma_0\in L^2(\Omega)$, problem \eqref{NSCHc}--\eqref{NSCHc-bic} admits a global weak solution. For the precise notion of solutions and the rigorous statements, we refer to Definition \ref{def-weak-1} and Theorem \ref{WEAK-SOL}.
\item[(2)] Under suitable stronger assumptions on the structural data and the initial data, problem \eqref{NSCHc}--\eqref{NSCHc-bic} admits a unique global strong  solution $(\bm{v},\varphi, \mu, \sigma)$, see Definition \ref{def-strong} and Theorem \ref{Reg-SOL}.
\item[(3)] Under the same structural assumptions as in (2), every global weak solution becomes a global strong solution for $t > 0$ (see Theorem \ref{regw}).
\item[(4)] For the global strong solution, if, in addition, $\sigma_0\in L^\infty(\Omega)$, then $\sigma$ is globally bounded (see Corollary \ref{non-blowup}).
\end{itemize}

In the following, we give some comments on the specific features of the problem, the novelties of the present work, and the proof strategies.

The energy dissipation relation \eqref{energy-id-0} serves as a starting point for the analysis. In order to derive global-in-time \textit{a priori} estimates from \eqref{energy-id-0}, two main arguments are required.
First, it is crucial to establish the coercivity of the total energy $\mathcal{E}(\bm{v},\varphi, \sigma)$. The singular potential $\Psi$ guarantees that the phase-field variable $\varphi$ takes its values in $[-1,1]$. The boundedness of $\varphi$ together with the entropy function $\sigma\ln \sigma$ allows us to control the interaction term $\beta(\varphi)\sigma$ that does not have a definite sign. In particular, it allows us to handle a nonlinear interaction function $\beta$,
going beyond the linear cases, e.g., $\beta(\varphi)=\chi(1-\varphi)$
considered in \cite{AS24,GHW1,RSS2023,Sch24}. Indeed, since only the value of $\beta$ in $[-1,1]$ is important, $\beta$ can be extended to zero outside a slightly larger interval like $(-2, 2)$ without affecting the model. Furthermore, the physical bound $\varphi\in [-1,1]$ ensures that both the average density $\rho(\varphi)$ and the average viscosity $\nu(\varphi)$ remain bounded and strictly positive. Secondly, a key difficulty lies in deriving additional estimates for $\varphi$ and $\sigma$, beyond the boundedness of the energy $\mathcal{E}(\bm{v},\varphi, \sigma)$, from the dissipative terms $\int_{\Omega} m(\varphi)|\nabla \mu|^{2} +  \sigma|\nabla w|^2 \, \mathrm{d} x$. This is achieved through a careful analysis of the coupling between the $\varphi$- and $\sigma$-equations (see Steps $3$ and $4$ in Section \ref{EXIST-WEAK}).

To prove the existence of a global finite energy/weak solution defined in Definition \ref{def-weak-1}, it is crucial to design a suitable approximation scheme being compatible with \textit{a priori} energy estimates. To avoid possible truncations in the $\sigma$-equation as in \cite{RSS2023}, we work with approximate solutions for the chemical density $\sigma$ within the class of classical solutions. By the semigroup theory for second-order parabolic equations and the strong maximum principle, these approximate solutions remain sufficiently smooth and strictly positive in $\Omega\times (0,\infty)$. Due to the singular potential $\Psi$, it does not seem convenient to construct approximate classical solutions for $(\varphi, \mu)$. Hence, we choose to treat the Navier--Stokes/Cahn--Hilliard subsystem for $(\bm{v}, \varphi, \mu)$ by a Faedo--Galerkin scheme in the spirit of \cite{Frigeri2016}. For this purpose, we first approximate the singular potential $\Psi$. Owing to the ``truncated'' shape of $\beta$ mentioned above, a standard polynomial-type regularization (cf. \cite{GMT2019,GiGrWu2018,MT16}) suffices, avoiding the more complex one used in \cite{GHW1}. As a price for this simplification, we have to deal with nonlinear, rather than linear, interaction functions.
Since the fourth-order Cahn--Hilliard equation with regular potential lacks a maximum principle, the physical bound of $\varphi$ cannot be guaranteed, and the average density $\rho(\varphi)$ may be negative outside the physical domain $[-1,1]$. To overcome this difficulty, we use a nonlinear extension of the linear density function $\rho(\varphi)$ to the entire line $\mathbb{R}$, which is denoted by $\widehat{\rho}(\varphi)$ so that $\widehat{\rho}$ has positive upper and lower bounds. In order to maintain a dissipative energy identity for the regularized problem, a further nonlinear correction term $\widehat{R}\bm{v}/2$ to the Navier--Stokes system is necessary, see Section \ref{modif-sys} for further details. Finally, to gain sufficient estimates to pass to the limit in the finite dimensional Galerkin scheme, we further introduce a $p$-Laplacian regularization in the modified Navier--Stokes system. In summary, the regularized problem can be solved by using the semi-Galerkin scheme mentioned above combined with the Schauder fixed-point theorem. The existence of a global finite energy/weak solution to the original problem \eqref{NSCHc}--\eqref{NSCHc-bic} can be obtained by passing to the limit with suitable compactness arguments. Here, we distinguish two types of variational solutions by different requirements on their initial data. The additional assumption $\sigma_0\in L^2(\Omega)$ for the global weak solution yields enhanced regularity properties, which in turn ensure its instantaneous regularization for $t > 0$.

Due to the highly nonlinear coupling structure of the system \eqref{NSCHc}, direct construction of a global strong solution seems rather difficult. We adopt a different strategy by studying three decomposed subsystems for the unknowns $\bm{v}$, $(\varphi, \mu)$ and $\sigma$, respectively. These results have independent interests and extend the previous literature on single equations (see, e.g. \cite{CGGG,Gior2021}). Then, based on the existence of a global weak solution and a novel bootstrap argument, we can establish the existence of a global strong solution. Uniqueness is a consequence of the energy method, relying on the crucial fact that the phase-field variable is strictly separated from the pure phases $\pm 1$ for all $t\geq 0$.

We note that the fourth-order Cahn--Hilliard equation yields lower temporal regularity but better spatial regularity for the phase-field variable $\varphi$. This regularity property, together with the boundedness of $\varphi$, leads to a fundamentally different behavior of the cross-diffusion structure compared to the Keller--Segel system in two dimensions, which has the so-called critical mass phenomenon, see e.g., \cite{GZ98,HorW01,NSY97}. In particular, for every global strong solution to problem \eqref{NSCHc}--\eqref{NSCHc-bic} with a bounded initial datum $\sigma_0$, the chemical density $\sigma$ remains bounded for all time. This finding reveals that, in the two-dimensional case, the diffusion process driven by the chemical potential $\mu$ (i.e., the Cahn--Hilliard dynamics) suppresses the growth of $\sigma$, precluding any blow-up.


\section{Main Results}
\setcounter{equation}{0}
\label{pm}
\subsection{Preliminaries}
First, we introduce some notation that will be used
throughout this paper.

Let $X$ be a real Banach space. We denote its norm by $\|\cdot\|_X$, its dual space by $X'$, and the duality pairing by $\langle \cdot,\cdot\rangle_{X',X}$. The bold letter $\bm{X}$ denotes the generic space of vectors or matrices, with each component belonging to $X$. Given a measurable set $I\subset \mathbb{R}$, $L^q(I;X)$ with $q\in [1,\infty]$ denotes the space of Bochner measurable $q$-integrable functions (for $q\in [1,\infty)$) or essentially bounded functions (for $q=\infty$) with values in the Banach space $X$. If $I=(a, b)$, we simply write $L^q(a,b;X)$. In addition, $L^q_{\uloc}([a,\infty); X)$ denotes the uniformly local variant of $L^q(a,\infty;X)$ consisting of all strongly measurable functions $f\colon
[a,\infty)\to X$ such that
\begin{equation*}
  \|f\|_{L^q_{\uloc}([a,\infty); X)}= \sup_{t\geq a}\|f\|_{L^q(t,t+1;X)} <\infty.
\end{equation*}
For a finite interval $[a,b]$, we set $L^q_{\uloc}([a,b]; X) := L^q(a,b;X)$. Next, we denote by $C([a,b];X)$ (resp. $C_{\mathrm{w}}([a,b];X)$) the Banach space of functions that are strongly (resp. weakly) continuous from $[a,b]$ to $X$. Accordingly, $BC([a,\infty);X)$ (or $BC_{\mathrm{w}}([a,\infty);X)$) stands for the Banach space of all bounded and continuous (or weakly continuous) functions. For $q\in [1,\infty]$, $W^{1,q}(a,b;X)$ is the space of all $f\in L^q(a,b;X)$ with $\partial_t f \in L^q(a,b;X)$, where $\partial_t$ denotes the vector-valued distributional derivative of $f$. The uniformly local space $W^{1,q}_{\uloc}([a,\infty);X)$ is defined by replacing $L^q(a,b;X)$ with $L^q_{\uloc}([a,\infty);X)$. Finally, we set $H^1(a,b;X):= W^{1,2}(a,b;X)$ and $H^1_{\uloc}([a,\infty);X) := W^{1,2}_{\uloc}([a,\infty);X)$.

Throughout the paper, we assume that $\Omega \subset\mathbb{R}^2$ is a bounded domain with a sufficiently smooth boundary $\partial\Omega$. For any $q \in [1,\infty]$, we denote by $L^{q}(\Omega)$ the Lebesgue space. For any positive integer $s$ and $q\in [1,\infty]$, we denote by $W^{s,q}(\Omega)$ the Sobolev space of function in $L^q(\Omega)$ with distributional derivatives of order up to $s$ belonging to $L^q(\Omega)$. If $q=2$, we use the standard notation $H^{s}(\Omega)$ for the Hilbert space $W^{s,2}(\Omega )$. For simplicity, the norm and inner product in the space $L^{2}(\Omega)$ (as well as in $\bm{L}^{2}(\Omega)$) are denoted by $\|\cdot\|$ and $(\cdot,\cdot)$, respectively. In view of the homogeneous Neumann boundary condition, we introduce the space
$$
H^2_{N}(\Omega):= \big\{f\in H^2(\Omega)\ |\  \partial_{\bm{n}}f=0 \ \textrm{on}\  \partial \Omega\big\}.
$$

For every $f\in (H^1(\Omega))'$, we define its generalized mean over $\Omega$ by
$\overline{f}=|\Omega|^{-1} \langle f,1\rangle_{(H^1(\Omega))',H^1(\Omega)}$. If $f\in L^1(\Omega)$, then it holds $\overline{f}= |\Omega|^{-1} \int_\Omega f \,\mathrm{d}x$. For convenience, we define linear subspaces with zero mean
\begin{align*}
L^2_{0}(\Omega):=\big\{f\in L^2(\Omega)\ |\ \overline{f} =0\big\},
\quad
V_0:= H^1(\Omega)\cap L_0^2(\Omega),
\quad
V_0^{-1}:= \big\{ f \in (H^1(\Omega))'\ |\  \overline{f}=0 \big\}.
\end{align*}
Let $\mathcal{A}_N\in \mathcal{L}(H^1(\Omega),(H^1(\Omega))')$ be the realization of the minus Laplacian $-\Delta$ subject to the homogeneous Neumann boundary condition such that
\begin{equation}\nonumber
	\langle \mathcal{A}_N u,v\rangle_{(H^1(\Omega))',H^1(\Omega)} := \int_\Omega \nabla u\cdot \nabla v \, \mathrm{d}x,\qquad \forall\, u,v\in H^1(\Omega).
\end{equation}
The restriction of $\mathcal{A}_N$ from $V_0$ onto $V_0^{-1}$ is an isomorphism. In addition, $\mathcal{A}_N$ is positively defined on $V_0$ and self-adjoint. We denote the inverse map by $\mathcal{N} =\mathcal{A}_N^{-1}: V_0^{-1} \to V_0$. For any $f\in V_0^{-1}$, $u= \mathcal{N} f \in V_0$ is the unique weak solution of the Neumann problem
$$
\begin{cases}
	-\Delta u=f, \quad \text{in} \ \Omega,\\
	\partial_{\bm{n}} u=0, \quad \ \  \text{on}\ \partial \Omega.
\end{cases}
$$
For every $f\in V_0^{-1}$, we define $\|f\|_{V_0^{-1}}=\|\nabla \mathcal{N} f\|$.
It is well-known that $f \to \|f\|_{V_0^{-1}}$ and $
f \to\big(\|f-\overline{f}\|_{V_0^{-1}}^2+|\overline{f}|^2\big)^\frac12$
are norms on $V_0^{-1}$ and $(H^1(\Omega))'$,
respectively, which are equivalent to the standard ones (see e.g., \cite{MZ04}). Recalling the well-known Poincar\'{e}--Wirtinger inequality
\begin{equation}
	\notag
\|f-\overline{f}\|\leq C \|\nabla f\|,\qquad \forall\,f\in H^1(\Omega),
\end{equation}
where $C>0$ depends only on $\Omega$, we find that $f\to \|\nabla f\|$ and  $f\to \big(\|\nabla f\|^2+|\overline{f}|^2\big)^\frac12$  are equivalent norms on $V_0$ and $H^1(\Omega)$, respectively. Moreover, we report the following standard Hilbert interpolation inequality
\begin{align}
	&\|f\|  \leq \|f\|_{V_0^{-1}}^{\frac12} \| \nabla f\|^{\frac12},
	 \qquad  \forall\, f \in V_0.
\label{inter-1}
\end{align}

Let us consider the following elliptic problem with a variable coefficient
\begin{equation}\label{epgq}
\begin{cases}
-\div(m(a) \nabla u) = f \quad & \text{in}\ \Omega \\
m(a) \partial_\n u = 0 \quad & \text{on} \ \partial\Omega.
\end{cases}
\end{equation}
The following result has been established in \cite{CGGG}:
\begin{lemma}\label{lem-mo}
Assume that $m\in C([-1,1])$ and
$m_{*} \leq m(r)\leq m^*$ for all $r \in [-1,1]$ with $m_{*}<m^*$ being two given positive constants, $a: \Omega \to [-1,1]$ is a measurable function.
For every $f\in V_0^{-1}$, problem \eqref{epgq} admits a unique weak solution $u\in V_0$ satisfying
\begin{equation}\label{def G_q}
( m(a) \nabla u ,\nabla v ) = \l f, v \r_{(H^1(\Omega))',H^1(\Omega)},
\qquad \forall \, v \in  V_{0}.
\end{equation}
In addition, if $m\in C^1([-1,1])$ and $a\in H^2(\Omega)$, then we have
\begin{align}
\|u\|_{H^2(\Omega)}\leq C(\|\nabla a\|\|a\|_{H^2(\Omega)}\|\nabla u\| +\|f\|),\qquad
\forall\, f\in L^2_0(\Omega).
\label{es-Ga-H2}
\end{align}
\end{lemma}
Based on Lemma \ref{lem-mo}, we can define the solution operator $\mathcal{G}_a:  V^{-1}_{0} \to  V_{0}$ for problem \eqref{epgq} such that $u = \mathcal{G}_a f$. Then we have
\begin{align}
& \sqrt{m_*}\|\sqrt{m(a)}\nabla \mathcal{G}_a f\|\leq \|\nabla \mathcal{N}f\|\leq \sqrt{m^*}\|\sqrt{m(a)}\nabla \mathcal{G}_a f\|,\qquad \forall\, f\in V_0^{-1},
\label{equiv-norm-h-1}
\\
&
\l f, \mathcal{G}_a  g \r_{(H^1(\Omega))',H^1(\Omega)}
= \l g, \mathcal{G}_a f \r_{(H^1(\Omega))',H^1(\Omega)}, \qquad \forall \, f,g \in  V^{-1}_{0},
\label{AUTO}
\end{align}
and
\begin{equation}
\label{inter-2}
\| f\|\leq \sqrt{m^*}
\| \nabla \mathcal{G}_a f \|^\frac12 \| \nabla f \|^\frac12,
\qquad \forall \, f \in  V_{0}.
\end{equation}

Let us now introduce Hilbert spaces for solenoidal vector-valued functions. As in \cite{S}, we denote by $\bm{L}^2_{0,\sigma}(\Omega)$, $\bm{H}^1_{0,\sigma}(\Omega)$ the closure of $C_{0,\sigma}^{\infty}(\Omega;\mathbb{R}^2)=\big\{\bm{u}\in C_0^{\infty}(\Omega;\mathbb{R}^2)\,|\, \mathrm{div}\, \bm{u}=0\big\}$ in  $\bm{L}^2(\Omega)$ and $\bm{H}^1(\Omega)$,
respectively.\footnote{The subscript $\sigma$ is a conventional notation for spaces of divergence-free functions in the literature. It should not be related to the chemical concentration $\sigma$ considered in this study.} We also set $\bm{W}^{1,q}_{0,\sigma}(\Omega)=\bm{W}^{1,q}(\Omega)\cap \bm{H}^1_{0,\sigma}(\Omega)$ for any $q\in (2,\infty)$.
Without ambiguity, we use $(\cdot,\cdot)$ and $\|\cdot\|$ for the inner product and the norm in $\bm{L}^2_{0,\sigma}(\Omega)$. For any function $\bm{u} \in \bm{L}^2(\Omega)$, the Helmholtz--Weyl decomposition yields (see \cite[Chapter \uppercase\expandafter{\romannumeral3}]{G})
\be
\bm{u}=\bm{u}_{0}+\nabla z,\quad\text{with}\ \ \bm{u}_{0} \in \bm{L}^2_{0,\sigma}(\Omega),\ z \in H^1(\Omega).\nonumber
\ee
Define the Leray projection $\bm{P}:\bm{L}^2(\Omega)\to \bm{L}^2_{0,\sigma}(\Omega)$ such that $\bm{P}(\bm{u})=\bm{u}_{0}$.
We have $\|\bm{P}(\bm{u})\|\leq \|\bm{u}\|$ for all $\bm{u}\in \bm{L}^2(\Omega)$.
The space $\bm{H}^1_{0,\sigma}(\Omega)$ is equipped with the inner product $(\bm{u},\bm{v})_{\bm{H}^1_{0,\sigma}}:=(\nabla \bm{u},\nabla \bm{v})$ and the norm $\|\bm{u}\|_{\bm{H}^1_{0,\sigma}}=\|\nabla \bm{u}\|$.
Owing to Korn's inequality
$$
\|\nabla \bm{u}\|\leq \sqrt{2}\|D\bm{u}\|
\leq \sqrt{2}\|\nabla \bm{u}\|,
\qquad \forall\, \bm{u}\in \bm{H}^1_{0,\sigma}(\Omega),
$$
$\|D\cdot\|$ gives an equivalent norm for $\bm{H}^1_{0,\sigma}(\Omega)$. Next, we introduce the Stokes operator $\bm{S}=\bm{P}(-\Delta)$ with the domain $D(\bm{S}):= \bm{H}^1_{0,\sigma}(\Omega)\cap\bm{H}^2(\Omega)$. Then $D(\bm{S})$ is a Hilbert space equipped with the inner product $(\bm{S}\bm{u},\bm{S}\bm{v})$ and the norm $\|\bm{S}\bm{u}\|$  (see, e.g., \cite[Chapter III]{S}). For any $\bm{u}\in D(\bm{S})$ and $\bm{\zeta} \in \bm{H}^1_{0,\sigma}(\Omega)$, it holds
$(\bm{S}\bm{u},\bm{\zeta})=(\nabla \bm{u},\nabla\bm{\zeta})$.

Finally, we recall some inequalities that will be used in the subsequent analysis.
\begin{lemma}[Generalized Young's inequality]
\label{You}
Let
\be
f(a):=\mathrm{e}^a-a-1, \quad g(b):=(b+1) \ln (b+1)-b.
\notag
\ee
Then, it holds
$$
a b \le f(a)+g(b), \quad \forall\, a, b \ge 0.
$$
\end{lemma}

The next lemma gives an interpolation inequality involving logarithmic norms (see \cite[Lemma A.5]{TW2014}):
\begin{lemma}
\label{GN-ln}
Let $\Omega\subset \mathbb{R}^2$ be a bounded domain with smooth boundary, $q \in (1,\infty)$,  $r\in (1,q)$  and $\alpha>0$. Then, there exists $C>0$ such that for each $\eta>0$,
\begin{equation}
\| u\|_{L^q(\Omega)}^q \leq \eta \| \nabla u\|^{q-r} \big\| u \ln^\alpha |u|\big\|_{L^r(\Omega)}^r
+ C \| u\|_{L^r(\Omega)}^q + C_\eta, \quad \forall \, u \in H^1(\Omega),
\notag
\end{equation}
for some positive constant $C_\eta$ depending on $\eta$.
\end{lemma}
The following lemma is a consequence of the classical Trudinger--Moser inequality (see \cite[Lemma 2.2]{W2020}):
\begin{lemma}
\label{TM-var}
Let $\Omega\subset \mathbb{R}^2$ be a bounded domain with smooth boundary. There exists $M>0$ depending on $\Omega$ such that if $\sigma\in C(\overline{\Omega})$ is nonnegative, $\sigma \not\equiv 0$, and $\varphi\in H^1(\Omega)$, then for each $\eta >0$, it holds
\begin{align}
\int_\Omega|\varphi|\sigma\,\mathrm{d}x\leq \frac{1}{\eta}\int_\Omega \sigma\ln\left(\frac{\sigma}{\overline{\sigma}}\right)\,\d x
+ \eta \left(\int_\Omega\sigma\,\d x\right)\|\nabla \varphi\|^2
+ M\eta \left(\int_\Omega\sigma\,\d x\right) \|\varphi\|_{L^1(\Omega)}^2 + \frac{M}{\eta}\int_\Omega \sigma\,\d x,
\notag
\end{align}
where $\overline{\sigma}=|\Omega|^{-1}\int_\Omega \sigma\,\d x>0$.
\end{lemma}

In this work, we shall use the following notation
$$
\rho_\ast=\min \lbrace \widetilde{\rho}_1,\widetilde{\rho}_2\rbrace,
\quad \rho^\ast=\max \lbrace \widetilde{\rho}_1,\widetilde{\rho}_2\rbrace.
$$
For any given vectors $\bm{a}, \bm{b}\in \R^2$, the notation $\bm{a}\otimes\bm{b}$ denotes the matrix $(a_i b_j)_{i,j=1}^2$. The symbols $C$, $C_i$ stand for generic positive constants that may even change within the same line. Specific dependence of these constants in terms of the data will be pointed out if necessary.

\subsection{Statement of results}
Throughout this paper, we make the following hypotheses. \smallskip
\begin{itemize}
\item[(H1)] The singular potential $\Psi$ belongs to the class of functions $C\big([-1,1]\big)\cap C^{2}\big((-1,1)\big)$. It can be decomposed into the following form
\begin{equation}
\Psi(r)=\Psi_{0}(r)-\frac{\theta_{0}}{2}r^2,\nonumber
\end{equation}
such that
\begin{equation}
\lim_{r\to \pm 1} \Psi_{0}'(r)=\pm \infty ,\quad \text{and}\ \ \   \Psi_{0}''(r)\ge \theta,\quad \forall\, r\in (-1,1),\nonumber
\end{equation}
where $\theta_0\in \mathbb{R}$ and $\theta$ is a strictly positive constant. In addition, there exists certain $\epsilon_0\in(0,1)$ such that $\Psi_{0}''$ does not decrease in $[1-\epsilon_0,1)$ and does not increase in $(-1,-1+\epsilon_0]$. We make the extension $\Psi_{0}(r)=+\infty$ for any $r\notin[-1,1]$. Without loss of generality, we also set $\Psi_0(0)=\Psi_0'(0)=0$. \smallskip
\item[(H2)] The viscosity function $\nu\in C^{1}(\mathbb{R})$ is globally Lipschitz continuous in $\mathbb{R}$ and
\be
\nu_{*} \leq \nu(r)\leq \nu^*,\quad \forall\, r \in \mathbb{R},\nonumber
\ee
where $\nu_{*}<\nu^*$ are given positive constants. \smallskip
\item[(H3)] The mobility function $m\in C^{1}(\mathbb{R})$ is globally Lipschitz continuous in $\mathbb{R}$ and
\be
m_{*} \leq m(r)\leq m^*,\quad \forall\, r \in \mathbb{R},\nonumber
\ee
where $m_{*}<m^*$ are given positive constants. \smallskip
\item[(H4)] The function $\beta \in C^{2}(\mathbb{R})$ satisfies
\be
\begin{aligned}
&|\beta(r)| \leq \beta^*, &&\quad \forall\, r \in \mathbb{R},
\notag\\
&\beta(r)=0,&&\quad \forall\, r\in (-\infty,-2]\cup [2, +\infty),
\end{aligned}
\ee
where $\beta^*$ is a given positive constant. In particular, the derivatives $\beta'$ and $\beta''$ are bounded in $\mathbb{R}$.
\end{itemize}

\begin{remark}\rm
The physically relevant logarithmic potential \eqref{Log} satisfies the assumption (H1). In addition, one can easily extend the linear viscosity function \eqref{meanr-v} to $\mathbb{R}$ in such a way that it fulfills (H2), see \cite[Remark 2.1]{GMT2019}. Since the singular potential $\Psi$ ensures that $\varphi\in [-1,1]$, the value of $\nu$ outside of $[-1,1]$ is not important and can be chosen appropriately as in (H2). Similarly, the only significant values of $m$, $\beta$ are those for $r\in[-1,1]$. A simple example for $\beta$ is that $\beta(r)=\chi (1-r)$ or $\beta(r)=\chi r$ for $r\in[-1,1]$, with a suitable extension outside $[-1,1]$ as in (H4). This choice appears in the modeling of two-phase flows with chemotaxis and mass transport \cite{GLSS,LW2018}, where the coefficient $\chi\in \mathbb{R}$ characterizes the strength of the chemotactic effect.
\end{remark}

To state the main results, we first reformulate the system \eqref{NSCHc} in a suitable form. Since exact values of the parameter $\varepsilon$ related to the interfacial thickness do not play a role in the subsequent analysis, without loss of generality, we set
$$\varepsilon=1.$$
Besides, a direct calculation (for sufficiently regular solutions $(\vv, \varphi, \mu, \sigma)$ with $\sigma>0$) yields that
\begin{align*}
&w \nabla \sigma
=
\big( \ln \sigma + \beta(\varphi)\big) \nabla \sigma
= \nabla \big(\sigma \ln \sigma -\sigma +\beta(\varphi) \sigma \big)
-\beta'(\varphi) \sigma \nabla \varphi,\\
&\sigma\nabla w = \sigma\nabla \big( \ln \sigma + \beta(\varphi)\big)= \nabla \sigma +\beta'(\varphi)\sigma\nabla\varphi.
\end{align*}
Then, up to a reinterpretation of the pressure (still denoted by $P$ for simplicity), we write the original system \eqref{NSCHc} as
\begin{equation}
 \label{NSCH}
\begin{cases}
\partial_t ( \rho(\varphi)\vv) + \div \big( \vv \otimes ( \rho(\varphi) \vv + \J) \big) - \div \big( 2\nu(\varphi) D \vv \big) + \nabla P = \big(\mu-\beta'(\varphi)  \sigma\big) \nabla \varphi,\\
\div \, \vv=0,\\
\partial_t \varphi +\vv\cdot \nabla \varphi = \div\big(m(\varphi)\nabla \mu\big),\\
\mu= - \Delta \varphi+ \Psi'(\varphi) + \beta'(\varphi) \sigma,\\
\partial_t \sigma + \vv \cdot \nabla \sigma - \Delta \sigma -\div \left( \beta'(\varphi)\sigma \nabla \varphi\right)=0,
\end{cases}
\end{equation}
in $\Omega \times (0,\infty)$,
with
$$
\J= -\rho'(\varphi) m(\varphi) \nabla \mu,\quad \rho(\varphi)= \widetilde{\rho}_1 \dfrac{1-\varphi}{2}+ \widetilde{\rho}_2 \dfrac{1+\varphi}{2},\\
$$
subject to the boundary and initial conditions
\begin{equation}
\label{NSCH-bic}
\begin{cases}
\vv=\mathbf{0}, \quad \partial_\n \varphi= \partial_\n \mu =\partial_\n \sigma=0 \quad &\text{on }  \partial \Omega \times (0,\infty),\\
\vv|_{t=0}=\vv_0, \quad \varphi|_{t=0}=\varphi_0, \quad \sigma|_{t=0}=\sigma_0 \quad &\text{in } \Omega.
\end{cases}
\end{equation}
\begin{remark}\rm
\label{rem:NS-2}
From the Cahn--Hilliard equation for $\varphi$ and \eqref{meanr-v}, we find the continuity equation
\begin{align}
 \partial_t \rho(\varphi) + \div( \rho(\varphi) \vv+ \J )=0,
 \label{mass-bal}
\end{align}
where $\J$ is given by \eqref{Jrhonu}.
With the aid of \eqref{mass-bal}, we can rewrite the momentum equation in \eqref{NSCH} as
\begin{align}
\rho(\varphi)\partial_t \vv +  \big(( \rho(\varphi) \vv + \J) \cdot\nabla \big)\vv - \div \big( 2\nu(\varphi) D \vv \big) + \nabla P = \big(\mu-\beta'(\varphi)  \sigma\big) \nabla \varphi.
\label{mome-new}
\end{align}
\end{remark}
\begin{remark}\rm
Some boundary conditions in \eqref{NSCHc-bic} have been reformulated as well (see \eqref{NSCH-bic}). Since we shall focus on the case of a non-degenerate mobility $m$, the boundary condition $m(\varphi)\partial_\n \mu=0$ reduces to $\partial_\n \mu=0$ on $\partial \Omega\times (0,\infty)$. On the other hand, taking into account the condition $\partial_\n \varphi=0$, we replace $\sigma\partial_\n w=0$ by $\partial_\n \sigma =0$ on $\partial \Omega\times (0,\infty)$.
\end{remark}

In this work, we introduce two classes of ``variational solutions" for the
initial-boundary value problem \eqref{NSCH}--\eqref{NSCH-bic}.
\bd \label{def-weak-1}
(1) \textbf{Finite energy solution}. For any given initial data $(\bm{v}_0,\varphi_0,\sigma_0)$ satisfying
\begin{equation}
\label{init-data}
\begin{split}
&\vv_0 \in \L_\sigma^2(\Omega),\quad \varphi_0 \in H^1(\Omega) \text{ with } \| \varphi_0\|_{L^\infty(\Omega)}\leq 1 \text{ and }
\left|\overline{\varphi_0} \right|<1,
\\
& \sigma_0 \ln \sigma_0 \in L^1(\Omega) \text{ such that } \sigma_0\geq 0  \text{ a.e. in  } \Omega,
\end{split}
\end{equation}
a quadruple $(\bm{v},\varphi,\mu,\sigma)$ is called a \emph{global finite energy solution} to problem \eqref{NSCH}--\eqref{NSCH-bic} in $\Omega\times [0,\infty)$, if it satisfies the regularity properties
\begin{equation}
\label{Weak-reg-1}
\begin{aligned}
&\vv \in L^\infty( 0,\infty; \L_\sigma^2(\Omega))
\cap L^2(0,\infty; \H_{0,\sigma}^1(\Omega)),
\\
&\varphi \in BC_{\mathrm{w}}([0,\infty);H^1(\Omega))\cap L_{\uloc}^2([0,\infty); H_N^2(\Omega))\cap H^{1}_{\uloc}([0,\infty);(H^1(\Omega))'),
\\
&\varphi \in L^\infty(\Omega \times (0,\infty))\ \text{ with }\   |\varphi(x,t)| <1 \ \text{a.e. in } \Omega \times (0,\infty),
\\
&\Psi'(\varphi)\in L^2_{\uloc}([0,\infty); L^2(\Omega)),
\\
&\mu \in L^2_{\uloc}([0,\infty); H^1(\Omega)), \quad \nabla \mu \in L^2(0,\infty; \bm{L}^2(\Omega)),
\\
&\sigma \in BC_{\mathrm{w}}([0,\infty);L^1(\Omega))\cap  L^2_{\uloc}([0,\infty);L^2(\Omega))\cap W^{1,\frac{4}{3}}_{\uloc}([0,\infty); (H_N^2(\Omega))'),
\\
& \sigma^\frac12\in L^2_{\uloc}([0,\infty);H^1(\Omega)), \quad \sigma\in L^\frac{4}{3}_{\uloc}([0,\infty);W^{1,\frac43}(\Omega)),
\\
& \sigma(x,t)\geq 0 \ \ \text{a.e. in } \Omega \times (0,\infty),
\end{aligned}
\end{equation}
moreover, the following equalities hold
\begin{align}
&- \big(\rho(\varphi_0) \vv_0, \bm{\zeta}(0)\big)+ \int_0^\infty\big[-\big( \rho(\varphi) \vv, \partial_t \bm{\zeta} \big)
-\big( \vv \otimes (\rho(\varphi) \vv+\J), \nabla \bm{\zeta} \big)
+ \big( 2\nu(\varphi) D \vv, D \bm{\zeta} \big) \big]\, \d t
\notag\\
&\quad
= \int_0^\infty \big( (\mu-\beta'(\varphi) \sigma) \nabla \varphi,\bm{\zeta} \big) \, \d t
\label{uu-weak}
\end{align}
for any $\bm{\zeta} \in C_0^1([0,\infty);C_{0,\sigma}^\infty(\Omega;\mathbb{R}^2))$, with
\begin{align}
\J= -\frac{\widetilde{\rho}_2-\widetilde{\rho}_1}{2} m(\varphi) \nabla \mu \quad \text{a.e. in } \Omega \times (0,\infty),
\label{J-weak}
\end{align}
and
\begin{align}
\label{phi-weak}
&\langle \partial_t \varphi,\xi \rangle_{(H^1(\Omega))',H^1(\Omega)}
+ ({\bm{v} \cdot \nabla \varphi},\xi)+ (m(\varphi)\nabla \mu,\nabla \xi)
=0,
\end{align}
for any  $\xi \in H^1(\Omega)$, almost everywhere in  $(0,\infty)$,
with
\begin{equation}
\label{mu-def}
\mu= -\Delta \varphi + \Psi'(\varphi) + \beta'(\varphi) \sigma \quad \text{a.e. in } \Omega \times (0,\infty),
\end{equation}
and
\begin{align}
&
\l\partial_t \sigma,\omega\r_{(H^2_N(\Omega))',H^2_N(\Omega)}
-(\sigma\bm{v},\nabla \omega)
-(\sigma, \Delta \omega)
=-(\beta'(\varphi)\sigma\nabla \varphi,\nabla \omega),
  \label{sig-weak1}
\end{align}
for any  $ \omega \in H^2_N(\Omega)$, almost everywhere in  $(0,\infty)$. In addition, the initial conditions are fulfilled
\begin{align}
& \varphi|_{t=0}=\varphi_{0},\quad \sigma|_{t=0}=\sigma_{0},\quad\textrm{a.e. in  } \Omega,	
\label{ini-w1}
\end{align}
and the following energy inequality holds
\begin{align}
&  \mathcal{E}(t)
+ \int_0^t \mathcal{D}(\tau)\,\d \tau\leq  \mathcal{E}(0),
\label{menergy-weak-fi}
\end{align}
for almost all $t>0$, where
\begin{align}
\mathcal{E}(t)
&=\int_{\Omega}\left( \frac{1}{2}\rho(\varphi)|\boldsymbol{v}|^{2}
+\frac{1}{2}|\nabla \varphi|^{2}
+ \Psi(\varphi) +\sigma(\ln \sigma-1)
+ \beta(\varphi) \sigma\right)(t) \, \mathrm{d} x,
\label{total-Energy}
\\
\mathcal{D}(t)
& =\int_{\Omega} \left(2 \nu(\varphi)|D \boldsymbol{v}|^{2}
+m(\varphi)|\nabla \mu|^{2} + \big| 2\nabla \sqrt{\sigma} + \sqrt{\sigma} \nabla \beta(\varphi) \big|^2  \right)(t)\, \mathrm{d} x.
\label{Total-Diss-1}
\end{align}

(2) \textbf{Weak solution}. Assume in addition that $\sigma_0\in L^2(\Omega)$. A quadruple $(\bm{v},\varphi,\mu,\sigma)$ is called a \emph{global weak solution} to problem \eqref{NSCH}--\eqref{NSCH-bic} in $\Omega\times [0,\infty)$, if $(\vv, \varphi, \mu, \sigma)$ is a global finite energy solution and $(\vv, \varphi, \sigma)$ satisfies the following additional regularity properties
\begin{equation}
\label{Weak-reg-1b}
\begin{aligned}
&\vv\in BC_{\mathrm{w}}([0,\infty);\L_\sigma^2(\Omega)),\quad \partial_t\bm{P}(\rho(\varphi)\vv)\in L_{\uloc}^{s}([0,\infty);(D(\bm{S}))'),
\\
&\varphi \in  L_{\uloc}^4([0,\infty); H^2(\Omega))\cap   L_{\uloc}^2([0,\infty),W^{2,q}(\Omega)),
\\
&\Psi'(\varphi)\in L^2_{\uloc}([0,\infty); L^q(\Omega)),
\\
&\sigma \in BC([0,\infty);L^2(\Omega))\cap L_{\uloc}^2([0,\infty); H^1(\Omega))\cap H_{\uloc}^1([0,\infty); (H^1(\Omega))'),
\end{aligned}
\end{equation}
for any $s\in [1,2)$, $q\in [2,\infty)$, as well as the following equalities (cf. \eqref{uu-weak}, \eqref{sig-weak1})
\begin{align}
&\l\partial_t(\rho(\varphi) \vv), \bm{\zeta} \r_{(D(\bm{S}))',D(\bm{S})}
-(  \rho(\varphi) \vv\otimes \vv, \nabla \bm{\zeta} )
-( \vv \otimes \J, \nabla \bm{\zeta} )
+ \big( 2\nu(\varphi) D \vv, D \bm{\zeta} \big)
\notag\\
&\quad
=  \big( (\mu-\beta'(\varphi) \sigma) \nabla \varphi, \bm{\zeta} \big),
\quad \forall\,\bm{\zeta}\in D(\bm{S}),
\label{uu-weak2}
\\
&\l \partial_t \sigma, \xi\r_{(H^1(\Omega))',H^1(\Omega)}
-(\sigma\vv, \nabla \xi)+ (\nabla \sigma, \nabla\xi)
=- \left( \beta'(\varphi)\sigma \nabla \varphi, \nabla \xi\right),
\quad \forall\, \xi \in H^1(\Omega),
\label{sig-weak2}
\end{align}
almost everywhere in $(0,\infty)$. Moreover, the initial condition is fulfilled
\begin{align}
& \vv|_{t=0}=\vv_{0},\quad\textrm{a.e. in  } \Omega.	
\label{ini-w2}
\end{align}
\ed


\begin{remark}\rm
In Definition \ref{def-weak-1}, we distinguish two types of variational solutions according to different regularity properties of the chemical density $\sigma$. The name ``finite energy solution" naturally follows from the energy dissipation property \eqref{menergy-weak-fi} and the requirement on the initial data. Due to the improved regularity properties of the ``weak solution" and \eqref{uu-weak2}, we can verify that $\partial_t\bm{P}(\rho(\varphi)\vv)\in L_{\uloc}^\frac{2q}{q+2}([0,\infty);(\bm{W}^{1,q}_{0,\sigma}(\Omega))')$, for any $q\in (2,\infty)$, so that \eqref{uu-weak2} is also satisfied for every $\bm{\zeta}\in \bm{W}^{1,q}_{0,\sigma}(\Omega)$.
\end{remark}

Our first result reads as follows:
\begin{theorem}[Existence of global finite energy/weak solutions]
\label{WEAK-SOL}
Let $\Omega$ be a bounded domain in $\mathbb{R}^2$ with a boundary $\partial \Omega$ of class $C^4$. Suppose that (H1)--(H4) are satisfied.

(1) For any initial data $(\bm{v}_0,\varphi_0,\sigma_0)$ satisfying
\eqref{init-data}, problem \eqref{NSCH}--\eqref{NSCH-bic} admits a global finite energy solution $(\vv,\varphi,\mu,\sigma)$ in $\Omega\times [0,\infty)$ in the sense of Definition \ref{def-weak-1}-(1).

(2) Assume in addition that $\sigma_0\in L^2(\Omega)$. Problem \eqref{NSCH}--\eqref{NSCH-bic} admits a global weak solution $(\vv,\varphi,\mu,\sigma)$ in $\Omega\times [0,\infty)$ in the sense of Definition \ref{def-weak-1}-(2).
\end{theorem}


\begin{remark}\rm
The uniqueness of finite energy solutions is unattainable due to the low regularity of the chemical concentration $\sigma$. On the other hand, because of difficulties from the unmatched densities and the non-constant mobility $m(\varphi)$, the uniqueness of weak solutions remains open as well (cf. \cite{AGG2023,CGGG,H}).
\end{remark}

Due to the lack of uniqueness, it is unclear whether every finite energy solution will become a weak solution for positive time. Nevertheless, we can draw the following conclusion:
\begin{corollary}
\label{finite-to-weak}
Let the assumptions in Theorem \ref{WEAK-SOL}-(1) be satisfied. For any initial data $(\bm{v}_0,\varphi_0,\sigma_0)$ satisfying
\eqref{init-data}, problem \eqref{NSCH}--\eqref{NSCH-bic} admits at least one global finite energy solution $(\vv,\varphi,\mu,\sigma)$ in $\Omega\times [0,\infty)$ that becomes a global weak solution for $t>0$.
\end{corollary}

Next, let us introduce the definition of a strong solution.
\begin{definition}
\label{def-strong}
For any initial data satisfying
\begin{equation}
\label{init-data1}
\begin{split}
&\vv_0 \in \H_{0,\sigma}^1(\Omega),\ \  \varphi_0 \in H_N^2(\Omega) \text{ with } \| \varphi_0\|_{L^\infty(\Omega)}\leq 1,
\left|\overline{\varphi_0} \right|<1  \text{ and } -\Delta \varphi_0 +\Psi'(\varphi_0)\in H^1(\Omega),
\\
& \sigma_0 \in H^1(\Omega) \text{ such that } \sigma_0\geq 0  \text{ a.e. in  } \Omega,
\end{split}
\end{equation}
a quadruple $(\bm{v},\varphi,\mu,\sigma)$ is called a \emph{global strong solution} to problem \eqref{NSCH}--\eqref{NSCH-bic} in $\Omega\times [0,\infty)$, if it satisfies the additional regularity properties (cf. \eqref{Weak-reg-1}, \eqref{Weak-reg-1b})
\begin{equation}
\label{strong-reg-2}
\begin{split}
&\vv \in L^\infty(0,\infty; \H^1_{0,\sigma}(\Omega))
\cap L^2_{\mathrm{uloc}}([0,\infty); \H^2(\Omega))\cap H^1_{\mathrm{uloc}}([0,\infty); \L_\sigma^2(\Omega)),\\
&\varphi \in L^\infty(0,\infty;W^{2,q}(\Omega))\cap H^1_{\mathrm{uloc}}([0,\infty);H^1(\Omega)),
\\
&\mu \in L^\infty(0,\infty; H^1(\Omega))\cap L^2_{\mathrm{uloc}}([0,\infty);H^3(\Omega)),
\\
&\Psi'(\varphi)\in L^\infty(0,\infty; L^q(\Omega)),
\\
&\sigma \in L^\infty(0,\infty;H^1(\Omega)) \cap L^2_{\mathrm{uloc}}([0,\infty); H^2(\Omega))\cap H^1_{\mathrm{uloc}}([0,\infty);L^2(\Omega)),
\end{split}
\end{equation}
for any $q \in [2,\infty)$. The strong solution $(\bm{v},\varphi,\mu,\sigma)$ satisfies the system \eqref{NSCH} almost everywhere in $\Omega \times (0,\infty)$, as well as the boundary and initial conditions
\begin{align}
& \vv=\mathbf{0}, \quad \partial_\n \varphi=\partial_\n \mu=\partial_\n \sigma=0 \quad &&\text{a.e. on }  \partial \Omega \times (0,\infty),
\label{boundary-s}\\
& \vv|_{t=0}=\vv_0,\quad \varphi|_{t=0}=\varphi_{0},\quad \sigma|_{t=0}=\sigma_0,\quad &&\textrm{a.e. in  } \Omega.	
\label{ini-s}
\end{align}
\end{definition}

Our second result asserts the existence and uniqueness of a global strong solution. To this aim, we impose the following additional assumption on the singular potential (see \cite{GGG2023}, see also \cite{GMT2019,HW,MZ04} for the special case $\kappa=1$):
\begin{itemize}
\item[(H5)] The singular potential $\Psi_0$ satisfies
\begin{align}
	\Psi_0''(r)\leq Ce^{C|\Psi_0'(r)|^{\kappa}},
\quad\forall\,r\in(-1,1),
\label{addi}
\end{align}
for some constants $C>0$ and $\kappa\in[1,2)$.
\end{itemize}

\begin{theorem}[Existence and uniqueness of a global strong solution]
\label{Reg-SOL}
Let $\Omega$ be a bounded domain in $\mathbb{R}^2$ with a boundary $\partial \Omega$ of class $C^4$. Suppose that (H1)--(H5) are satisfied and, in addition,  $\beta\in C^3(\mathbb{R})$.

(1) For any initial data satisfying \eqref{init-data1},
problem \eqref{NSCH}--\eqref{NSCH-bic} admits a unique global strong solution $(\vv,\varphi,\mu,\sigma)$ in $\Omega\times [0,\infty)$ in the sense of Definition \ref{def-strong}. Moreover, there exists a constant $\delta_1\in (0,1)$ such that
\begin{equation}
  \|\varphi(t)\|_{L^{\infty}(\Omega)} \leq 1-\delta_1,\quad \forall\, t \geq 0,\label{sep1}
\end{equation}
where $\delta_1$ depends on $\|\nabla \vv_0\|$, $\|\varphi_0\|_{H^2(\Omega)}$, $\|\nabla (-\Delta \varphi_0+\Psi'(\varphi_0))\|$, $\|\sigma_0\|_{H^1(\Omega)}$, $\overline{\varphi_0}$, coefficients of the system, and $\Omega$, but it is independent of time.

(2) Assume in addition that $\sigma_0\in L^\infty(\Omega)$, then we have
\begin{align}
\|\sigma(t)\|_{L^\infty(\Omega)}\leq C, \quad \forall\, t\geq 0,
\label{sig-Linfty}
\end{align}
where the positive constant $C$ is independent of time.
\end{theorem}
\begin{remark}\rm
The growth condition \eqref{addi} ensures the strict separation property \eqref{sep1} for the phase-field $\varphi$. An alternative condition that can also lead to \eqref{sep1} but only involves the first derivative of $\Psi_0$ reads follows (see \cite{Gal2025}): for some $\kappa>1/2$, it holds
\begin{align}
\frac{1}{|\Psi_0'(1-2\delta)|}=O\Big(\frac{1}{|\ln \delta|^{\kappa}}\Big), \quad\frac{1}{|\Psi_0'(-1+2\delta)|}=O\Big(\frac{1}{|\ln \delta|^{\kappa}}\Big),\quad
\text{as}\ \ \delta\rightarrow0^{+}.
\notag
\end{align}
\end{remark}
\begin{remark}\rm
The $L^\infty$-bound \eqref{sig-Linfty} reveals that the so-called Cahn--Hilliard--Keller--Segel type coupling structure in the system \eqref{NSCH} is rather different from that in the classical Keller--Segel system. Within the class of strong solutions (subject to a bounded initial datum $\sigma_0$), the chemical concentration $\sigma$ never blows up in finite time and is instead uniformly bounded for all $t\geq 0$.
\end{remark}

Finally, we establish the propagation of regularity for global weak solutions.
\begin{theorem}[Instantaneous regularity of global weak solutions]
	\label{regw}
Suppose that $\Omega$ is a bounded domain in $\mathbb{R}^2$ with a boundary $\partial \Omega$ of class $C^4$, the assumptions (H1)--(H5) are satisfied and $\beta\in C^3(\mathbb{R})$.
Let $(\vv,\varphi,\mu,\sigma)$ be a global weak solution to problem \eqref{NSCH}--\eqref{NSCH-bic} given by Theorem \ref{WEAK-SOL}-(2).
Then, it becomes a global strong solution for $t>0$. Moreover, for any $\tau>0$, there exists a constant $\delta_2= \delta_2(\tau)\in (0,1)$, such that
\begin{equation}
	\|\varphi(t)\|_{L^{\infty}(\Omega)} \leq 1- \delta_2,\quad \forall\, t \geq \tau,
\label{sep1w}
\end{equation}
where $\delta_2$ depends on $\|\vv_0\|$, $\|\varphi_0\|_{H^1(\Omega)}$, $\|\sigma_0\|$, $\int_\Omega \Psi(\varphi_0) \, \d x$, $\overline{\varphi_0}$, coefficients of the system, $\Omega$, and $\tau$.
\end{theorem}

As a direct consequence of Corollary \ref{finite-to-weak} and Theorem \ref{regw}, we can conclude the following result on the finite energy solution:
\begin{corollary}
Suppose that $\Omega$ is a bounded domain in $\mathbb{R}^2$ with a boundary $\partial \Omega$ of class $C^4$, the assumptions (H1)--(H5) are satisfied and $\beta\in C^3(\mathbb{R})$. For any initial data $(\bm{v}_0,\varphi_0,\sigma_0)$ satisfying
\eqref{init-data}, problem \eqref{NSCH}--\eqref{NSCH-bic} admits at least one global finite energy solution $(\vv,\varphi,\mu,\sigma)$ in $\Omega\times [0,\infty)$ that becomes a global strong solution for $t>0$.
\end{corollary}

\section{A Regularized System and Its Semi-Galerkin Approximation}
\label{modif-sys}
\setcounter{equation}{0}

In this section, we introduce a regularized system for problem \eqref{NSCH}--\eqref{NSCH-bic}, with a suitable approximation for the singular potential $\Psi$ and a smooth bounded truncation of the linear density function $\rho$. The resulting regularized system will be solved by an appropriate semi-Galerkin scheme. Roughly speaking, we proceed as follows: first, given a finite dimensional Galerkin ansatz $(\uu^k,\psi^k)$ for the fluid velocity and the phase variable ($k\in \mathbb{Z}^+$), we solve the advection-diffusion-reaction equation for the chemical concentration $\sigma^k$; second, with the sufficiently smooth solution $\sigma^k$, we solve the Faedo--Galerkin approximation of the regularized Navier--Stokes/Cahn--Hilliard system and obtain a finite dimensional solution $(\bm{v}^k,\varphi^k)$; finally, we prove the existence of a fixed point for the nonlinear mapping $\mathbf{F}: \mathbf{F}(\bm{u}^k,\psi^k)= (\bm{v}^k, \varphi^k)$
by Schauder's fixed point theorem, which yields a local solution $(\bm{v}^k, \varphi^k, \sigma^k)$ of the semi-Galerkin scheme.

\subsection{The regularized system}
For any given parameter $\epsilon \in (0,1)$, we adopt the following standard approximation of the singular nonlinearity $\Psi_0$ as in, e.g., \cite{GiGrWu2018,GMT2019}:
\be
\Psi_ {0,\epsilon}(r)=
\begin{cases}
 \displaystyle{\sum_{j=0}^2 \frac{1}{j!}}
 \Psi_0^{(j)}(1-\epsilon) \left[r-(1-\epsilon)\right]^j,
 \qquad\qquad \   \forall\,r\geq 1-\epsilon,\\
 \Psi_0(r), \qquad \qquad \qquad \qquad \qquad \qquad \qquad \quad
 \forall\, r\in[-1+\epsilon, 1-\epsilon],\\
 \displaystyle{\sum_{j=0}^2
 \frac{1}{j!}} \Psi_0^{(j)}(-1+\epsilon)\left[ r-(-1+\epsilon)\right]^j,
 \qquad\ \ \ \forall\, r\leq -1+\epsilon.
 \end{cases}
 \label{vPsi}
\ee
Define
$$
\Psi_\epsilon(r)=\Psi_{0,\epsilon}(r)-\frac{\theta_0}{2}r^2.
$$
It follows that $\Psi_{0,\epsilon}, \Psi_{\epsilon}\in C^{2}(\mathbb{R})$. Let $\epsilon_0\in (0,1)$ be the constant given in (H1). There exists $\epsilon_1\in (0,\epsilon_0)$ such that for any $\epsilon\in (0,\epsilon_1]$,
$$
\Psi_0'(-1+\epsilon)\leq -1,\quad \Psi_0'(1-\epsilon)\geq 1,\quad
\Psi_0''(-1+\epsilon)\geq 4|\theta_0|+1,\quad \Psi_0''(1-\epsilon)\geq 4|\theta_0|+1.
$$
and
\begin{align*}
&\Psi_ {0,\epsilon}''(r)\ge \theta,\qquad \Psi_ {\epsilon}(r)\geq -L,\qquad \forall\, r\in \mathbb{R},
\end{align*}
where $L>0$ is a constant independent of $\epsilon$ and $r$.
Moreover, we have
$$
\Psi_{0,\epsilon}(r)\leq \Psi_{0}(r),\quad \forall\, r\in [-1,1] \quad \text{and}\quad
|\Psi_{0,\epsilon}'(r)|\leq |\Psi_{0}'(r)|,\quad \forall\, r\in (-1,1).
$$

Since the Cahn--Hilliard equation is a fourth-order parabolic equation for which the maximum principle does not hold in general, the solution $\varphi$ to the approximate problem with the regularized potential \eqref{vPsi} may not stay in $[-1,1]$ as time evolves. To overcome this difficulty, we follow \cite{Frigeri2016} and replace the linear density function $\rho$
by its $C^2$-extension $\widehat{\rho}:\mathbb{R}\rightarrow \mathbb{R}^{+}$, satisfying
\begin{align}
& \widehat{\rho }(r)=\rho (r),\quad \forall \,r\in [-1,1],\
\label{apprho1} \\
&  \frac12 \rho_{\ast}\leq \widehat{\rho }(r) \leq 2\rho^{\ast},\quad
| \widehat{\rho }^{(j) }(r)| \leq C_*,\quad j=1,2, \quad \forall\, r\in \mathbb{R},
  \label{apprho2}
\end{align}
for some constant $C_*>0$.
\begin{remark}\rm
The above nonlinear modification of $\rho$ leads to a correction to the mass balance equation \eqref{mass-bal}.
More precisely, if the linear density function \eqref{meanr-v} is replaced by  $\widehat{\rho}$, then we have
\begin{align}
 \partial_t \widehat{\rho}(\varphi) + \div( \widehat{\rho}(\varphi) \vv+ \widehat{\J} )= \widehat{R},
 \label{mass-ba2}
\end{align}
where the modified mass flux $\widehat{\J}$ and the reminder term $\widehat{R}$ are given by
\begin{align}
\widehat{\J}=-\widehat{\rho}'(\varphi)m(\varphi)\nabla\mu, \quad \widehat{R}=-m(\varphi)\nabla \widehat{\rho}'(\varphi)\cdot \nabla \mu.
\label{R}
\end{align}
\end{remark}

Next, we approximate the initial datum $(\varphi_0, \sigma_0)$.
For any given function $\varphi_{0}\in H^1(\Omega)$ with $\|  \varphi_{0} \|_{L^{\infty}(\Omega)} \le 1$ and $|\overline{\varphi_{0}}|<1$,
we define $\varphi_{0,n}$ as the unique solution to the Neumann problem
\begin{equation*}
\begin{cases}
 \varphi_{0,n} - \dfrac{1}{n} \Delta \varphi_{0,n}= \left(1-\dfrac{1}{n}\right) \varphi_{0}, \, \ \quad \text{in}\ \Omega,
 \\
  \partial_{\bm{n}}  \varphi_{0,n} =0,  \qquad  \qquad \qquad \qquad \quad \ \ \ \text{on}\ \partial\Omega,
\end{cases}
\end{equation*}
where $n\geq 2$ is an arbitrary integer. As in \cite{GHW1}, we have
$\varphi_{0,n}\in H^2_N(\Omega)\cap H^3(\Omega)$ with $\overline{\varphi_{0,n}} \in \big[-|\overline{\varphi_{0}}|,\,|\overline{\varphi_{0}}|\big]$,
moreover, it holds
\begin{align}
&\|  \varphi_{0,n} \|_{L^\infty(\Omega)} \le 1-\frac{1}{n},
\label{Lvp1k}
\\
&\|\varphi_{0,n}\| \leq \|\varphi_0\|,\quad  \|\nabla \varphi_{0,n}\| \leq \|\nabla \varphi_0\|,\quad
\|\Delta \varphi_{0,n}\|\leq 2n \|\varphi_0\|,
\label{Lvp2k}
\\
&\lim_{n\to \infty}\|\varphi_{0,n}-\varphi_{0}\|_{H^1(\Omega)} =0.
\label{Lvp3k}
\end{align}
For any $\sigma_0$ satisfying $\sigma_0\ln \sigma_0\in L^1(\Omega)$ and $\sigma_0\geq 0$ almost everywhere in $\Omega$, we consider a family of approximations $\{\sigma_{0,n}\}_{n\in \mathbb{Z}^+}$ with the following properties (see, e.g., \cite{ad2003}, \cite[Section 2.2]{W2016})
$$
\sigma_{0,n} \in C^{\infty}_0(\Omega),  \quad \sigma_{0,n}\ge 0 \ \text { in } \Omega,\quad \sigma_{0,n}\not\equiv 0,\quad
\sigma_{0,n} \rightarrow \sigma_0 \text { in }\ L \ln L(\Omega) \ \text { as } n\to \infty.
$$
Here, $L \ln L(\Omega)$ denotes the standard Orlicz space associated with the Young function $(0, \infty) \ni z \mapsto$ $z \ln (1+z)$. Without loss of generality, we assume that
$$  \int_\Omega \sigma_{0,n}\ln \sigma_{0,n}\,\mathrm{d}x \leq \int_\Omega \sigma_{0}\ln \sigma_{0}\,\mathrm{d}x+1,\quad
\forall\,n\in \mathbb{Z}^+.
$$
Thanks to Lemma \ref{You}, $\|\sigma_{0,n}\|_{L^1(\Omega)}$ is also uniformly bounded with respect to $n$.

For any given parameters
\begin{align}
\quad \gamma\in (0,1],
\quad \epsilon\in\left(0,\min\Big\{\epsilon_1, \frac12(1-|\overline{\varphi_0}|)\Big\}\right],
\quad n\in \mathbb{Z}^+\setminus\{1\},
\label{param-1}
\end{align}
let us consider the following regularized problem $(\bm{S}_{\gamma,\epsilon,n})$:
\begin{subequations}
\label{reg.sys}
\begin{alignat}{3}
&\partial_t\big(\widehat{\rho}(\varphi) \vv \big)
+ \div \big(\vv \otimes ( \widehat{\rho}(\varphi) \vv + \widehat{\J}) \big) - \div \big( 2\nu(\varphi) D \vv \big) - \gamma \div \big(|\nabla \vv|^2\nabla \vv\big) + \nabla P
\notag \\
&\quad
= \big(\mu-\beta'(\varphi)  \sigma\big) \nabla \varphi
+ \frac{\widehat{R}}{2}\vv,
\label{reg.1} \\
& \quad \text{with}\quad \widehat{\J}= -\widehat{\rho}'(\varphi) m(\varphi) \nabla \mu,\quad
\widehat{R}= -m(\varphi)\nabla \widehat{\rho}'(\varphi)\cdot \nabla \mu,
\label{reg.1a}\\
&\mathrm{div}\, \bm{v}=0,
\label{reg.1b}\\
&\partial_t\varphi+\bm{v}\cdot\nabla\varphi =\mathrm{div}\big(m(\varphi)\nabla \mu\big),
\label{reg.2} \\
&\mu= - \Delta \varphi + \Psi_{\epsilon}'(\varphi) + \beta'(\varphi)\sigma,
\label{reg.3}\\
&\partial_t\sigma+\bm{v}\cdot\nabla\sigma -\Delta \sigma -  \mathrm{div}
(\beta'(\varphi)\sigma\nabla \varphi)
=0,
\label{reg.4}
\end{alignat}
\end{subequations}
in $\Omega \times (0,\infty)$,
subject to the boundary conditions
\begin{alignat}{3}
&\bm{v}=\mathbf{0},\quad\partial_{\boldsymbol{n}} \varphi
= \partial_{\boldsymbol{n}} \mu
= \partial_{\boldsymbol{n}}\sigma =0,
\quad \textrm{on}\ \partial\Omega\times(0,\infty),
\label{reg.boundary}
\end{alignat}
and the initial conditions
\begin{alignat}{3}
&\bm{v}|_{t=0}=\bm{v}_{0},
\quad \varphi|_{t=0}=\varphi_{0,n},
\quad \sigma|_{t=0}=\sigma_{0,n},
\qquad &\textrm{in}&\ \Omega.
\label{reg.ini}
\end{alignat}
\begin{remark}\rm
The extra term $\frac{\widehat{R}}{2}\vv$ on the right-hand side of \eqref{reg.1} yields a correction to the momentum equation, which enables us to maintain the dissipative structure of the regularized system. For the construction of finite energy/weak solutions, we also require an additional regularization involving the $p$-Laplacian $- \gamma \div \big(|\nabla \vv|^2\nabla \vv\big)$ in \eqref{reg.1}.
\end{remark}

\subsection{The semi-Galerkin scheme}
In what follows, we solve the regularized problem \eqref{reg.sys}--\eqref{reg.ini} by means of a suitable semi-Galerkin scheme. Consider the family of eigenvalues $\{\lambda_i\}_{i=1}^{\infty}$ and the corresponding eigenfunctions $\{\bm{y}_{i}(x)\}_{i=1}^{\infty}$ of the Stokes problem
\be
(\nabla \bm{y}_{i},\nabla \bm{\zeta})=\lambda_{i}(\bm{y}_{i},\bm{\zeta}),\quad  \forall\, \bm{\zeta} \in {\bm{H}^1_{0,\sigma}(\Omega)},\quad \textrm{with}\ \|\bm{y}_{i}\|=1.
\notag
\ee
It is well known that $0<\lambda_1\leq \lambda_2 \leq \cdots \to +\infty$, $\{\bm{y}_{i}\}_{i=1}^{\infty}$ forms a complete orthonormal basis in $\bm{L}^2_{0,\sigma}(\Omega)$ and is orthogonal in $\bm{H}^1_{0,\sigma}(\Omega)$.
Next, we consider the family of eigenvalues $\{\ell_i\}_{i=1}^{\infty}$ and the corresponding eigenfunctions $\{z_{i}(x)\}_{i=1}^{\infty}$ of the Laplacian subject to a homogeneous Neumann boundary condition
\be
(\nabla z_{i},\nabla w)=\ell_{i}(z_{i},w),\quad  \forall\, w \in {H^1(\Omega)}, \quad \textrm{ with }  \|z_{i}\|=1.
\notag
\ee
Then $0=\ell_1<\ell_2 \leq \cdots \to +\infty$, $\{z_{i}\}_{i=1}^{\infty}$ with $z_1=1$ forms a complete orthonormal basis in $L^2(\Omega)$ and is orthogonal in $H^1(\Omega)$.
For every integer $k\geq 1$, we denote the finite-dimensional subspace of $\bm{L}^2_{0,\sigma}(\Omega)$ by
$$ \bm{Y}_{k}:=\textrm{span} \{\bm{y}_{1}(x) ,\cdots,\bm{y}_{k}(x)\}.$$
The orthogonal projection on $\bm{Y}_{k}$ with respect to the inner product in $\bm{L}^2_{0,\sigma}(\Omega)$ is denoted by $\bm{P}_{\bm{Y}_{k}}$. Similarly, we denote the finite-dimensional subspace of $L^2(\Omega)$ by
$$ Z_{k}:=\textrm{span} \{z_{1}(x) ,\cdots,z_{k}(x)\},$$
and the orthogonal projection on $Z_{k}$ with respect to the inner product in $L^2(\Omega)$ is denoted by $\bm{P}_{Z_{k}}$.
We note that $\bigcup_{k=1}^\infty \bm{Y}_{k}$ is dense in $\bm{L}^2_{0,\sigma}(\Omega)$, $\bm{H}^1_{0,\sigma}(\Omega)$ and $D(\bm{S})$, while $\bigcup_{k=1}^\infty Z_{k}$ is dense in $L^2(\Omega)$, $H^1(\Omega)$ and $H_N^2(\Omega)$.
Since $\Omega$ is a $C^4$-domain, we have $\bm{y}_{i}\in \bm{H}^1_{0,\sigma}(\Omega)\cap \bm{H}^4(\Omega)$ and $z_i\in H^2_N(\Omega)\cap H^4(\Omega)$ for all $i\in \mathbb{Z}^+$.
Moreover, for any fixed $k\in \mathbb{Z}^+$, the following inverse inequalities hold
\begin{align*}
&\|\bm{v}\|_{\bm{H}^j(\Omega)}\leq C_k\|\bm{v}\|,\quad \forall\, \bm{v}\in \bm{Y}_k,\quad j=1, 2, 3, 4,
\\
& \|\varphi\|_{H^j(\Omega)}\leq C_k\|\varphi\|,\quad \forall\,\varphi\in Z_k,\quad j=1, 2, 3, 4.
\end{align*}
The symbol $C_k$ denotes a generic positive constant that depends on the approximating parameter $k$.
\medskip

For the initial velocity field $\bm{v}_{0}\in \bm{L}^2_{0,\sigma}(\Omega)$, its finite dimensional approximation $\bm{P}_{\bm{Y}_{k}} \bm{v}_{0}$ satisfies
\begin{align*}
\|\bm{P}_{\bm{Y}_{k}} \bm{v}_{0}\|\leq \|\bm{v}_{0}\|
\quad \text{and} \quad
\lim_{k\to +\infty} \|\bm{P}_{\bm{Y}_{k}} \bm{v}_{0}-  \bm{v}_{0}\|=0.
\end{align*}
Concerning the finite dimensional approximation for the regularized initial datum $\varphi_{0,n}\in H^2_N(\Omega)$, we have  $\bm{P}_{Z_{k}}\varphi_{0,n}\in H^2_N(\Omega)$ and
$$
\lim_{k\to+\infty}\|\bm{P}_{Z_{k}}\varphi_{0,n} -\varphi_{0,n}\|_{H^2(\Omega)}=0,
$$
moreover, for every given $n\geq 2$, there exists an integer $\widehat{k}$ (sufficiently large) such that
\be
\|\bm{P}_{Z_{\widehat{k}}}\varphi_{0,n} -\varphi_{0,n}\|_{C(\overline{\Omega})}\le C\|\bm{P}_{Z_{\widehat{k}}}\varphi_{0,n}-\varphi_{0,n}\|_{H^2(\Omega)} \le \frac{1}{2n},
\notag
\ee
where the positive constant $C$ depends only on $\Omega$. Thus, for all integers $k\geq \widehat{k}$, we have
\be
\label{PZK}
\|\bm{P}_{Z_{k}}\varphi_{0,n}\|_{L^\infty(\Omega)}\le 1-\frac{1}{2n},
\quad   \|\bm{P}_{Z_{k}}\varphi_{0,n}\|_{H^1(\Omega)}
\leq   \|  \varphi_{0,n}\|_{H^1(\Omega)}\leq \|  \varphi_0\|_{H^1(\Omega)}.
\ee

Let $T>0$ be an arbitrarily fixed final time.
For every integer $k\geq \widehat{k}$, we consider the approximate solution $(\bm{v}^k,\varphi^k,\mu^k,\sigma^k)$ to the regularized system \eqref{reg.sys}--\eqref{reg.ini} such that the Galerkin ansatz
\be
\bm{v}^{k}(x,t):=\sum_{i=1}^{k}a_{i}^{k}(t)\bm{y}_{i}(x),
\quad \varphi^{k}(x,t):=\sum_{i=1}^{k}b_{i}^{k}(t)z_{i}(x),
\quad
\mu^{k}(x,t):=\sum_{i=1}^{k}c_{i}^{k}(t)z_{i}(x),
\notag
\ee
satisfy
\begin{align}
&\big(\partial_t  (\widehat{\rho}(\varphi^k) \bm{ v}^{k}),\bm{\zeta}\big) -\big(\vv^k \otimes ( \widehat{\rho}(\varphi^k) \vv^k + \widehat{\J}^k), \nabla \bm{ \zeta}\big)
+\big(  2\nu(\varphi^{k}) D\bm{v}^{k},D\bm{\zeta}\big)
+\gamma \big(|\nabla\vv^k|^2\nabla\vv^k,\nabla \bm{\zeta}\big)
\notag \\
&\quad =\big((\mu^{k} -\beta'(\varphi^k) \sigma^{k})\nabla \varphi^{k},\bm {\zeta}\big)+ \frac12 \big(\widehat{R}^k\vv^k,\bm{\zeta}\big)
\notag\\
&\qquad + \frac12\big(\widehat{\rho}'(\varphi^k)(\bm{I}-\bm{P}_{Z_{k}}) \big((\bm{v}^k\cdot\nabla\varphi^k)-\div \big(m(\varphi^k)\nabla \mu^k\big)\big)\vv^k,\bm{\zeta}\big),\quad \forall\,\bm{\zeta} \in \bm{Y}_{k},
\label{atest.1}\\
&(\partial_t\varphi^k,\xi)+(\bm{v}^k\cdot\nabla\varphi^k,\xi)
=-\big(m(\varphi^k)\nabla \mu^k,\nabla \xi\big),
\quad \forall\, \xi\in Z_{k},
\label{atest.2}\\
& (\mu^k,\xi)=\big(- \Delta \varphi^k + \Psi'_{\epsilon}(\varphi^k)+\beta'(\varphi^k)\sigma^k, \xi\big),
\quad \forall\, \xi\in Z_{k},
 \label{atest.3}
\end{align}
in $(0,T)$, where
\begin{align}
\label{atest.1b}
\widehat{\J}^k= -\widehat{\rho}'(\varphi^k) m(\varphi^k) \nabla \mu^k,\quad
\widehat{R}^k= -m(\varphi^k)\nabla \widehat{\rho}'(\varphi^k)\cdot \nabla \mu^k,
\end{align}
in $\Omega\times(0,T)$.
Moreover, the unknown function $\sigma^k$ satisfies
\begin{align}
&\partial_t\sigma^k+\bm{v}^k\cdot\nabla\sigma^k -\Delta\sigma^k-  \mathrm{div}\big(\beta'(\varphi^k)\sigma^k\nabla\varphi^k\big)
=0, \quad \text{in}\ \Omega\times(0,T).
\label{atest.4}
\end{align}
The solution $(\bm{v}^k,\varphi^k,\mu^k,\sigma^k)$
is subject to the following boundary and initial conditions
\begin{align}
 & \bm{v}^k=\bm{0},\quad \partial_{\bm{n}}\varphi^k=\partial_{\boldsymbol{n}} \mu^k=\partial_{\boldsymbol{n}} \sigma^k=0,&&\quad \textrm{on}\ \partial\Omega\times(0,T),
\label{boundary1}
\\
& \bm{v}^{k}|_{t=0}=\bm{P}_{\bm{Y}_{k}} \bm{v}_{0},\quad \varphi^{k}|_{t=0}=\bm{P}_{Z_{k}}\varphi_{0,n},\quad  \sigma^{k}|_{t=0}=\sigma_{0,n},&&\quad \text{in}\ \Omega.
\label{atest.ini0}
\end{align}

\begin{remark}\rm
Owing to \eqref{atest.2}, we find that
$$
 \partial_t\varphi^k + \bm{P}_{Z_{k}}(\bm{v}^k\cdot\nabla\varphi^k)
=\bm{P}_{Z_{k}}\div \big(m(\varphi^k)\nabla \mu^k\big).
$$
This yields a further correction of the mass balance equation due to the finite dimensional projection (cf. \eqref{mass-ba2}), that is,
\begin{align}
 \partial_t\widehat{\rho}(\varphi^k) + \div(\widehat{\rho}(\varphi^k)\vv^k+\widehat{\J}^k)
=\widehat{R}^k+\widehat{\rho}'(\varphi^k)(\bm{I}-\bm{P}_{Z_{k}})\big((\bm{v}^k\cdot\nabla\varphi^k)-\div \big(m(\varphi^k)\nabla \mu^k\big)\big),
\label{R-app}
\end{align}
where $\bm{I}$ denotes the identity operator. The last term on the right-hand side of Equation \eqref{atest.1} appears for a similar reason.
\end{remark}

The following proposition asserts the local well-posedness of the semi-Galerkin scheme \eqref{atest.1}--\eqref{atest.ini0}.
\bp[Local solvability of the semi-Galerkin scheme]\label{p1}
Suppose that assumptions (H1)--(H4) are satisfied, $T \in (0,+\infty)$, and the initial data $(\bm{v}_0,\varphi_0,\sigma_0)$ satisfy \eqref{init-data}. For every positive integer $k\geq \widehat{k}$, the semi-Galerkin scheme \eqref{atest.1}--\eqref{atest.ini0} admits a unique local solution $(\bm{v}^{k},\varphi^{k},\mu^{k},\sigma^{k})$ on a certain time interval $[0,T_{k}]\subset[0,T]$ satisfying
\begin{align}
&\bm{v}^{k} \in H^1(0,T_{k};\bm{Y}_k),
\quad \varphi^{k} \in H^1(0,T_{k};Z_{k}),
\quad \mu^k\in H^1(0,T_k;Z_{k}),
\notag \\
&\sigma^{k}\in C^{2,1}(\overline{\Omega}\times [0,T_k]),
\quad \sigma^k(x, t) > 0  \ \ \text{in } \ \overline{\Omega}\times (0,T_k].
\notag
\end{align}
The existence time $T_{k}\in (0,T]$ depends on the initial data, $\Omega$, $k$, $\gamma$, $\epsilon$, $n$  and coefficients of the system.
\ep

\begin{proof}
The proof of Proposition \ref{p1} consists of several steps. \medskip

\textbf{Step 1.} Let $\widetilde{M}$ be a sufficiently large constant that satisfies $\widetilde{M}\geq 2(\|\bm{v}_0\|^2+  \|\varphi_{0,n}\|^2  +1)$. The exact value of $\widetilde{M}$ will be determined later. Let $\uu^k$, $\psi^k$ be two given functions
\be
\bm{u}^{k}=\sum_{i=1}^{k}\widetilde{a}_{i}^{k}(t)\bm{y}_{i}(x)\in C^\delta([0,T];\bm{Y}_{k}),\quad \psi^{k}=\sum_{i=1}^{k}\widetilde{b}_{i}^{k}(t)z_{i}(x)\in C^\delta([0,T];Z_{k})
\nonumber
\ee
for some $\delta\in ( 0,1/2)$, which fulfill
\begin{align}
&\widetilde{a}_{i}^{k}(0)=(\bm{v}_{0},\bm{y}_{i}),\quad i=1,\cdots,k,\quad
\text{and}\quad \sup_{t\in [0,T]}  \sum_{i=1}^{k}|\widetilde{a}_{i}^{k}(t)|^{2}\le \widetilde{M},
\nonumber\\
&\widetilde{b}_{i}^{k}(0)=(\varphi_{0,n},z_{i}),\quad i=1,\cdots,k,\quad
\text{and}\quad \sup_{t\in [0,T]}  \sum_{i=1}^{k}|\widetilde{b}_{i}^{k}(t)|^{2}\le \widetilde{M}.
\nonumber
\end{align}
Then $\bm{u}^{k}|_{t=0}=\bm{P}_{\bm{Y}_{k}} \bm{v}_{0}$, $\psi^{k}|_{t=0}=\bm{P}_{Z_{k}}\varphi_{0,n}$ and
\begin{align}
 \sup_{t\in [0,T]}\|\bm{u}^{k}(t)\|^2\leq \widetilde{M},
 \quad
 \sup_{t\in [0,T]}\|\psi^{k}(t)\|^2\leq \widetilde{M}.
 \label{aLL-up}
\end{align}
We first consider the following auxiliary equation for the chemical concentration
\begin{alignat}{3}
&\partial_t\sigma^k+\bm{u}^k\cdot\nabla\sigma^k -\Delta\sigma^k- \mathrm{div}\big(\beta'(\psi^k)\sigma^k\nabla \psi^k\big)=0,
\label{1atest.4}
\end{alignat}
in $\Omega\times (0,T)$,
subject to the boundary and initial conditions
\begin{alignat}{3}
&\partial_{\boldsymbol{n}} \sigma^k =0  \quad \text{on}\ \partial\Omega\times(0,T),
\qquad \sigma^{k}|_{t=0}=\sigma_{0,n} \quad \text{in}\ \Omega.
\label{boundary2}
\end{alignat}
Define the space
\begin{align}
\widehat{X}&=L^{\infty}(0,T;H^1(\Omega)) \cap  L^{2}(0,T;H^2_N(\Omega))\cap H^1(0,T;L^2(\Omega)).
\notag
\end{align}
Then we have
\bl\label{fp}
Given $(\bm{u}^k, \psi^k)$ and $\sigma_{0,n}$ as above, problem \eqref{1atest.4}--\eqref{boundary2} admits a unique classical solution $\sigma^{k}$ in $\overline{\Omega}\times [0,T]$ such that
$$
\sigma^{k}\in C^{2,1}( \overline{\Omega} \times [0,T]),
\quad \sigma^k(x, t) > 0  \ \ \text{in } \ \overline{\Omega}\times (0,T].
$$
Moreover, $\sigma^k$ is bounded in $\widehat{X}$ and the following estimate holds
\begin{align}
&\|\sigma^{k}(t)\|^{2}
 \leq \|\sigma_{0,n}\|^2 \, \mathrm{e}^{Ct(M_1+1)},
 \quad \forall\, t\in [0,T],
\notag
\end{align}
where $M_1=\sup_{t\in[0,T]}\|\psi^k(t)\|^2$, the positive constant $C$ depends on $k$, but it is independent of $M_1$.
\el

Since $\Omega$ belongs to $C^4$, the given functions $(\bm{u}^k,\psi^k)$ have sufficient spatial regularity. The proof of Lemma \ref{fp} then follows an argument analogous to that of \cite[Lemma 3.1]{GHW1}, with minor modifications according to assumption (H4). For brevity, we omit the details here.

Thanks to Lemma \ref{fp}, we can define the following mapping
\begin{align}
\mathbf{F}^k_{1}:\  C^\delta([0,T];\bm{Y}_{k})\times C^\delta([0,T];Z_{k}) \ &\to\ \  \widehat{X}, \notag\\
(\bm{u}^{k},\psi^{k})\ &\mapsto\ \ \sigma^{k},\notag
\end{align}
which is bounded from $C([0,T];\bm{Y}_{k})\times C([0,T];Z_{k})$ to $\widehat{X}$.

Next, we show that $\mathbf{F}^k_{1}$ is continuous with respect to the given data $(\bm{u}^{k},\psi^{k})$ in the topology of $X$, where
\begin{align}
X&= L^{\infty}(0,T;L^2(\Omega)) \cap  L^{2}(0,T;H^1(\Omega)).
\notag
\end{align}
Let $\bm{u}^{k}_{1}$, $\bm{u}^{k}_{2}$ be two given vectorial functions with the same initial value $\bm{P}_{\bm{Y}_{k}} \bm{v}_{0}$, while $\psi^{k}_{1}$, $\psi^{k}_{2}$ be two given scalar functions with the same initial value $\bm{P}_{Z_{k}}\varphi_{0,n}$. Both $(\bm{u}^{k}_{i},\psi^{k}_{i})$, $i=1,2$, satisfy the condition \eqref{aLL-up}.
Let $\sigma^{k}_{i} =\mathbf{F}^k_{1}(\bm{u}^{k}_{i},\psi^{k}_i)$, $i=1,2$, be the two corresponding solutions to problem \eqref{1atest.4}--\eqref{boundary2} determined by Lemma \ref{fp} (with the same initial datum $\sigma_{0,n}$).
We denote their differences by
$$
\bm{U}^k= \bm{u}^k_1-\bm{u}^k_2,\quad \Theta^{k}=\psi_{1}^{k}-\psi_{2}^{k},\quad \Sigma^{k}=\sigma_{1}^{k}-\sigma_{2}^{k},
$$
which fulfill
\begin{align}
	&\partial_t  \Sigma^{k} +{\bm{u}^k_1 \cdot \nabla \Sigma^k} -\Delta \Sigma^k
	  =  -\bm{U}^k \cdot \nabla \sigma^k_2
        + \mathrm{div}\big(\Sigma^k \nabla \beta(\psi_{1}^k)+\sigma_{2}^k \nabla (\beta(\psi_1^k)-\beta(\psi_2^k))\big),
\label{2atest.3}
\end{align}
in $\Omega\times (0,T)$, subject to the boundary and initial conditions
\begin{alignat}{3}
&\partial_{\boldsymbol{n}} \Sigma^k =0  \quad \text{on}\ \partial\Omega\times(0,T),
\qquad \Sigma^{k}|_{t=0}=0 \quad \text{in}\ \Omega.
\label{boundary3}
\end{alignat}
Testing \eqref{2atest.3} by $\Sigma^k $, integrating over $\Omega$, using the fact $\mathrm{div}\,\bm{u}^k_1=\mathrm{div}\,\bm{u}^k_2=0$ and integration by parts, we get
\begin{align}
&\frac{1}{2}\frac{\mathrm{d}}{\mathrm{d}t} \|\Sigma^k\|^2+\|\nabla\Sigma^k\|^2
\notag \\
&\quad =
(\bm{U}^{k} \sigma^{k}_2, \nabla\Sigma^k)
- \big( \Sigma^k \nabla \beta(\psi_{1}^k), \nabla \Sigma^k\big)
- \big( \sigma_{2}^k  \nabla (\beta(\psi_1^k)-\beta(\psi_2^k)),  \nabla \Sigma^k \big)
\notag \\
&\quad =: \sum_{j=1}^3 I_j.
\label{diffsig}
\end{align}
Applying H\"{o}lder's inequality, Young's inequality, and the Sobolev embedding theorem, we can estimate the right-hand side of \eqref{diffsig} as follows
\begin{align}
|I_1|
&\le C\|\bm{U}^{k}\|_{\bm{L}^\infty(\Omega)}\|\sigma^{k}_{2}\|\|\nabla \Sigma^{k}\|\notag\\
&\le C \|\bm{U}^{k}\|_{\bm{H}^2(\Omega)}\|\sigma^{k}_{2}\|\|\nabla\Sigma^k\|\notag\\
& \le  C_k\|\sigma^{k}_{2}\|^2\|\bm{U}^{k}\|^2 +\frac16 \|\nabla\Sigma^k\|^2,
\notag
\end{align}
\be
\begin{aligned}
|I_2|
&\leq \|\Sigma^k\|\|\beta'(\psi_1^k)\|_{L^\infty(\Omega)}   \|\nabla\psi_1^k\|_{\bm{L}^{\infty}(\Omega)} \|\nabla\Sigma^k\|
\\
&\leq C \|\psi_1^k\|_{H^3(\Omega)}^2\|\Sigma^k\|^2
+\frac16\|\nabla\Sigma^k\|^2
\\
&\leq C_k  \|\psi_1^k\|^2\|\Sigma^k\|^2 +\frac16\|\nabla\Sigma^k\|^2,
\end{aligned}
\notag
\ee
and
\be
\begin{aligned}
|I_3| &\leq \|\sigma^k_2\|\left\|\nabla \int_0^1 \beta'(s\psi_1^k+(1-s)\psi_2^k)\Theta^k\,\mathrm{d}s\right\|_{\bm{L}^\infty(\Omega)} \|\nabla\Sigma^k\|
\notag \\
&\leq C \|\sigma^k_2\| \big(\|\nabla\psi_1^k  \|_{\bm{L}^\infty(\Omega)} +\|\nabla\psi_2^k  \|_{\bm{L}^\infty(\Omega)}\big)\|\Theta^k\|_{L^\infty(\Omega)}\|\nabla\Sigma^k\|
\notag \\
&\quad + C  \|\sigma^k_2\|  \|\nabla \Theta^k\|_{\bm{L}^\infty(\Omega)} \|\nabla\Sigma^k\|
\notag \\
&\leq C  \|\sigma^k_2\|^2 \big(1+\|\psi_1^k\|^2_{H^3(\Omega)} +\|\psi_2^k\|^2_{H^3(\Omega)} \big) \| \Theta^k\|_{H^3(\Omega)}^2
+\frac16\|\nabla\Sigma^k\|^2
\\
&\leq C_k \|\sigma^k_2\|^2 \big(1+\|\psi_1^k\|^2 +\|\psi_2^k\|^2 \big) \|\Theta^k\|^2 +\frac16\|\nabla\Sigma^k\|^2.
\end{aligned}
\notag
\ee
Here, we have used the assumption (H4) and the facts that $\bm{u}_i^k$ and $\psi_i^k$, $i=1,2$, are finite dimensional.
Combining the above estimates and integrating \eqref{diffsig} on $[0, t] \subset [0, T] $, we obtain
\be
\begin{aligned}
&  \|\Sigma^k(t) \|^{2}
+\int_0^t  \|\nabla\Sigma^k(s)\|^{2}\,\mathrm{d}s
  \leq C_k \int_0^t \|\Sigma^k(s)\|^{2}\,\mathrm{d}s
+C_k \int_0^t  \left( \|\bm{U}^{k}(s)\|^2 + \| \Theta^k(s)\|^{2} \right) \,\mathrm{d}s,
\label{sig}
\end{aligned}
\ee
where estimates for $\|\sigma_{2}^{k}\|_{L^{\infty}(0,T;L^2(\Omega)) }$,
$\|\psi_{1}^{k}\|_{L^{\infty}(0,T;L^2(\Omega)) }$, $\|\psi_{2}^{k}\|_{L^{\infty}(0,T;L^2(\Omega)) }$
have been used. An application of Gronwall's lemma to \eqref{sig} yields that
\begin{align}
& \|\Sigma^k(t)\|^{2} +\int_0^t \|\nabla\Sigma^k(s)\|^2\, \mathrm{d}s
\le C_T\Big(\sup_{s\in [0,t]}\|\bm{U}^{k}(s)\|^2 +\sup_{s\in [0,t]}\|\Theta^{k}(s)\|^2\Big),
\quad \forall\,t\in[0,T],
\label{auphih1}
\end{align}
where the constant $C_T>0$ depends on $T$ and $k$.
As a consequence, the solution operator $\mathbf{F}^k_{1}$ is continuous with respect to $(\bm{u}^{k},\psi^k)$ as a mapping from $C([0,T];\bm{Y}_{k})\times C([0,T];Z_{k})$ to $X$.
\medskip

\textbf{Step 2.} Given the function $\sigma^{k}=\Phi^k_{1}(\bm{u}^k,\psi^k)$ obtained in Step 1, we now look for the ansatz
\be
\bm{v}^k=\sum_{i=1}^{k}a_{i}^{k}(t)\bm{y}_{i}(x),\quad \varphi^k=\sum_{i=1}^{k}b_{i}^{k}(t)z_{i}(x),\quad  \mu^{k}(x,t):=\sum_{i=1}^{k}c_{i}^{k}(t)z_{i}(x),
\nonumber
\ee
that satisfy the following auxiliary system for the fluid velocity and phase-field variable:
\begin{align}
	& \big(\partial_t (\widehat{\rho}(\varphi^k) \bm{v}^k),\bm{\zeta}\big)
	- \big( ( \bm{v}^k\otimes (\widehat{\rho}(\varphi^k)\bm{v}^k + \widehat{\J}^k),\nabla \bm{ \zeta}\big)
    +\big(  2\nu(\varphi^{k}) D\bm{v}^k, D\bm{\zeta}\big)
    +\gamma \big(|\nabla\vv^k|^2\nabla\vv^k,\nabla \bm{\zeta}\big)
    \notag \\
&\quad =\big((\mu^{k}-\beta'(\varphi^k)\sigma^{k})\nabla \varphi^{k},\bm {\zeta}\big)+ \frac12 \big(\widehat{R}^k\vv^k,\bm{\zeta}\big)
\notag\\
&\qquad + \frac12\big(\widehat{\rho}'(\varphi^k)(\bm{I}-\bm{P}_{Z_{k}}) \big((\bm{v}^k\cdot\nabla\varphi^k)-\div \big(m(\varphi^k)\nabla \mu^k\big)\big)\bm{v}^k,\bm{\zeta}\big),
\quad \forall\, \bm{\zeta} \in \bm{Y}_{k},
	\label{aatest3.c}\\
	&  (\partial_t  \varphi^{k},\xi)+( \bm{v}^{k} \cdot \nabla  \varphi^{k},\xi) +\big(m(\varphi^k)\nabla \mu^{k},\nabla \xi\big)
=0, \quad \forall\, \xi\in Z_{k},
\label{g1.a}\\
& (\mu^k,\xi)
=\big(- \Delta \varphi^k + \Psi'_{\epsilon}(\varphi^k)+\beta'(\varphi^k) \sigma^k, \xi\big),
\quad \forall\, \xi\in Z_{k},
\label{g4.d}
\end{align}
 in $(0,T)$,
where
\begin{align}
\label{atest.1c}
\widehat{\J}^k= -\widehat{\rho}'(\varphi^k) m(\varphi^k) \nabla \mu^k,\quad
\widehat{R}^k= -m(\varphi^k)\nabla \widehat{\rho}'(\varphi^k)\cdot \nabla \mu^k,\quad \text{in}\ \Omega\times(0,T).
\end{align}
The (finite dimensional) solution $(\bm{v}^k,\varphi^k,\mu^k)$
is subject to the boundary and initial conditions
\begin{align}
 & \bm{v}^k=\bm{0},\quad \partial_{\bm{n}}\varphi^k=\partial_{\boldsymbol{n}} \mu^k=0,&&\quad \textrm{on}\ \partial\Omega\times(0,T),
\label{aatest.boundary1}
\\
& \bm{v}^{k}|_{t=0}=\bm{P}_{\bm{Y}_{k}} \bm{v}_{0},\quad \varphi^{k}|_{t=0}=\bm{P}_{Z_{k}}\varphi_{0,n},&&\quad \text{in}\ \Omega.
\label{aatest3.cini}
\end{align}
Then we have the following lemma, whose proof is sketched in the Appendix.
\bl\label{NSSa}
Given $\sigma^k=\mathbf{F}^k_{1}(\bm{u}^k,\psi^k)$, the Faedo--Galerkin scheme \eqref{aatest3.c}--\eqref{aatest3.cini} admits a unique solution   $(\bm{v}^k,\varphi^k,\mu^k)$ on $[0,T]$ such that
$$
\bm{v}^k\in C^1([0,T];\bm{Y}_{k}),\quad
\varphi^k \in C^1([0,T];Z_{k}),\quad
\mu^k\in C^1([0,T];Z_{k}).
$$
Moreover, $(\bm{v}^k,\varphi^k)$ is bounded in $H^1(0,T;\bm{Y}_{k})\times H^1(0,T;Z_{k})$.
\el

Thanks to Lemma \ref{NSSa}, we can define the following mapping, which is determined by the unique solution to problem \eqref{aatest3.c}--\eqref{aatest3.cini}:
\begin{align*}
\mathbf{F}^k_{2}:\quad  \widehat{X} &\to\  H^1(0,T;\bm{Y}_{k})\times H^1(0,T;Z_{k}),
\\
\sigma^{k} &\mapsto\ (\bm{v}^k,\varphi^k).
\end{align*}
Furthermore, $\mathbf{F}^k_2$ is a bounded operator from $\widehat{X}$ to $H^1(0,T;\bm{Y}_{k})\times H^1(0,T;Z_{k})$.

Next, we verify the continuity of $\mathbf{F}^k_2$. Given $\sigma^{k}_{i}\in \widehat{X}$, $i=1,2$, we define the corresponding solutions $(\bm{v}^{k}_{i},\varphi^{k}_{i})=\mathbf{F}^k_2(\sigma^{k}_{i})$, $i=1, 2$, determined by Lemma \ref{NSSa} (with the same initial data $(\bm{P}_{\bm{Y}_{k}} \bm{v}_{0}, \bm{P}_{Z_{k}}\varphi_{0,n})$).
The associated chemical potentials are denoted by $\mu_i^k$, $i=1,2$, respectively. As before, we denote the differences
 $$
 \Sigma^{k}=\sigma_{1}^{k}-\sigma_{2}^{k},\quad
 \bm{V}^{k}=\bm{v}^{k}_{1}-\bm{v}^{k}_{2},\quad
 \Phi^{k}= \varphi_{1}^{k}-\varphi_{2}^{k},\quad
 \Upsilon^{k}=  \mu_{1}^{k}-\mu_{2}^{k}.
 $$
We first observe that for $i=1,2$, \eqref{aatest3.c} can be rewritten as (cf. \eqref{R-app})
\begin{align}
	& \big(  \widehat{\rho}(\varphi_i^k) \partial_t \bm{v_i}^k,\bm{\zeta}\big)
	+\big( \big( (\widehat{\rho}(\varphi_i^k)\bm{v}_i^k + \widehat{\J}_i^k)\cdot \nabla \big)\bm{v}_i^k, \bm{ \zeta}\big)
    +\big(  2\nu(\varphi_i^{k}) D\bm{v}_i^k, D\bm{\zeta}\big)
    +\gamma \big(|\nabla\vv_i^k|^2\nabla\vv_i^k,\nabla \bm{\zeta}\big)
    \notag \\
&\quad =\big((\mu_i^{k}-\beta'(\varphi_i^k)\sigma_i^{k})\nabla \varphi_i^{k},\bm {\zeta}\big)- \frac12 \big(\widehat{R}_i^k\vv_i^k,\bm{\zeta}\big)
\notag\\
&\qquad - \frac12\big(\widehat{\rho}'(\varphi_i^k)(\bm{I}-\bm{P}_{Z_{k}}) \big((\bm{v}_i^k\cdot\nabla\varphi_i^k)-\div \big(m(\varphi_i^k)\nabla \mu_i^k\big)\big)\bm{v}_i^k,\bm{\zeta}\big),
\quad \forall\, \bm{\zeta} \in \bm{Y}_{k},
	\label{aatest3.c2}
\end{align}
Taking the difference of \eqref{aatest3.c} for $\bm{v}^{k}_{i}$, $i=1,2$, and testing the resultant by $\bm{\zeta}=\bm{V}^k$, we get
\begin{align}
& \frac12\frac{\mathrm{d}}{\mathrm{d}t} \big(\widehat{\rho}(\varphi_1^k)\bm{V}^k,\bm{V}^k\big)
+\big(  2\nu(\varphi^{k}_1) D\bm{V}^k, D\bm{V}^k\big)
+\gamma\big(|\nabla \bm{v}^k_1|^2\nabla \bm{V}^k, \nabla\bm{V}^k\big)
\notag \\
&\quad =
- \big((\widehat{\rho}(\varphi_1^k)-\widehat{\rho}(\varphi_2^k))\partial_t \bm{v}_2^k,\bm{V}^k\big)
-\big(\widehat{\rho}(\varphi_1^k)( \bm{V}^k \cdot \nabla)  \bm{v}_2^k,\bm{V}^k\big)
\notag \\
&\qquad
-\big((\widehat{\rho}(\varphi_1^k)-\widehat{\rho}(\varphi_2^k))( \bm{v}_2^k \cdot \nabla)  \bm{v}_2^k,\bm{V}^k\big)
-\big(2(\nu(\varphi^{k}_1)-\nu(\varphi^{k}_2)) D\bm{v}_2^k, \nabla \bm{V}^k\big)
\notag \\
&\qquad
-\gamma\big((|\nabla \bm{v}^k_1|^2-|\nabla \bm{v}^k_2|^2)\nabla \vv^k_2, \nabla\bm{V}^k\big)
+ \big(\widehat{\rho}'(\varphi_1^k) m(\varphi_1^k) (\nabla \Upsilon^k \cdot\nabla) \vv_2^k,\bm{V}^k\big)
\notag\\
&\qquad
+ \big((\widehat{\rho}'(\varphi_1^k) m(\varphi_1^k) -\widehat{\rho}'(\varphi_2^k) m(\varphi_2^k)) (\nabla \mu_2^k \cdot\nabla) \vv_2^k,\bm{V}^k\big)
\notag \\
&\qquad
+ \big(\mu_1^k \nabla \Phi^k, \bm{V}^k\big)
+ \big(\Upsilon^k \nabla \varphi_2^k, \bm{V}^k\big)
-\big(\beta'(\varphi_1^k)\sigma_1^k\nabla \Phi^k, \bm{V}^k\big)
\notag\\
&\qquad
-\big((\beta'(\varphi_1^k)\sigma_1^k-\beta'(\varphi_2^k)\sigma_2^k)\nabla \varphi^k_2, \bm{V}^k\big)
- \frac12 \big( (m(\varphi^k_1)\nabla \widehat{\rho}'(\varphi^k_1)\cdot \nabla \Upsilon^k) \vv^k_2,\bm{V}^k\big)
  \notag \\
&\qquad
- \frac12 \big( ((m(\varphi^k_1)\nabla \widehat{\rho}'(\varphi^k_1)-m(\varphi^k_2)\nabla \widehat{\rho}'(\varphi^k_2))\cdot \nabla \mu^k_2) \vv^k_2,\bm{V}^k\big)
 \notag \\
 &\qquad + \left(\frac12\big(\widehat{\rho}'(\varphi^k_1)(\bm{I}-\bm{P}_{Z_{k}}) \big((\bm{v}_1^k\cdot\nabla\varphi^k_1)-\div \big(m(\varphi^k_1)\nabla \mu^k_1\big)\big)\bm{v}_2^k,\bm{V}^k\big)\right.
 \notag\\
&\qquad\qquad  -\left. \frac12\big(\widehat{\rho}'(\varphi^k_2)(\bm{I}-\bm{P}_{Z_{k}}) \big((\bm{v}^k_2\cdot\nabla\varphi^k_2)-\div \big(m(\varphi^k_2)\nabla \mu^k_2\big)\big)\bm{v}^k_2,\bm{V}^k\big)\right)
\notag\\
&\quad =: \sum_{j=4}^{17}I_j,
\label{diffvm}
\end{align}
where we have used integration by parts and the following identity (cf. \eqref{R-app} for $\varphi_1^k$)
\begin{align*}
& \frac12 \big(\partial_t \widehat{\rho}(\varphi_1^k)\bm{V}^k,\bm{V}^k\big)
- \big(\widehat{\rho}(\varphi_1^k)( \bm{v}_1^k \cdot \nabla)  \bm{V}^k,\bm{V}^k\big)
+ \big(\widehat{\rho}'(\varphi_1^k) m(\varphi_1^k) (\nabla \mu_1^k \cdot\nabla) \bm{V}^k,\bm{V}^k\big)
\\
&\quad = -\frac12 \big( (m(\varphi^k_1)\nabla \widehat{\rho}'(\varphi^k_1)\cdot \nabla \mu^k_1) \bm{V}^k,\bm{V}^k\big)
\notag\\
&\qquad +  \frac12\big(\widehat{\rho}'(\varphi^k_1)(\bm{I}-\bm{P}_{Z_{k}}) \big((\bm{v}^k_1\cdot\nabla\varphi^k_1)-\div \big(m(\varphi^k_1)\nabla \mu^k_1\big)\big)\bm{V}^k,\bm{V}^k\big).
\end{align*}
Let us estimate the right-hand side of \eqref{diffvm} term by term. Using H\"{o}lder's and Young's inequalities, we have
\begin{align*}
|I_4|
&\leq \|\widehat{\rho}(\varphi_1^k)-\widehat{\rho}(\varphi_2^k)\|_{L^\infty(\Omega)} \|\partial_t  \bm{v}_2^k\|\|\bm{V}^k\|
\\
&\leq C\|\Phi^k\|_{L^\infty(\Omega)}\|\partial_t  \bm{v}_2^k\|\|\bm{V}^k\| \\
&\leq \frac{\nu_*}{4}\|D\bm{V}^k\|^2
+ C_k  \|\partial_t  \bm{v}_2^k\|\|\Phi^k\|^2,
\end{align*}
\begin{align}
|I_{5}|
&\le \|\widehat{\rho}(\varphi_1^k)\|_{L^\infty(\Omega)} \|\nabla \bm{v}^{k}_{2}\|_{\bm{L}^\infty(\Omega)}
\|\bm{V}^k\|^2
 \leq C\|\bm{v}^{k}_{2}\|_{\bm{H}^3(\Omega)}\|\bm{V}^k\|^2
 \leq C_k\|\bm{v}^{k}_{2}\| \|\bm{V}^k\|^2,
\notag
\end{align}
\begin{align}
|I_{6}|
&\le \|\widehat{\rho}(\varphi_1^k)-\widehat{\rho}(\varphi_2^k)\|_{L^\infty(\Omega)}\| \bm{v}^{k}_{2}\|_{\bm{L}^6(\Omega)}
\|\nabla \bm{v}^{k}_{2}\|_{\bm{L}^3(\Omega)}\|\bm{V}^{k}\|
\notag\\
&\leq C\|\Phi^k\|_{L^\infty(\Omega)}\| \bm{v}^{k}_{2}\|_{\bm{H}^2(\Omega)}^2\|\bm{V}^{k}\|
\notag \\
&\leq \frac{\nu_*}{4}\|D\bm{V}^k\|^2
+ C_k \| \bm{v}^{k}_{2}\|^4 \|\Phi^k\|^2,\notag
\end{align}
\begin{align}
|I_{7}|
&\le 2\|\nu(\varphi_1^k)-\nu(\varphi_2^k)\|_{L^\infty(\Omega)}\|D \bm{v}^k_2\|\|\nabla \bm{V}^k\|
\notag\\
& \le C\|\Phi^k\|_{L^\infty(\Omega)}  \|D\bm{v}^k_2\|\|D \bm{V}^k\|
\notag\\
& \le \frac{\nu_*}{4}\|D\bm{V}^k\|^2
+ C_k\|\bm{v}^k_2\|^2\|\Phi^{k}\|^2,
\notag
\end{align}
\begin{align}
|I_8|
&\leq \gamma \big(\|\nabla \bm{v}^k_1\|_{\bm{L}^4(\Omega)} +\|\nabla \bm{v}^k_2\|_{\bm{L}^4(\Omega)}\big)\|\nabla \vv^k_2\|_{\bm{L}^4(\Omega)}\| \nabla\bm{V}^k\|^2_{\bm{L}^4(\Omega)}
\notag\\
&\leq C_k\big(\| \bm{v}^k_1\|^2 +\| \bm{v}^k_2\|^2\big) \| \bm{V}^k\|^2.
\notag
\end{align}
To proceed, we need some estimates for the chemical potential.
Taking $\xi=\mu_1^k$ in \eqref{g4.d} for $\mu_1^k$, we find
\begin{align*}
\| \mu_1^k\|^2
&\leq   \| \Delta\varphi_1^k \|  \| \mu_1^k\| + \left( \| \Psi'_\epsilon(\varphi_1^k) \| + \|\beta'(\varphi^k_1)\|_{L^\infty(\Omega)} \| \sigma_1^k\| \right) \| \mu_1^k\|.
\end{align*}
Since for any given $k$ there exists a constant $\widetilde{C}_k>0$ such that $\| \varphi_i^k \|_{C([0,T]; H^2(\Omega))}\leq \widetilde{C}_k$, $i=1,2$, it follows from the Sobolev embedding theorem and the construction of $\Psi_\epsilon'$ that
\begin{equation}
\label{muk-1}
\| \mu_1^k\|\leq  C_k+ C \| \sigma_1^k\|,
\end{equation}
where the positive constant $C_k$ depends on $k$ and $\epsilon$.
The same result holds for $\mu_2^k$. Next, taking $\xi=\Upsilon^k$ in \eqref{g4.d} for the difference $\Upsilon^k$, we easily obtain
\begin{align}
\| \Upsilon^k\|^2
&\leq   \big( \| \Delta \Phi^k \|
+ \| \Psi'_\epsilon(\varphi_1^k) - \Psi'_\epsilon(\varphi_2^k)\|\big) \| \Upsilon^k\|
\notag\\
&\quad +  \big( \|\beta'(\varphi^k_1)\|_{L^\infty(\Omega)}\| \Sigma^k\|
+ \|\beta'(\varphi^k_1)-\beta'(\varphi^k_2)\|_{L^\infty(\Omega)}\| \sigma^k_2\| \big) \| \Upsilon^k\|,
\notag
\end{align}
which entails
\begin{equation}
\label{muk-dif}
\| \Upsilon^k\|\leq  C_k (1+\| \sigma^k_2\|)\| \Phi^k\| + C \| \Sigma^k\|.
\end{equation}
As a consequence, we can deduce that
\begin{align*}
|I_9|
&\leq \|\widehat{\rho}'(\varphi_1^k) m(\varphi_1^k)\|_{L^\infty(\Omega)} \|\nabla \Upsilon ^k\| \|\nabla \vv_2^k\|_{\bm{L}^3(\Omega)} \|\bm{V}^k\|_{\bm{L}^6(\Omega)}
\\
&\leq C_k\big((1+\| \sigma^k_2\|)\| \Phi^k\|
+ \| \Sigma^k\|\big) \|\bm{v}^k_2\| \|\bm{V}^k\|
\\
&\leq C_k(1+\| \sigma^k_2\|^2) (\| \Phi^k\|^2 + \| \Sigma^k\|^2)
+ \|\bm{v}^k_2\|^2 \|\bm{V}^k\|^2,
\end{align*}
\begin{align*}
|I_{10}|&\leq \|\widehat{\rho}'(\varphi_1^k) m(\varphi_1^k) -\widehat{\rho}'(\varphi_2^k) m(\varphi_2^k)\|_{L^\infty(\Omega)} \|\nabla \mu_2^k\| \| \nabla \vv_2^k\|_{\bm{L}^3(\Omega)} \|\bm{V}^k\|_{\bm{L}^6(\Omega)}
\\
&\leq C_k \|\Phi^k\|(1+\|\sigma_2^k\|) \|\bm{v}^k_2\| \|\bm{V}^k\|
\\
&\leq C_k(1+\| \sigma^k_2\|^2) \| \Phi^k\|^2
+ \|\bm{v}^k_2\|^2 \|\bm{V}^k\|^2,
\end{align*}
\begin{align*}
|I_{11}| + |I_{12}|
&\leq  \|\mu_1^k\|\|\nabla \Phi^k\|_{\bm{L}^3(\Omega)} \|\bm{V}^k\|_{\bm{L}^6(\Omega)}
+ \|\Upsilon^k\| \| \nabla \varphi_2^k\|_{\bm{L}^3(\Omega)} \| \bm{V}^k\|_{\bm{L}^6(\Omega)}
\\
&\leq C_k(1+\|\sigma_1^k\|) \|\Phi^k\| \|\bm{V}^k\|
+ C_k \big((1+\| \sigma^k_2\|) \| \Phi^k\|
+ \|\Sigma^k\|\big) \|\varphi_2^k\| \| \bm{V}^k\|
\\
&\leq  C_k(1+\| \sigma^k_1\|^2 + \| \sigma^k_2\|^2)
\| \Phi^k\|^2 + C(1+\|\varphi^k_2\|^2) \|\bm{V}^k\|^2
+ C \|\Sigma^k\|^2,
\end{align*}
\begin{align*}
|I_{13}|
&\leq \|\beta'(\varphi_1^k)\|_{L^\infty}\|\sigma_1^k\|\|\nabla \Phi^k\|_{\bm{L}^3(\Omega)} \| \bm{V}^k\|_{\bm{L}^6(\Omega)}
\\
&\leq C_k\|\sigma_1^k\|\|\Phi^k\|\| \bm{V}^k\|
\\
&\leq C_k\|\sigma_1^k\|^2 \|\Phi^k\|^2
+ \| \bm{V}^k\|^2,
\end{align*}
\begin{align*}
|I_{14}|
&\leq \|\beta'(\varphi_1^k)\sigma_1^k-\beta'(\varphi_2^k)\sigma_2^k\|
\|\nabla \varphi^k_2\|_{\bm{L}^3(\Omega)} \|\bm{V}^k\|_{\bm{L}^6(\Omega)}
\\
&\leq C_k(\|\Sigma^k\| +\|\sigma_2^k\|\|\Phi^k\|) \|\varphi_2^k\|\|\bm{V}^k\|
\\
&\leq C_k\|\Sigma^k\|^2 + C_k\|\varphi_2^k\|^2\|\bm{V}^k\|^2
+ C_k\|\sigma_2^k\|^2\|\Phi^k\|^2,
\end{align*}
\begin{align*}
|I_{15}|
&\leq \|m(\varphi^k_1)\nabla \widehat{\rho}'(\varphi^k_1)\|_{\bm{L}^\infty(\Omega)}
\| \nabla \Upsilon^k\| \| \vv^k_2\|_{\bm{L}^3(\Omega)} \|\bm{V}^k\|_{\bm{L}^6(\Omega)}
\\
&\leq C_k\big((1+\| \sigma^k_2\|)\| \Phi^k\|
+ \| \Sigma^k\|\big) \| \vv^k_2\| \|\bm{V}^k\|
\\
&\leq C_k (1+\| \sigma^k_2\|^2) \| \Phi^k\|^2
+ C_k\|\Sigma^k\|^2 + C_k\| \vv^k_2\|^2 \|\bm{V}^k\|^2,
\end{align*}
\begin{align*}
|I_{16}|
&\leq \|m(\varphi^k_1)\nabla \widehat{\rho}'(\varphi^k_1)
-m(\varphi^k_2)\nabla \widehat{\rho}'(\varphi^k_2)\|_{L^\infty(\Omega)}
\| \nabla \mu^k_2\| \|\vv^k_2\|_{\bm{L}^3(\Omega)} \|\bm{V}^k\|_{\bm{L}^6(\Omega)}
\\
&\leq C_k\|\Phi^k\|(1+\|\sigma_2^k\|)\| \vv^k_2\| \|\bm{V}^k\|
\\
&\leq C_k (1+\| \sigma^k_2\|^2) \| \Phi^k\|^2
+ C_k\| \vv^k_2\|^2 \|\bm{V}^k\|^2,
\end{align*}
and finally, some calculations in a similar fashion yield
\begin{align*}
|I_{17}|
&\leq \|\widehat{\rho}'(\varphi^k_1)\|_{L^\infty(\Omega)}
\big\|(\bm{I}-\bm{P}_{Z_{k}}) \big((\bm{v}_1^k\cdot\nabla\varphi^k_1)-\div \big(m(\varphi^k_1)\nabla \mu^k_1\big)\big)
\notag\\
&\quad -(\bm{I}-\bm{P}_{Z_{k}}) \big((\bm{v}_2^k\cdot\nabla\varphi^k_2)-\div \big(m(\varphi^k_2)\nabla \mu^k_2\big)\big)\big\|
\|\vv_2^k\|_{\bm{L}^3(\Omega)} \|\bm{V}^k\|_{\bm{L}^6(\Omega)}
\notag\\
&\quad + \|\widehat{\rho}'(\varphi^k_1)- \widehat{\rho}'(\varphi^k_2)\|_{L^\infty(\Omega)} \|(\bm{I}-\bm{P}_{Z_{k}}) \big((\bm{v}_2^k\cdot\nabla\varphi^k_2)-\div \big(m(\varphi^k_2)\nabla \mu^k_2\big)\big)\| \|\vv_2^k\|_{\bm{L}^3(\Omega)} \|\bm{V}^k\|_{\bm{L}^6(\Omega)}
\notag\\
&\leq C_k\big( \|\vv^k_1\| \|\Phi^k\|
+ \|\vv^k\|\|\varphi^k_2\|
+ \|\varphi^k_1\|\|\Upsilon^k\|
+ \|\Phi^k\|\|\mu^k_2\|
+ \|\Phi^k\|\|\varphi^k_2\|\|\mu^k_2\|\big)
\|\vv_2^k\| \|\bm{V}^k\|
\notag\\
&\quad + C_k(\|\Upsilon^k\|
+ \|\Phi^k\|\|\mu^k_2\|) \|\vv_2^k\| \|\bm{V}^k\|
+ C_k\|\Phi^k\| \big(\|\vv^k_2\| \|\varphi^k_2\|
+ \|\varphi^k_2\| \|\mu_2^k\|
+ \|\mu^k_2\|\big) \|\vv_2^k\| \|\bm{V}^k\|
\notag\\
&\leq C_k(\|\bm{V}^k\|^2 + \|\Phi^k\|^2 + \|\Sigma^k\|^2\big).
\end{align*}
In the above estimates, we have essentially used the fact that $\bm{v}^k$, $\varphi^k$ and $\mu^k$ are finite dimensional. Hence, all related estimates for higher-order norms depend on the parameter $k$ at this stage.

Next, taking the difference of \eqref{g1.a} for $\varphi^k_i$, we see that
\begin{align}
(\partial_t  \Phi^{k},\xi)
&= -(\Phi^{k}\bm{v}^{k}_{1},\nabla \xi)
- (\varphi^{k}_{2}\bm{V}^{k},\nabla \xi)
- \big(m(\varphi^k_1)\nabla \Upsilon^{k},\nabla \xi\big)
\notag \\
&\quad - \big((m(\varphi^k_1)-m(\varphi^k_2))\nabla \mu^{k}_2,\nabla \xi\big),
\quad \forall\, \xi\in  Z_{k}.
\label{atest111.a}
\end{align}
Since $\overline{\Phi^k}=0$, we can take $\xi=\mathcal{N} \Phi^{k}$ in \eqref{atest111.a} and obtain
\begin{align}
\frac{1}{2}\frac{\mathrm{d}}{\mathrm{d}t} \|\Phi^{k} \|_{V_0^{-1}}^2
&= -\big(\Phi^{k}\bm{v}^{k}_{1},\nabla \mathcal{N} \Phi^{k} \big)
- \big(\varphi^{k}_{2}\bm{V}^{k},\nabla\mathcal{N} \Phi^{k} \big)
- \big(m(\varphi^k_1)\nabla \Upsilon^{k},\nabla \mathcal{N} \Phi^{k} \big)
\notag \\
&\quad
- \big((m(\varphi^k_1)-m(\varphi^k_2))\nabla \mu^{k}_2,\nabla \mathcal{N} \Phi^{k} \big).
\label{vhmna}
\end{align}
The terms on the right-hand side of \eqref{vhmna} can be estimated in a similar manner as above, so we get
\begin{align}
\frac{1}{2}\frac{\mathrm{d}}{\mathrm{d}t} \|\Phi^{k} \|_{V_0^{-1}}^2
&\leq C_k \|\Phi^{k}\|^2
+ C_k \|\Phi^{k}\|_{V_0^{-1}}^2
+ C_k \|\Sigma^k\|^2
+ \frac{\nu_*}{4}\|D\bm{V}^{k}\|^2.
\label{vhmnb}
\end{align}
Furthermore, the interpolation inequality \eqref{inter-1} yields
\begin{align}
\|\Phi^k\|^2
&\leq C_k  \|\Phi^{k} \|_{V_0^{-1}}^2.
\notag
\end{align}
Collecting the above estimates, we can deduce from \eqref{diffvm} and \eqref{vhmnb} that
\begin{align}
& \frac{1}{2}\frac{\mathrm{d}}{\mathrm{d}t}
\left(\big(\widehat{\rho}(\varphi_1^k)\bm{V}^k,\bm{V}^k\big)
+ \|\Phi^{k} \|_{V_0^{-1}}^2 \right)
\le \frac{C_k}{\rho_*}\left(\big(\widehat{\rho}(\varphi_1^k)\bm{V}^k,\bm{V}^k\big)
+ \|\Phi^{k} \|_{V_0^{-1}}^2 \right)
+ C_k \|\Sigma^k\|^2.
\label{astarphi}
\end{align}
Applying Gronwall's lemma to \eqref{astarphi}, and recalling that  $\bm{V}^k|_{t=0}=\bm{0}$, $\Phi^k|_{t=0}=0$, we obtain
\begin{align}
&\big(\widehat{\rho}(\varphi_1^k(t))\bm{V}^k(t),\bm{V}^k(t)\big)
+ \|\Phi^{k}(t) \|_{V_0^{-1}}^2
 \le  C_k e^{C_k t} \int_0^t\| \Sigma^k(s)\|^2\,\mathrm{d}s,
 \quad \forall\, t\in [0,T].
\label{aum}
\end{align}
Since both $\bm{V}^k$ and $\Xi^k$ are finite dimensional and $\widehat{\rho}(\varphi_1^k)\geq \rho_*$, we conclude from \eqref{aum} that the solution operator $\mathbf{F}^k_2$ is continuous as a mapping from $X$ to $C([0,T];\bm{Y}_{k})\times C([0,T];Z_{k})$.
\medskip

\textbf{Step 3.} Define the composite mapping $\mathbf{F}^k:=\mathbf{F}_{2}^k\circ \mathbf{F}_{1}^k $ as
\begin{align*}
\mathbf{F}^k:\ \ C^\delta([0,T];\bm{Y}_{k}) \times C^\delta([0,T];Z_{k}) \ &\ \to H^1(0,T;\bm{Y}_{k})\times H^1(0,T;Z_{k}),
\\
(\bm{u}^k,\psi^k)\ &\ \mapsto(\bm{v}^{k},\varphi^k).
\end{align*}
From the compactness of $H^1(0,T;\bm{Y}_{k})$ in
$C^\delta([0,T];\bm{Y}_{k})$ and $H^1(0,T;Z_{k})$ in
$C^\delta([0,T];Z_{k})$ (recalling that $\bm{Y}_{k}$ and $Z_{k}$ are finite-dimensional), we find that $\mathbf{F}^k$ is a compact operator from $C^\delta([0,T];\bm{Y}_{k}) \times C^\delta([0,T];Z_{k})$ into itself.
On the other hand, it follows from the continuous dependence estimates \eqref{auphih1} and \eqref{aum} that
\be
\begin{aligned}
&\sup_{t\in [0,T]}\|\bm{v}_1^{k}(t)-\bm{v}^{k}_2(t)\| +
\sup_{t\in [0,T]}\|\varphi_1^{k}(t)-\varphi^{k}_2(t)\|
\\
&\quad \leq C_T\left(\sup_{t\in [0,T]}\|\bm{u}_1^{k}(t)-\bm{u}_2^{k}(t)\|
+\sup_{t\in [0,T]}\|\psi_1^{k}(t)-\psi_2^{k}(t)\|\right).
\end{aligned}
\notag
\ee
Due to the boundedness of $(\bm{v}^k_i,\varphi^k_i)$ in $H^1(0,T;\bm{Y}_{k})\times H^1(0,T;Z_{k})$, $i=1,2$, we can conclude by interpolation that $\mathbf{F}^k$ is a continuous operator from $C^\delta([0,T];\bm{Y}_{k}) \times C^\delta([0,T];Z_{k})$ into itself.

Take
$$
\widetilde{M}= 2\mathrm{e}( C_0+ 1) +2C_{0,k} + 2(\|\bm{v}_0\|^2+  \|\varphi_{0,n}\|^2  +1),
$$
where the positive constants $C_0$, $C_{0,k}$ are given in \eqref{iniC0} and \eqref{me2}, respectively.
According to Lemma \ref{fp} and estimate \eqref{auvm1}, there exists a sufficiently small time $T_*\in (0,T]$ depending on $\widetilde{M}$ such that
\begin{equation}
\|\bm{v}^{k}(t)\|^2+ \|\varphi^{k}(t)\|^2
\leq
2 \mathrm{e} ( C_0 + 1 ) + 2 C_{0,k}<\widetilde{M},\qquad \forall\, t\in [0,T_*].
\notag
\end{equation}
Define
\be
\begin{aligned}
\bm{B}_k&=\Big\{(\bm{u}^k,\psi^k)\in C^\delta([0,T_*];\bm{Y}_{k}) \times C^\delta([0,T_*];Z_{k})\ \Big|\ \sup_{t\in[0,T_*]}\|\bm{u}^k(t)\|^2\leq \widetilde{M},\\
& \qquad  \sup_{t\in[0,T_*]}\|\psi^k(t)\|^2\leq \widetilde{M},\ \bm{u}^{k}(0)= \bm{P}_{\bm{Y}_{k}} \bm{v}_{0},\ \psi^{k}(0)= \bm{P}_{Z_{k}} \varphi_{0,n}\Big\},
\end{aligned}
\notag
\ee
which is a closed convex set in $C^\delta([0,T_*];\bm{Y}_{k}) \times C^\delta([0,T_*];Z_{k})$. Then for any $(\bm{u}^k,\psi^k)\in \bm{B}_k$, we find that
$$
(\bm{v}^k,\varphi^k)=\mathbf{F}^k(\bm{u}^k,\psi^k)\in H^1(0,T_*;\bm{Y}_{k})\times H^1(0,T_*;Z_{k})\subset\subset C^\delta([0,T_*];\bm{Y}_{k}) \times C^\delta([0,T_*];Z_{k}),
$$
and the pair $(\bm{v}^k,\varphi^k)$ satisfies
\begin{align}
\sup_{t\in[0,T_{*}]}\|\bm{v}^k(t)\|^2 \le\widetilde{M},\quad \sup_{t\in[0,T_{*}]}\|\varphi^k(t)\|^2 \le\widetilde{M}.\notag
\end{align}
As a result, it holds $(\bm{v}^k,\varphi^k)\in \bm{B}_k$.

We now recall the classical Schauder fixed-point theorem:
\bl\label{SH}
Assume that $\bm{K}$ is a closed convex set in a Banach space
$\mathcal{B}$. Let $\mathcal{T}$ be a continuous mapping of $\bm{K}$ into itself that satisfies that the image $\mathcal{T}\bm{K}$ is precompact.
Then there exists a fixed point in $\bm{K}$ for $\mathcal{T}$.
\el
Applying Lemma \ref{SH}, we can conclude that for the small time $T_*\in (0,T]$ chosen above, the mapping $\mathbf{F}^k$ admits a fixed point $(\bm{v}^k,\varphi^k)$ in the set $\bm{B}_k$. Then $\sigma^k$ can be determined by $(\bm{v}^k,\varphi^k)$ as in Lemma \ref{fp}. Subsequently, $\mu^k$ is determined by \eqref{atest.3}.
This gives a local solution $(\bm{v}^k,\varphi^k,\mu^k,\sigma^k)$ to the semi-Galerkin scheme \eqref{atest.1}--\eqref{atest.ini0} in the interval $[0,T_*]$. The uniqueness of the approximate solution $(\bm{v}^k,\varphi^k,\mu^k,\sigma^k)$ can be established by the standard energy method, using the facts that $\vv^k$, $\varphi^k$, $\mu^k$ are finite dimensional and $\sigma^k$ is sufficiently smooth. Thus, we omit the details here.

The proof of Proposition \ref{p1} is complete.
\end{proof}

\section{Existence of Global Finite Energy/Weak Solutions}
\label{EXIST-WEAK}
\setcounter{equation}{0}

In this section, we prove Theorem \ref{WEAK-SOL} and Corollary \ref{finite-to-weak} on the existence of a global finite energy/weak solution.

\subsection{Uniform estimates}\label{ue}
We derive several estimates for the approximate solutions $\left\{ (\bm{v}^k,\varphi^k, \mu^k,\sigma^k) \right\}$ to the semi-Galerkin scheme \eqref{atest.1}--\eqref{atest.ini0} obtained in Proposition \ref{p1}, which are uniform with respect to the approximating parameters $k$, $\epsilon$, $\gamma$, $n$ and $t\in [0,T]$.
\medskip

\textbf{Step 1. Conservation of mass.} Taking $\xi=1$ in \eqref{atest.2}, we obtain
\begin{equation}
\label{cons-mass1}
\overline{\varphi^k}(t)= \frac{1}{|\Omega|} \int_{\Omega} \varphi^k(x,t) \, \d x= \frac{1}{|\Omega|} \int_{\Omega} \varphi_{0,n}(x) \, \d x=\overline{\varphi_{0,n}}, \quad \forall \, t \in [0,T_k].
\end{equation}
Similarly, multiplying \eqref{atest.4} by $1$ and integrating over $\Omega$, we get
\begin{equation}
\label{cons-mass2}
\| \sigma^k(t)\|_{L^1(\Omega)}= \int_{\Omega} \sigma^k(x,t) \, \d x=  \int_{\Omega} \sigma_{0,n}(x) \, \d x =\| \sigma_{0,n}\|_{L^1(\Omega)}, \quad \forall \, t \in [0,T_k].
\end{equation}

\textbf{Step 2. Estimates from the basic energy law.}
Thanks to Proposition \ref{p1}, we are allowed to take $\bm{\zeta}=\boldsymbol{v}^k$ in \eqref{atest.1}. Then, using the modified mass balance equation \eqref{R-app}, we find
\begin{align}
&\frac{\d}{\d t} \int_\Omega \frac12\widehat{\rho}(\varphi^k)|\vv^k|^2\,\d x
+\int_\Omega  2\nu(\varphi^{k}) |D\bm{v}^{k}|^2\,\d x
+\gamma \int_\Omega |\nabla\vv^k|^4\,\d x
\notag \\
&\quad =\int_\Omega (\mu^{k} -\beta'(\varphi^k) \sigma^{k})\nabla \varphi^{k}\cdot \vv^k\, \d x.
\label{BEL-k-1}
\end{align}
Next, choosing $\xi= \mu^k$ in \eqref{atest.2} and $\xi= \partial_t\varphi^k$ in \eqref{atest.3}, respectively, we can deduce that
\begin{align}
& \frac{\d}{\d t} \int_\Omega \left(\frac{1}{2}|\nabla \varphi^k|^2+ \Psi_\epsilon(\varphi^k)\right)\,\d x + \int_\Omega \sigma^k\partial_t \beta(\varphi^k)  \,\d x
+ \int_\Omega (\vv^k\cdot\nabla \varphi^k)\mu^k\, \d x
\notag\\
&\quad +\int_\Omega m(\varphi^k)|\nabla \mu^k|^2\,\d x=0.
\label{BEL-k-2}
\end{align}
Finally, testing \eqref{atest.4} by $\ln \sigma^k+\beta(\varphi^k)$, using integration by parts and the fact $\mathrm{div}\, \vv^k=0$, we have
\begin{align}
& \frac{\d}{\d t} \int_\Omega  \sigma^k (\ln \sigma^k -1)\,\d x
+ \int_\Omega (\partial_t \sigma^k) \beta(\varphi^k)\,\d x
- \int_\Omega (\vv^k\cdot \nabla \varphi^k) \sigma^k \beta'(\varphi^k) \,\d x
\notag\\
& \quad +
\int_\Omega \sigma^k|\nabla (\ln \sigma^k +\beta(\varphi^k)|^2\,\d x=0.
\label{BEL-k-3}
\end{align}
Adding \eqref{BEL-k-1}--\eqref{BEL-k-3} together, we obtain the following energy identity
\begin{align}
& \frac{\mathrm{d}}{\mathrm{d}t}\widehat{\mathcal{E}}^k(t)
+ \widehat{\mathcal{D}}^k(t)
 =0,\quad \text{for a.e.}\ t\in(0,T_k),
\label{menergy}
\end{align}
where
\begin{align*}
\widehat{\mathcal{E}}^k(t)
&=\int_{\Omega}\left( \frac{1}{2}\widehat{\rho}(\varphi^k)|\boldsymbol{v}^k|^{2}
+\frac{1}{2}|\nabla \varphi^k|^{2}
+ \Psi_\epsilon(\varphi^k) +\sigma^k(\ln \sigma^k-1)
+ \beta(\varphi^k) \sigma^k\right) \, \mathrm{d} x,
\\
\widehat{\mathcal{D}}^k(t)
& =\int_{\Omega} \left(2 \nu(\varphi^k)|D \boldsymbol{v}^k|^{2}
+\gamma|\nabla \vv^k|^4
+m(\varphi^k)|\nabla \mu^k|^{2}
+ \sigma^k|\nabla(\ln \sigma^k+ \beta(\varphi^k))|^{2}\right)\, \mathrm{d} x.
\end{align*}

Using (H4) and taking $a=2\beta^*$, $b=\sigma^k-1$ in Lemma \ref{You}, we observe that
\begin{align}
\left|\int_\Omega \beta(\varphi^k) \sigma^k\,\mathrm{d}x\right|
&\leq \beta^*\int_{\{0\leq \sigma^k\leq 1\}} \sigma^k\,\mathrm{d}x
+ \beta^*\int_{\{\sigma^k\geq 1\}} 1\,\mathrm{d}x
+ \beta^*\int_{\{\sigma^k\geq 1\}} (\sigma^k-1)\,\mathrm{d}x
\notag \\
&\leq 2\beta^*|\Omega|
+ \frac{1}{2}\int_{\{\sigma^k\geq 1\}} \left(e^{2\beta^*}-2\beta^*-1 + \sigma^k
(\ln \sigma^k-1)+1\right)\mathrm{d}x
\notag\\
&\leq 2\beta^*|\Omega| + \frac{1}{2}e^{2\beta^*}|\Omega|
+ \frac{1}{2}\int_{\{\sigma^k\geq 1\}} \sigma^k
\ln \sigma^k\,\mathrm{d}x.
\label{es-cross1}
\end{align}
Then from (H4), the construction of the regularized initial data and \eqref{es-cross1} (applied to $\bm{P}_{Z_{k}} \varphi_{0,n}$ and $\sigma_{0,n}$), we can control the initial approximate energy as follows
\begin{align}
\widehat{\mathcal{E}}^k(0)
&=
\int_{\Omega}\left(\frac{1}{2}\widehat{\rho}(\bm{P}_{Z_{k}} \varphi_{0,n})|\bm{P}_{\bm{Y}_{k}} \bm{v}_{0}|^{2}
+\frac{1}{2}|\nabla \bm{P}_{Z_{k}}\varphi_{0,n}|^{2} +\Psi_\epsilon(\bm{P}_{Z_{k}}\varphi_{0,n})
 \right) \mathrm{d}x
\notag \\
&\quad +
\int_{\Omega}\big(
\sigma_{0,n}(\ln \sigma_{0,n}-1)
+ \beta(\bm{P}_{Z_{k}}\varphi_{0,n})\sigma_{0,n}\big) \mathrm{d}x
\notag \\
&\leq \rho^*\|\bm{v}_{0}\|^{2}
+ \frac12\|\varphi_{0,n}\|_{H^1(\Omega)}^2 + |\Omega|\max_{r\in[-1,1]}|\Psi_0(r)|
 + \frac{3}{2} \int_\Omega \sigma_0\ln \sigma_0\, \mathrm{d}x + C(\beta^*)  |\Omega|
\notag \\
&\le C\Big(\rho^*, \beta^*, \|  \bm{v}_{0}\|, \|\varphi_{0}\|_{H^1(\Omega)}, \max_{r\in[-1,1]}|\Psi_0(r)|,
\int_\Omega \sigma_0\ln \sigma_0\, \mathrm{d}x, \Omega\Big)
\notag\\
&=: \mathcal{E}_0,
\label{iniE0}
\end{align}
where the upper bound $\mathcal{E}_0$ is independent of the approximating parameters $k$, $\gamma$, $\epsilon$, $n$.
On the other hand, from (H1), the construction of $\Psi_\epsilon$ and \eqref{es-cross1}, we also find that the total energy $\mathcal{E}^k(t)$ is uniformly semi-coercive from below, that is,
\begin{align}
&\widehat{\mathcal{E}}^k(t)\geq \frac12 \int_{\Omega}\left( \rho_*|\boldsymbol{v}^k|^{2}
+ |\nabla \varphi^k|^{2}
+\sigma^k\ln \sigma^k
\right) (t) \,\mathrm{d} x -C_*,
\label{Lowerbd-2}
\end{align}
where the constant $C_*>0$ is independent of $k$, $\gamma$, $\epsilon$, $n$ and time $t$.

Integrating \eqref{menergy} with respect to time, using \eqref{iniE0}, \eqref{Lowerbd-2}, we can deduce the following uniform bounds
\begin{align}
&\|\bm{v}^k\|_{L^{\infty}(0, T_k ; \bm{L}^2(\Omega))}
+\|\varphi^k\|_{L^{\infty}(0, T_k ; H^1(\Omega))}
\notag \\
&\quad +\|\sigma^k \ln \sigma^k\|_{L^{\infty}(0, T_k ; L^{1}(\Omega))} +\|\Psi_{0,\epsilon}(\varphi^k)\|_{L^{\infty}(0, T_k; L^{1}(\Omega))}
\leq C,
\label{energy-a}
\end{align}
and
\begin{align}
&\|\bm{v}^k\|_{L^{2}(0, T_k ; \bm{H}^1(\Omega))}
+ \gamma^\frac14\|\nabla \bm{v}^k\|_{L^{4}(0, T_k ; \bm{L}^4(\Omega))}
+\|\nabla \mu^k\|_{L^{2}(0, T_k ; \bm{L}^{2}(\Omega))}
\notag\\
&\quad + \big\| (\sigma^k)^\frac{1}{2}\nabla(\ln \sigma^k +\beta(\varphi^k))\big\|_{L^{2}(0, T_k; \bm{L}^{2}(\Omega))}
\leq C,
\label{energy-b}
\end{align}
where the constant $C>0$ depends on $\mathcal{E}_0$, $\Omega$ and coefficients of the system, but it is independent of $k$, $\gamma$, $\epsilon$, $n$ and time $t$.

\begin{remark}\rm
\label{app-glosol}
For any $k\geq \widehat{k}$, the estimates \eqref{energy-a}, \eqref{energy-b} allow us to extend the (unique) local solution
$(\bm{v}^k,\varphi^k,\mu^k,\sigma^k)$
from $[0,T_k]$ to the whole interval $[0,\infty)$. This yields a unique global solution at the approximate level. Moreover, the estimates \eqref{energy-a}, \eqref{energy-b} hold on $[0,\infty)$, and the constant $C>0$ is independent of $k$, $\gamma$, $\epsilon$, $n$ as well as time $t$.
\end{remark}

\textbf{Step 3. Estimate of the chemical potential $\mu^k$.}
Thanks to Remark \ref{app-glosol}, hereafter we shall work with the global approximate solution $(\bm{v}^k,\varphi^k,\mu^k,\sigma^k)$ that is defined in $[0,\infty)$.

The construction of $\Psi_{0,\epsilon}'$ implies that   $\Psi_{0,\epsilon}'(r)=\Psi_0'(r)$ for $r\in [-1+\epsilon,1-\epsilon]$. Hence, thanks to \cite[Proposition A.1]{MZ04} and the mass conservation property \eqref{cons-mass1}, we deduce that
$$
|\Psi_{0,\epsilon}'(r)|\leq c_1 \Psi_{0,\epsilon}'(r)\big( r - \overline{\varphi^k}\big)+c_2,
\qquad \forall\, r\in [-1+\epsilon,\,1-\epsilon],
$$
where the positive constants $c_1$, $c_2$ depend on $\overline{\varphi_{0,n}}$ (and indeed only on $\overline{\varphi_{0}}$ since $\overline{\varphi_{0,n}} \in \big[-|\overline{\varphi_{0}}|,\,|\overline{\varphi_{0}}|\big]$), but not on $k$, $\gamma$, $\epsilon$, $n$.
In addition, we have
$$
\Psi_{0,\epsilon}'(r)\big(r-\overline{\varphi^k}\big)\geq \left(1-\epsilon-\overline{\varphi^k}\right)\Psi_{0,\epsilon}'(r) \geq \frac12(1-|\overline{\varphi_{0,n}}|)\Psi_{0,\epsilon}'(r),\quad \forall\, r\geq 1-\epsilon,
$$
and
$$
\Psi_{0,\epsilon}'(r)\big(r-\overline{\varphi^k}\big)\geq \left(-1+\epsilon-\overline{\varphi^k}\right)\Psi_{0,\epsilon}'(r)\geq -\frac12(1-|\overline{\varphi_{0,n}}|)\Psi_{0,\epsilon}'(r),\quad \forall\, r\leq -1+\epsilon.
$$
Thus, there exist two positive constants $\widetilde{c}_1$, $\widetilde{c}_2$ depending on $\overline{\varphi_{0}}$, but not on $k$, $\gamma$, $\epsilon$, $n$ such that
\begin{align}
& \|\Psi_{0,\epsilon}'(\varphi^k)\|_{L^1(\Omega)}\leq \widetilde{c}_1 \int_\Omega \Psi_{0,\epsilon}'(\varphi^k)\big(\varphi^k-\overline{\varphi^k}\big)\,\mathrm{d}x +\widetilde{c}_2.
\label{Psi-L1}
\end{align}
Next, we apply Lemma \ref{TM-var} (with $\eta=1$) together with the estimates \eqref{cons-mass2}, \eqref{energy-a} to conclude
\begin{align}
\left|\int_{\Omega}  \sigma^k \varphi^k\, \mathrm{d}x\right|
 &\leq   \int_{\Omega} \sigma^k\big(\ln \sigma^k - \ln \overline{\sigma_{0,n}}\big)\mathrm{d}x + \|\sigma_{0,n}\|_{L^1(\Omega)}\|\nabla \varphi^k\|^2
 \notag \\
&\quad + M  \|\sigma_{0,n}\|_{L^1(\Omega)} \|\varphi^k\|_{L^1(\Omega)}^2 + M\|\sigma_{0,n}\|_{L^1(\Omega)}
\notag\\
&\leq C,
\label{Lowerbd-1}
\end{align}
where the constant $C>0$ depends on $\mathcal{E}_0$, $\Omega$, but not on $k$, $\gamma$, $\epsilon$, $n$.
Taking $\xi=\varphi^k-\overline{\varphi^k}$ in \eqref{atest.3}, we find that
\begin{align}
&   \|\nabla \varphi^k\|^2+ \int_\Omega \Psi_{0,\epsilon}'(\varphi^k) (\varphi^k-\overline{\varphi^k})\,\mathrm{d}x
\notag \\
&\quad = \int_\Omega \big(\mu^k-\overline{\mu^k}\big) \big( \varphi^k-\overline{\varphi^k}\big) \,\mathrm{d}x
+ \theta_0\|\varphi^k-\overline{\varphi^k}\|^2
- \int_\Omega \beta'(\varphi^k)\sigma^k(\varphi^k-\overline{\varphi^k})\,\mathrm{d}x
\notag \\
&\quad \leq C\|\nabla \mu^k\|\|\nabla \varphi^k\|+ C|\theta_0|\|\nabla \varphi^k\|^2
+C,
\label{Psi-L2}
\end{align}
where in the last line we have used (H4), \eqref{Lowerbd-1}, and the Poincar\'{e}--Wirtinger inequality. On the other hand, testing \eqref{atest.3} by $\xi=1$ yields
\begin{align}
|\overline{\mu^{k}}|
&=
\frac{1}{|\Omega|} \left| \int_\Omega \Psi'_\epsilon(\varphi^{k}) \,\mathrm{d}x + \int_\Omega \beta'(\varphi^k)\sigma^{k} \,\mathrm{d}x \right|
\notag \\
 & \le C\|\Psi'_\epsilon(\varphi^{k})\|_{L^1(\Omega)}
 +C \int_\Omega \sigma^{k}(\ln \sigma^k-1)\,\mathrm{d}x+C.
\label{average valuea}
\end{align}
Then from \eqref{energy-a}--\eqref{average valuea} and the Poincar\'{e}--Wirtinger inequality, we can conclude
that %
\begin{align}
\sup_{t\geq 0}\|\Psi'_\epsilon(\varphi^{k})\|_{L^2(t,t+1;L^1(\Omega))}\leq C,
\label{Psid-1}
\end{align}
and
\begin{equation}
\sup_{t\geq 0} \|  \mu^{k} \|_{L^{2}(t,t+1;H^1(\Omega))}\le C,
\label{mu}
\end{equation}
where the constant $C>0$ is independent of $k$, $\gamma$, $\epsilon$, $n$.
\medskip

\textbf{Step 4. Sobolev estimates of $(\varphi^k,\sigma^k)$, Part I.}
A direct calculation yields
\begin{align}
&
\int_\Omega \sigma^k \big| \nabla (\ln \sigma^k +\beta(\varphi^k))\big|^2 \, \d x
\notag \\
&\quad =\int_\Omega \frac{|\nabla \sigma^k|^2}{\sigma^k} \, \d x
+ \int_\Omega |\beta'(\varphi^k)|^2 \sigma^k |\nabla \varphi^k|^2 \, \d x
+\int_\Omega 2 \beta'(\varphi^k) \nabla \sigma^k \cdot \nabla \varphi^k \, \d x
\notag \\
&\quad =\int_\Omega 4 \big|\nabla \sqrt{\sigma^k}\big|^2 \, \d x
+ \int_\Omega |\beta'(\varphi^k)|^2 \sigma^k |\nabla \varphi^k|^2 \, \d x
-\int_{\Omega} 2 \beta'(\varphi^k) \sigma^k \, \Delta \varphi^k \, \d x
\notag \\
&\qquad -\int_{\Omega} 2 \beta''(\varphi^k)  \sigma^k |\nabla  \varphi^k|^2 \, \d x.
\label{dis-1}
\end{align}
We now estimate the last two terms on the right-hand side of \eqref{dis-1}. Using (H4), \eqref{energy-a}, the Sobolev embedding theorem and the elliptic regularity theory, we obtain
\begin{align}
&\left|\int_{\Omega} 2 \beta'(\varphi^k) \sigma^k \, \Delta \varphi^k \, \d x \right| +
\left|\int_{\Omega} 2 \beta''(\varphi^k)  \sigma^k |\nabla  \varphi^k|^2 \, \d x \right|
\notag \\
&\quad \leq 2 \|\beta'(\varphi^k)\|_{L^\infty(\Omega)}\|\sigma^k\|\| \Delta \varphi^k\| + 2 \|\beta''(\varphi^k)\|_{L^\infty(\Omega)}  \|\sigma^k\| \|\nabla  \varphi^k\|_{\bm{L}^4(\Omega)}^2
\notag\\
&\quad \leq C\|\sigma^k\|(\| \Delta \varphi^k\| + \|\nabla  \varphi^k\| \|\nabla  \varphi^k\|_{\bm{H}^1(\Omega)})
\notag\\
&\quad \leq C\|\sigma^k\|\| \Delta \varphi^k\|.
\label{dis-1b}
\end{align}
Testing \eqref{atest.3} by $\xi=-\Delta \varphi^k$, we have
\begin{align}
&\| \Delta \varphi^k\|^2 + \int_{\Omega}\Psi_{0,\epsilon}''(\varphi^k)|\nabla \varphi^k|^2 \, \d x
\notag\\
& \quad =
\int_{\Omega} \nabla \mu^k \cdot \nabla \varphi^k \, \d x
+ \theta_0  \int_{\Omega} |\nabla \varphi^k|^2 \, \d x
 + \int_{\Omega} \beta'(\varphi^k)\sigma^k \Delta \varphi^k  \, \d x
\notag \\
&\quad \leq \|\nabla \mu^k\|\| \nabla \varphi^k\| + |\theta_0| \|\nabla \varphi^k\|^2 + \frac12 \|\Delta\varphi^k\|^2 + C \|\sigma^k\|^2,
\label{phi-H2-b}
\end{align}
which together with \eqref{energy-a} yields
\begin{equation}
\label{phi-H2-est}
\frac12 \| \Delta \varphi^k\|^2  + \int_{\Omega}\Psi_{0,\epsilon}''(\varphi^k)|\nabla \varphi^k|^2 \, \d x
\leq
C \left( 1+ \| \nabla \mu^k\| + \|\sigma^k\|^2\right).
\end{equation}
Combining \eqref{dis-1}, \eqref{dis-1b} and \eqref{phi-H2-est}, we end up with
\begin{equation}
\label{dis-2}
\begin{split}
\int_\Omega  \big|\nabla \sqrt{\sigma^k}\big|^2 \, \d x
\leq \frac14\int_\Omega \sigma^k \big| \nabla (\ln \sigma^k+\beta(\varphi^k))\big|^2 \, \d x
+ C\left(1+\| \nabla \mu^k\|^2 + \|\sigma^k\|^2\right).
\end{split}
\end{equation}
Now, we are left to control the crucial term $\|\sigma^k\|$. To this end, an application of Lemma \ref{GN-ln} with $u=\sqrt{\sigma^k}$, $q=4$, $r=2$, $\alpha=\frac12$ leads to
\begin{align}
\| \sigma^k\|^2
= \big\| \sqrt{\sigma^k}\big\|_{L^4(\Omega)}^4
&\leq \eta \big\| \nabla \sqrt{\sigma^k}\big\|^{2}
\big\| \sqrt{\sigma^k} \ln^\frac12 \sqrt{\sigma^k}\big\|^2
+C \big\|  \sqrt{\sigma^k} \big\|^4 +C_\eta
\notag \\
&= \eta \big\| \nabla \sqrt{\sigma^k}\big\|^{2}
\int_{\Omega} \big|\sigma^k \ln \sigma^k\big| \, \d x
+C \left( \int_{\Omega} \sigma^k \, \d x \right)^2 +C_\eta.
\notag
\end{align}
Thus, in light of \eqref{cons-mass2} and \eqref{energy-a}, we infer from the above inequality that
\begin{equation}
\label{s-L2}
\| \sigma^k\|^2
\leq C\eta \big\| \nabla \sqrt{\sigma^k}\big\|^{2}
+ C_\eta,
\end{equation}
where $C>0$ is independent of $k$, $\gamma$, $\epsilon$, $n$ and $t$. Choosing $\eta$ sufficiently small in \eqref{s-L2}, and using \eqref{dis-2}, we arrive at
\begin{equation}
\label{dis-3}
\begin{split}
\big\|\nabla \sqrt{\sigma^k}\big\|^2
\leq \int_\Omega \sigma^k \big| \nabla (\ln \sigma^k+\beta(\varphi^k))\big|^2 \, \d x
+ C \big(1+ \|\nabla \mu^k\|^2\big).
\end{split}
\end{equation}
Owing to \eqref{energy-b}, an integration of \eqref{dis-3} in time yields
\begin{equation}
\label{dis-4}
\sup_{t\geq 0}\big\| \nabla \sqrt{\sigma^k} \big\|_{L^2(t,t+1; \bm{L}^2(\Omega))}
\leq C,
\end{equation}
where $C>0$ is independent of $k$, $\gamma$, $\epsilon$, $n$ and $t$.
This enables us to deduce from \eqref{s-L2} that
\begin{equation}
\label{s-L2-est}
\sup_{t\geq 0}\| \sigma^k \|_{L^2(t,t+1; L^2(\Omega))} \leq C.
\end{equation}
The above estimate combined with \eqref{energy-a} and \eqref{phi-H2-est} further entails that
\begin{equation}
\label{phi-h-2}
\sup_{t\geq 0}\| \varphi^k\|_{L^2(t,t+1; H^2(\Omega))}
\leq C.
\end{equation}
Finally, we can apply \eqref{energy-a} and the growth property of $\Psi_{\epsilon}$ to conclude that
\begin{align}
	\sup_{t\geq 0}\|\Psi'_{\epsilon}(\varphi^k)\|_{L^2(t,t+1;L^2(\Omega))}\leq C,
	\label{Psip0b}
\end{align}
where $C>0$ depends on $\epsilon$, $\mathcal{E}_0$, $\Omega$ and the coefficients of the system, but it is independent of $k$, $\gamma$, $n$.

\medskip

\textbf{Step 5. Sobolev estimates of $(\varphi^k, \sigma^k)$, Part II.}
We proceed to derive refined estimates with the help of additional information from $\|\sigma_{0,n}\|$. Multiplying \eqref{atest.4} by $\sigma^k$ and integrating over $\Omega$, after integration by parts, we get
\begin{equation}
\label{sL2-1}
\frac12 \ddt \|\sigma^k \|^2 + \|\nabla \sigma^k \|^2
=-\int_{\Omega} \beta'(\varphi^k) \sigma^k \nabla \varphi^k \cdot \nabla \sigma^k \, \d x.
\end{equation}
It follows from H\"{o}lder's inequality, the Gagliardo--Nirenberg inequality, and Young's inequality that
\begin{align}
 \left|-\int_{\Omega} \beta'(\varphi^k) \sigma^k \nabla \varphi^k \cdot \nabla \sigma^k \, \d x\right|
& \leq  \|\beta'(\varphi^k)\|_{L^\infty(\Omega)}
\|\sigma^k\|_{L^4(\Omega)}\|\nabla \varphi^k\|_{\bm{L}^4(\Omega)}\|\nabla \sigma^k\|
 \notag \\
& \leq C \big(\|\sigma^k\| + \|\sigma^k\|^\frac12\|\nabla\sigma^k\|^\frac12\big)\|\Delta \varphi^k\|^\frac12\|\nabla \varphi^k\|^\frac12 \|\nabla \sigma^k\|
\notag\\
& \leq  \frac{1}{2}\|\nabla \sigma^k\|^2
+ C \big(1+\|\Delta \varphi^k\|^2\big)\|\sigma^k\|^2
\notag \\
&\leq \frac{1}{2}\|\nabla \sigma^k\|^2
+ C \big(1+\|\nabla \mu^k\|)\|\sigma^k\|^2 + C\|\sigma^k\|^4,
\label{sigL2-1}
\end{align}
where we have used (H4), \eqref{energy-a}, \eqref{phi-H2-est}. Hence, from \eqref{sL2-1} and \eqref{sigL2-1} we infer that
\begin{equation}
\label{sL2-2}
  \ddt \|\sigma^k \|^2 + \|\nabla \sigma^k \|^2
\leq C \big(1+\|\nabla \mu^k\|\big) \|\sigma^k\|^2
+ C\|\sigma^k\|^4.
\end{equation}
It remains to control the right-hand side of \eqref{sL2-2}.
Applying Lemma \ref{GN-ln} to $\sigma^k$ with the choice of parameters $q=2$, $r=1$, $\alpha=1$ and using \eqref{cons-mass2}, \eqref{energy-a}, we obtain
\begin{equation}
\label{sL3-3}
\| \sigma^k\|^2
\leq \eta \| \nabla \sigma^k\|\|\sigma^k\ln \sigma^k\|_{L^1(\Omega)}+ C\|\sigma^k\|^2_{L^1(\Omega)}+C_\eta.
\end{equation}
Choosing $\eta>0$ in \eqref{sL3-3} sufficiently small ($\eta$ is independent of the approximating parameters and time $t$), and inserting the resultant into \eqref{sL2-2}, we deduce from Young's inequality that
\begin{equation}
\label{sL2-4}
\begin{split}
  \ddt \|\sigma^k \|^2 + \frac12 \|\nabla \sigma^k \|^2
\leq C \left( 1+ \| \nabla \mu^k\|^2 \right) \| \sigma^k\|^2 +C.
\end{split}
\end{equation}
Recalling the following interpolation inequality
\begin{align}
\| u\|\leq C \| u\|_{L^1(\Omega)}^\frac12 \| \nabla u\|^\frac12 + C \| u\|_{L^1(\Omega)}, \quad \forall \, u \in H^1(\Omega),
\label{int-L2-H1}
\end{align}
applying \eqref{cons-mass2} and Young's inequality, we then arrive at
\begin{equation}
\label{sL2-5}
\begin{split}
 \ddt \|\sigma^k \|^2 + \frac14 \|\nabla \sigma^k \|^2
\leq C\| \nabla \mu^k\|^2 \| \sigma^k\|^2 +C,
\end{split}
\end{equation}
where the constant $C>0$ only depends on the parameters of the system, $\mathcal{E}_0$ and $\Omega$.
Thanks to the Poincar\'{e}--Wirtinger inequality and \eqref{cons-mass2}, we can rewrite \eqref{sL2-5} as follows
\begin{equation}
\label{sL2-6}
\begin{split}
\ddt \|\sigma^k \|^2 + \left( \varpi - C \|\nabla \mu^k\|^2 \right) \| \sigma^k\|^2
\leq C,
\end{split}
\end{equation}
where the constant $\varpi>0$ only depends on $\Omega$.
An application of Gronwall's lemma to \eqref{sL2-6} yields
\begin{align}
\| \sigma^k(t)\|^2
&\leq \| \sigma_{0,n}\|^2 \, {\rm exp} \left(C \int_0^t \| \nabla \mu^k(\tau)\|^2 \, \d \tau - \varpi t \right)
\notag \\
&\quad + C \int_{0}^t {\rm exp} \left(  C \int_\tau^t \| \nabla \mu^k(s)\|^2 \, \d s - \varpi (t-\tau) \right) \, \d \tau,
\quad \forall \, t \geq 0.
\label{sigma-L2-0}
\end{align}
In light of the following estimate
\begin{equation}
\label{mu-int}
\int_{0}^\infty \| \nabla \mu^k (\tau)\|^2 \, \d \tau \leq C,
\end{equation}
which is a consequence of \eqref{menergy}, \eqref{iniE0} and \eqref{Lowerbd-2}, we find that
\begin{align}
\| \sigma^k(t)\|^2
&\leq \| \sigma_{0,n}\|^2 \, {\rm exp} \left( C-\varpi t \right)
+ C \int_{0}^t {\rm exp} \left(  C- \varpi (t-\tau) \right) \, \d \tau
\notag \\
&\leq \| \sigma_{0,n}\|^2 \, {\rm exp} \left( C -\varpi t  \right)
+ \frac{C}{\varpi} {\rm exp} \left(  C \right)
, \quad \forall \, t \geq 0.
\label{sigma-L2}
\end{align}
Next, integrating \eqref{sL2-5} in time on $(t,t+1)$ for any $t \geq 0$, we get
\begin{equation}
\int_t^{t+1} \| \nabla \sigma^k(\tau)\|^2 \, \d \tau
\leq 4 \max_{ \tau\in [t,t+1]} \| \sigma^k(\tau)\|^2
\left( 1+C \int_{t}^{t+1} \| \nabla \mu^k(\tau)\|^2 \, \d \tau\right) +C.
\notag
\end{equation}
Hence, by \eqref{mu-int} and \eqref{sigma-L2}, we have
\begin{equation}
\label{sigma-H1}
\sup_{t \geq 0}
\int_t^{t+1} \| \nabla \sigma^k(\tau)\|^2 \, \d \tau
\leq  C \left( \| \sigma_{0,n}\|^2  {\rm exp} (C) + \frac{C}{\varpi} {\rm exp} (C) \right) +C.
\end{equation}
The estimates \eqref{sigma-L2} and \eqref{sigma-H1} entail that
$$
\sigma^k\ \ \text{is uniformly bounded in}\ \ L^\infty(0,\infty; L^2(\Omega))\cap L_{\uloc}^2([0,\infty); H^1(\Omega)).
$$
By interpolation, we easily find that $\sigma^k$ is uniformly bounded in $L^4_{\uloc}([0,\infty);L^4(\Omega))$. Finally, recalling \eqref{phi-H2-b} and exploiting \eqref{int-L2-H1} to $\sigma^k$, we get
\begin{align}
&\| \Delta \varphi^k\|^2 + \int_{\Omega}\Psi_{0,\epsilon}''(\varphi^k)|\nabla \varphi^k|^2 \, \d x
\notag\\
& \quad \leq \|\nabla \mu^k\|\| \nabla \varphi^k\| + |\theta_0| \|\nabla \varphi^k\|^2 + \frac12 \|\Delta\varphi^k\|^2 + C \|\nabla \sigma^k\|.
\label{phi-H2-bb}
\end{align}
Thus, from \eqref{energy-a}, \eqref{energy-b} and \eqref{sigma-H1}, we infer that
\begin{equation}
\label{phi-L4H2}
\sup_{t\geq 0} \| \varphi^k\|_{L^4(t,t+1; H^2(\Omega))}\leq C.
\end{equation}

\textbf{Step 6. Estimates on time derivatives.} In what follows, we derive estimates for the time derivative of $\bm{P}_{\bm{Y}_{k}} (\widehat{\rho}(\varphi^k)\vv^k)$, $\varphi^k$ and $\sigma^k$.

The estimate for $\partial_t\varphi^k$ is straightforward.
By the Sobolev embedding $H^2(\Omega)\hookrightarrow L^\infty(\Omega)$ and the estimates \eqref{energy-a}, \eqref{phi-h-2}, we see that
\begin{equation}
\sup _{t\geq 0}\int_t^{t+1} \|  \varphi^{k}(\tau)\bm{v}^{k}(\tau) \|^2\,\mathrm{d}\tau
\leq \|\bm{v}^{k}\|_{L^{\infty}(0,\infty;\bm{L}^2(\Omega))}^2 \sup _{t\geq 0}\int_t^{t+1} \|  \varphi^{k}(\tau)\|_{H^2(\Omega)}^2 \,\mathrm{d}\tau  \le C.
\nonumber
\end{equation}
This combined with \eqref{atest.2}, \eqref{energy-a}, \eqref{energy-b} and (H3) yields
\begin{equation}
\sup _{t\geq 0}\|\partial_{t}\varphi^{k}\|_{L^{2}(t,t+1;(H^1(\Omega))')} \le C.
\label{phimt}
\end{equation}

Next, we estimate $\bm{P}_{\bm{Y}_{k}} \partial_t(\widehat{\rho}(\varphi^k)\vv^k)$.
From \eqref{atest.1}, for any $\bm{\zeta}\in D(\bm{S})$, we have
\begin{align}
&\l\bm{P}_{\bm{Y}_{k}} \partial_t(\widehat{\rho}(\varphi^k)\vv^k),\bm{\zeta} \r_{D(\bm{S})',D(\bm{S})}
= (\partial_t(\widehat{\rho}(\varphi^k)\vv^k),\bm{P}_{\bm{Y}_{k}} \bm{\zeta} )\notag\\
&\quad = \big(\vv^k \otimes ( \widehat{\rho}(\varphi^k) \vv^k
+ \widehat{\J}^k), \nabla \bm{P}_{\bm{Y}_{k}}\bm{ \zeta}\big)
-\big( 2\nu(\varphi^{k}) D\bm{v}^{k},D\bm{P}_{\bm{Y}_{k}}\bm{\zeta}\big)
\notag\\
&\qquad -\gamma \big(|\nabla\vv^k|^2\nabla\vv^k,\nabla \bm{P}_{\bm{Y}_{k}}\bm{\zeta}\big)
+ \big((\mu^{k} -\beta'(\varphi^k) \sigma^{k})\nabla \varphi^{k},\bm{P}_{\bm{Y}_{k}}\bm {\zeta}\big)
+ \frac12 \big(\widehat{R}^k\vv^k,\bm{P}_{\bm{Y}_{k}}\bm{\zeta}\big)
\notag\\
&\qquad + \frac12\big(\widehat{\rho}'(\varphi^k)(\bm{I}-\bm{P}_{Z_{k}}) \big((\bm{v}^k\cdot\nabla\varphi^k)-\div \big(m(\varphi^k)\nabla \mu^k\big)\big)\vv^k,\bm{P}_{\bm{Y}_{k}}\bm{\zeta}\big).
\label{rhovt-1}
\end{align}
The right-hand side of \eqref{rhovt-1} can be estimated as follows
\begin{align*}
&\big|\big(\vv^k \otimes ( \widehat{\rho}(\varphi^k) \vv^k + \widehat{\J}^k), \nabla \bm{P}_{\bm{Y}_{k}}\bm{ \zeta}\big)\big|
\\
&\quad \leq \|\widehat{\rho}(\varphi^k)\|_{L^\infty(\Omega)} \|\vv^k\|_{\bm{L}^4(\Omega)}^2\|\nabla \bm{P}_{\bm{Y}_{k}}\bm{ \zeta}\|
\notag\\
&\qquad + \|\widehat{\rho}'(\varphi^k) m(\varphi^k)\|_{L^\infty(\Omega)} \| \nabla \mu^k\|\|\vv^k\|_{\bm{L}^4(\Omega)} \|\nabla \bm{P}_{\bm{Y}_{k}}\bm{ \zeta}\|_{\bm{L}^4(\Omega)}
\\
&\quad \leq C\|\nabla \vv^k\| \|\bm{\zeta}\|_{\bm{H}^1(\Omega)}
+ C  \| \nabla \mu^k\|\|\nabla \vv^k\|^\frac{1}{2} \|\bm{\zeta}\|_{\bm{H}^2(\Omega)},
\end{align*}
%
\begin{align*}
\big|\big( 2\nu(\varphi^{k}) D\bm{v}^{k}, D\bm{P}_{\bm{Y}_{k}}\bm{\zeta}\big)\big|
&\leq 2\|\nu(\varphi^{k})\|_{L^\infty(\Omega)}\|\nabla \bm{v}^{k}\|\|\bm{P}_{\bm{Y}_{k}}\bm{\zeta}\|_{\bm{H}^1(\Omega)}
\\
&\leq C\|\nabla \bm{v}^{k}\|\|\bm{\zeta}\|_{\bm{H}^1(\Omega)},
\end{align*}
\begin{align*}
\gamma \big|\big(|\nabla\vv^k|^2\nabla\vv^k,\nabla \bm{P}_{\bm{Y}_{k}}\bm{\zeta}\big)\big|
& \leq C\gamma^\frac14\big(\gamma \|\nabla\vv^k\|_{\bm{L}^4(\Omega)}^4\big)^\frac34 \|\bm{P}_{\bm{Y}_{k}}\nabla \bm{\zeta}\|_{\bm{L}^4(\Omega)}
\\
&\leq C\gamma^\frac14\big(\gamma \|\nabla\vv^k\|_{\bm{L}^4(\Omega)}^4\big)^\frac34 \|\bm{\zeta}\|_{\bm{H}^2(\Omega)},
\end{align*}
\begin{align*}
&\big|\big((\mu^{k} -\beta'(\varphi^k) \sigma^{k})\nabla \varphi^{k},\bm{P}_{\bm{Y}_{k}}\bm {\zeta}\big)\big|
\\
&\quad \leq \|\mu^k\|_{L^4(\Omega)}\|\nabla \varphi^k\|\|\bm{P}_{\bm{Y}_{k}}\bm{\zeta}  \|_{\bm{L}^4(\Omega)}
+ \|\beta'(\varphi^k)\|_{L^\infty(\Omega)}\|\sigma^k\|\|\nabla \varphi^k\|\|\bm{P}_{\bm{Y}_{k}}\bm{\zeta}  \|_{\bm{L}^\infty(\Omega)}
\\
&\quad \leq C\|\mu^k\|_{H^1(\Omega)}\|\bm{\zeta} \|_{\bm{H}^1(\Omega)}
+ C \|\sigma^k\|\|\bm{\zeta}\|_{\bm{H}^2(\Omega)},
\end{align*}
\begin{align}
\left|\frac12 \big(\widehat{R}^k\vv^k,\bm{P}_{\bm{Y}_{k}}\bm{\zeta}\big)\right|
& \leq  \|m(\varphi^k) \widehat{\rho}''(\varphi^k)\|_{L^\infty(\Omega)} \|\nabla  \varphi^k\|_{\bm{L}^4(\Omega)}\|  \nabla \mu^k\|\|\vv^k\|_{\bm{L}^4(\Omega)}\|\bm{P}_{\bm{Y}_{k}}\bm{\zeta}  \|_{\bm{L}^\infty(\Omega)}
\notag \\
&\leq C\|\varphi^k\|_{H^2(\Omega)}^\frac12 \|  \nabla \mu^k\|\|\nabla \vv^k\|^\frac12
\|\bm{\zeta}  \|_{\bm{H}^2(\Omega)},
\label{es-Rhat}
\end{align}
\begin{align}
&\left|\frac12\big(\widehat{\rho}'(\varphi^k)(\bm{I}-\bm{P}_{Z_{k}}) \big((\bm{v}^k\cdot\nabla\varphi^k)-\div \big(m(\varphi^k)\nabla \mu^k\big)\big)\vv^k,\bm{P}_{\bm{Y}_{k}}\bm{\zeta}\big)\right|
\notag\\
&\quad \leq \frac12 \|\widehat{\rho}'(\varphi^k)\|_{L^\infty(\Omega)} \|\bm{v}^k\cdot\nabla\varphi^k \| \|\vv^k\| \|\bm{P}_{\bm{Y}_{k}}\bm{\zeta}  \|_{\bm{L}^\infty(\Omega)}
\notag\\
&\qquad + \frac12 \| \div \big(m(\varphi^k)\nabla \mu^k\big) \|_{(H^1(\Omega))'} \|\widehat{\rho}'(\varphi^k)\vv^k\cdot \bm{P}_{\bm{Y}_{k}}\bm{\zeta}  \|_{H^1(\Omega)}
\notag\\
&\quad \leq C  \|\bm{v}^k\|^\frac32\|\nabla \vv^k\|^\frac12 \|\nabla \varphi^k\|^\frac12\|\varphi^k\|_{H^2(\Omega)}^\frac12 \|\bm{P}_{\bm{Y}_{k}}\bm{\zeta}\|_{\bm{H}^2(\Omega)}
\notag\\
&\qquad +\|m(\varphi^k)\|_{L^\infty(\Omega)}\|\nabla \mu^k\|\|\widehat{\rho}'(\varphi^k)\|_{L^\infty(\Omega)}
\|\vv^k\|\| \bm{P}_{\bm{Y}_{k}}\bm{\zeta}  \|_{\bm{L}^\infty(\Omega)}
\notag\\
&\qquad +\|m(\varphi^k)\|_{L^\infty(\Omega)}\|\nabla \mu^k\|
\|\rho''(\varphi^k)\|_{L^\infty(\Omega)}\|\nabla \varphi^k\|_{\bm{L}^4(\Omega)}\|\vv^k\|_{\bm{L}^4(\Omega)}
\| \bm{P}_{\bm{Y}_{k}}\bm{\zeta}  \|_{\bm{L}^\infty(\Omega)}
\notag\\
&\qquad + \|m(\varphi^k)\|_{L^\infty(\Omega)}\|\nabla \mu^k\|
\|\rho'(\varphi^k)\|_{L^\infty(\Omega)} \|\nabla \vv^k\|
\| \bm{P}_{\bm{Y}_{k}}\bm{\zeta}  \|_{\bm{L}^\infty(\Omega)}
\notag\\
&\qquad + \|m(\varphi^k)\|_{L^\infty(\Omega)}\|\nabla \mu^k\|
\|\rho'(\varphi^k)\|_{L^\infty(\Omega)} \| \vv^k\|_{\bm{L}^4(\Omega)}
\|\nabla \bm{P}_{\bm{Y}_{k}}\bm{\zeta}  \|_{\bm{L}^4(\Omega)}
\notag\\
&\quad \leq C \|\nabla \vv^k\|^\frac12 \|\varphi^k\|_{H^2(\Omega)}^\frac12 \|\bm{\zeta}\|_{\bm{H}^2(\Omega)} + C \|\nabla \mu^k\|\| \bm{\zeta}  \|_{\bm{H}^2(\Omega)}
\notag\\
&\qquad + C \|\nabla \mu^k\|
\|\varphi^k\|_{H^2(\Omega)}^\frac12 \|\nabla \vv^k\|^\frac12
\|\bm{\zeta}  \|_{\bm{H}^2(\Omega)}
+ C\|\nabla \mu^k\|   \|\nabla \vv^k\| \|\bm{\zeta}  \|_{\bm{H}^2(\Omega)}
\notag\\
&\qquad + C\|\nabla \mu^k\| \| \nabla \vv^k\|^\frac12
\|\bm{\zeta}  \|_{\bm{H}^2(\Omega)}.
\label{es-proj}
\end{align}
Collecting the above estimates, we can deduce from \eqref{energy-b}, \eqref{s-L2-est} and \eqref{phi-L4H2} that
\be
 \sup_{t\geq 0}\| \bm{P}_{\bm{Y}_{k}}\partial_{t} (\widehat{\rho}(\varphi^k)\bm{v}^{k})\|_{ L^{\frac{8}{7}}(t,t+1;(D(\bm{S}))')}
 \le C,
 \label{vmt2d}
\ee
where $C>0$ depends on $\gamma$, $\mathcal{E}_0$, $\|\sigma_{0,n}\|$, $\Omega$, $T$,  but it is independent of $k$, $\epsilon$. In particular, the $p$-Laplacian regularization has been used to handle \eqref{es-proj}.

Finally, we estimate $\partial_t\sigma^k$. Multiplying \eqref{atest.4} by $\zeta\in H^1(\Omega)$ and integrating over $\Omega$, we have
\begin{equation*}
\begin{split}
\l \partial_t \sigma^k, \zeta \r_{(H^1(\Omega))',H^1(\Omega)}
&=
\int_{\Omega} \sigma^k \vv^k \cdot \nabla \zeta \, \d x - \int_{\Omega} \nabla \sigma^k \cdot \nabla \zeta \, \d x
-\int_{\Omega} \beta'(\varphi^k)\sigma^k \nabla \varphi^k \cdot \nabla \zeta \, \d x
\\
& \leq
\| \sigma^k\|_{L^4(\Omega)} \| \vv^k\|_{\bm{L}^4(\Omega)}  \| \nabla \zeta\|
+ \| \nabla \sigma^k\| \| \nabla \zeta\|\notag\\
&\quad  + \|\beta'(\varphi^k)\|_{L^\infty(\Omega)} \| \sigma^k\|_{L^4(\Omega)}  \|\nabla \varphi^k\|_{\bm{L}^4(\Omega)}
\| \nabla \zeta\|,
\end{split}
\end{equation*}
which combined with the Ladyzhenskaya inequality and
\eqref{energy-a}, \eqref{sigma-L2} yields
\begin{align}
\| \partial_t \sigma^k\|_{(H^1(\Omega))'}
& \leq \| \sigma^k\|_{L^4(\Omega)} \| \vv^k\|_{\bm{L}^4(\Omega)}
+ \| \nabla \sigma^k\|  + C\| \sigma^k\|_{L^4(\Omega)}  \|\nabla \varphi^k\|_{\bm{L}^4(\Omega)}
\notag \\
&\leq C \| \sigma^k\|^\frac12 \| \sigma^k\|_{H^1(\Omega)}^\frac12
\| \vv^k\|^\frac12 \| \nabla \vv^k\|^\frac12
 + \| \nabla \sigma^k\|
 \notag \\
&\quad +C \| \sigma^k\|^\frac12 \| \sigma^k\|_{H^1(\Omega)}^\frac12  \|\nabla \varphi^k\|^\frac12 \| \nabla\varphi^k\|_{\bm{H}^1(\Omega)}^\frac12
\notag \\
&\leq C  \big(\| \sigma^k\|_{H^1(\Omega)} + \| \nabla \vv^k\| +  \| \varphi^k\|_{H^2(\Omega)}\big).
\notag
\end{align}
Owing to \eqref{energy-b}, \eqref{phi-h-2} and \eqref{sigma-H1}, we can  conclude
\begin{equation}
\label{sigma_t}
\sup_{t\geq 0}\| \partial_t \sigma^k\|_{L^2(t,t+1; (H^1(\Omega))')}
\leq C,
\end{equation}
where $C>0$ depends on $\|\sigma_{0,n}\|$, $\mathcal{E}_0$, $\Omega$, $T$, but is independent of $k$, $\gamma$, $\epsilon$.

\subsection{Passage to the limit}
\label{limit-k}
We are ready to pass to the limit as $k \to \infty$ in the semi-Galerkin scheme \eqref{atest.1}--\eqref{atest.ini0} and establish the existence of a global weak solution $(\bm{v}^{\gamma,\epsilon,n},\varphi^{\gamma,\epsilon,n}, \mu^{\gamma,\epsilon,n}, \sigma^{\gamma,\epsilon,n})$ to the regularized problem $(\bm{S}_{\gamma,\epsilon,n})$. For simplicity of notation, we denote the solution  $(\bm{v}^{\gamma,\epsilon,n},\varphi^{\gamma,\epsilon,n}, \mu^{\gamma,\epsilon,n}, \sigma^{\gamma,\epsilon,n})$ by
$(\bm{v}^{\sharp},\varphi^{\sharp}, \mu^{\sharp}, \sigma^{\sharp})$.

\bp\label{p2}
Let the parameters $\gamma$, $\epsilon$, $n$ be fixed as in \eqref{param-1}. Suppose that (H1)--(H4) are satisfied and the initial data $(\bm{v}_0,\varphi_0,\sigma_0)$ are given as in Theorem \ref{WEAK-SOL}. The regularized problem \eqref{reg.sys}--\eqref{reg.ini} admits a global weak solution
$(\bm{v}^{\sharp},\varphi^{\sharp}, \mu^{\sharp}, \sigma^{\sharp})$
 in $\Omega\times [0,\infty)$ such that
\begin{align}
&\bm{v}^{\sharp} \in L^{\infty}(0,\infty;\bm{L}^2_{0,\sigma}(\Omega)) \cap L^{2}(0,\infty;\bm{H}^1_{0,\sigma}(\Omega)),
\notag\\
& \gamma^\frac14 \nabla \vv^{\sharp} \in L^4(0,\infty;\bm{L}^4(\Omega)),
\qquad \bm{P}\partial_{t} (\widehat{\rho}(\varphi^{\sharp})\bm{v}^{\sharp})\in  L^{\frac{8}{7}}_{\mathrm{uloc}}([0,\infty);(D(\bm{S}))'),
\notag \\
&\varphi^{\sharp} \in BC_\mathrm{w}([0,\infty);H^1(\Omega))\cap L^{4}_{\mathrm{uloc}}([0,\infty);H^2_{N}(\Omega)) \cap H^{1}_{\mathrm{uloc}}([0,\infty);(H^1(\Omega))'),
\notag \\
&\mu^{\sharp} \in   L^{2}_{\mathrm{uloc}}([0,\infty);H^1(\Omega)),\qquad \ \, \nabla\mu^{\sharp} \in   L^{2}(0,\infty;\bm{L}^2(\Omega)),
\notag\\
& \Psi_{\epsilon}(\varphi^{\sharp})\in L^\infty(0,\infty;L^1(\Omega)),\qquad \Psi_{\epsilon}'(\varphi^{\sharp})\in L^2_{\mathrm{uloc}}([0,\infty);L^2(\Omega)),
\notag \\
&\sigma^{\sharp} \in BC([0, \infty); L^{2}(\Omega))\cap L^2_{\mathrm{uloc}}([0,\infty);H^1(\Omega))\cap H^1_{\mathrm{uloc}}([0,\infty);(H^1(\Omega))'),
\notag \\
&
(\sigma^{\sharp})^\frac12 \in L^2_{\mathrm{uloc}}([0,\infty);H^1(\Omega)),
\notag \\
&\sigma^{\sharp}(x, t) \geq 0 \quad \text {a.e. in}\ \Omega\times(0,\infty).
\notag
\end{align}
%
and the following identities hold
\begin{subequations}
\begin{alignat}{3}
& \big\langle\partial_t  (\widehat{\rho}(\varphi^{\sharp}) \bm{ v}^{\sharp}),\bm{\zeta} \big\rangle_{(D(\bm{S}))',D(\bm{S})}
 - (\widehat{\rho}(\varphi^{\sharp}) \bm{v}^{\sharp}  \otimes\bm{ v}^{\sharp} ,D\bm{ \zeta})
 - \big( \bm{v}^{\sharp}\otimes \widehat{\J}^{\sharp}, \nabla \bm{\zeta}\big)
&& \notag \\
& \qquad
+\big(2\nu(\varphi^{\sharp} ) D\bm{v}^{\sharp} ,D\bm{ \zeta}\big)
+ \gamma\big(|\nabla \bm{v}^{\sharp}|^2\nabla \bm{v}^{\sharp}, \nabla \bm{ \zeta}\big) && \notag \\
& \quad=\big((\mu^{\sharp}
-\beta'(\varphi^{\sharp})\sigma^{\sharp} ) \nabla \varphi^{\sharp} , \bm {\zeta}\big)
+ \frac12\big(\widehat{R}^{\sharp} \vv^{\sharp} ,\bm{\zeta} \big),
\quad&& \textrm{a.e. in }(0,\infty), &\label{test3.c-reg}
\\
& \quad \text{with}\quad \widehat{\J}^{\sharp}
= -\widehat{\rho}'(\varphi^{\sharp}) m(\varphi^{\sharp}) \nabla \mu^{\sharp},\quad
\widehat{R}^{\sharp}
= -m(\varphi^{\sharp})\nabla \widehat{\rho}'(\varphi^{\sharp})\cdot \nabla \mu^{\sharp},
\notag \\
& \langle \partial_t \varphi^{\sharp} ,\xi \rangle_{(H^1(\Omega))',H^1(\Omega)}
-\big(\varphi^{\sharp}\bm{v}^{\sharp},  \nabla\xi\big)
+ \big(m(\varphi^{\sharp} )\nabla \mu^{\sharp} ,\nabla \xi\big) =0,\quad && \textrm{a.e. in }(0,\infty),&\label{test1.a-reg}
\\
& \mu^{\sharp} = - \Delta \varphi^{\sharp} +\Psi_{\epsilon}'(\varphi^{\sharp})
+\beta'(\varphi^{\sharp}) \sigma^{\sharp}, \quad && \textrm{a.e. in }\Omega\times(0,\infty),&\label{test4.d-reg}
\\
& \langle \partial_t\sigma^{\sharp},\xi \rangle_{(H^1(\Omega))',H^1(\Omega)}
- (\sigma^{\sharp}\bm{v}^{\sharp},\nabla \xi)
+ (\nabla\sigma^{\sharp} , \nabla \xi)
&& \notag\\
&\quad = - (\beta'(\varphi^{\sharp})\sigma^{\sharp} \nabla \varphi^{\sharp},\nabla \xi),
&& \textrm{a.e. in }(0,\infty)& \label{test2.b-reg}
\end{alignat}
\end{subequations}
 for all test functions $\bm {\zeta} \in D(\bm{S})$, $\xi\in H^1(\Omega)$. The initial conditions are fulfilled
\begin{align}
&\bm{v}^{\sharp} |_{t=0}=\bm{v}_{0},\quad \varphi^{\sharp} |_{t=0}=\varphi_{0,n},\quad
\sigma^{\sharp} |_{t=0}=\sigma_{0,n},
\quad\textrm{a.e. in } \Omega.	
\notag
\end{align}
Moreover, the following energy inequality holds
\begin{align}
&  \widehat{\mathcal{E}}^\sharp(t)
+ \int_0^t \widehat{\mathcal{D}}^\sharp(\tau)\,\d \tau\leq  \widehat{\mathcal{E}}^\sharp(0),
\label{menergy-b}
\end{align}
for almost all $t>0$, where
\begin{align*}
\widehat{\mathcal{E}}^\sharp(t)
&=\int_{\Omega}\left( \frac{1}{2}\widehat{\rho}(\varphi^\sharp)|\boldsymbol{v}^\sharp|^{2}
+\frac{1}{2}|\nabla \varphi^\sharp|^{2}
+ \Psi_\epsilon(\varphi^\sharp) +\sigma^\sharp(\ln \sigma^\sharp-1)
+ \beta(\varphi^\sharp) \sigma^\sharp\right)(t) \, \mathrm{d} x,
\\
\widehat{\mathcal{D}}^\sharp(t)
& =\int_{\Omega} \left(2 \nu(\varphi^\sharp)|D \boldsymbol{v}^\sharp|^{2}
+ \gamma|\nabla \vv^\sharp|^4
+ m(\varphi^\sharp)|\nabla \mu^\sharp|^{2}
+ \big| 2\nabla \sqrt{\sigma^\sharp} + \sqrt{\sigma^\sharp} \nabla \beta(\varphi^\sharp) \big|^2
\right)(t)\, \mathrm{d} x.
\end{align*}
\ep
\begin{proof}
We have shown that for every integer $k\geq \widehat{k}$, the semi-Galerkin scheme \eqref{atest.1}--\eqref{atest.ini0} admits a unique global solution $(\bm{v}^{k},\varphi^{k},\mu^{k},\sigma^{k})$ in $\Omega \times [0,\infty)$ with estimates that are independent of the approximation parameter $k$. In particular, the uniform estimates \eqref{energy-a}, \eqref{energy-b}, \eqref{mu},
 \eqref{Psip0b}, \eqref{sigma-L2}, \eqref{sigma-H1}, \eqref{phi-L4H2}, \eqref{phimt}, \eqref{vmt2d}, \eqref{sigma_t} are sufficient for us to apply theorems of weak compactness and the Aubin--Lions--Simon lemma (see \cite{si87}) to extract a suitable subsequence that approaches a limit
$(\bm{v}^{\sharp},\varphi^{\sharp}, \mu^{\sharp}, \sigma^{\sharp})$
 in corresponding topologies as $k\to \infty$ in $[0,T]$ for arbitrary fixed $T>0$. With sufficient information on $\sigma^k$ (based on the assumption that $\sigma_{0,n}\in L^2(\Omega)$), the convergence in \eqref{atest.4} is standard, while the convergence for the Navier--Stokes/Cahn--Hilliard part \eqref{atest.1}--\eqref{atest.1b} follows a similar argument for \cite[Lemma 3]{Frigeri2016} (see a corrigendum in \cite[Remark 4.3]{Frigeri2021}).
 In particular, the contribution of the last term on the right-hand side of \eqref{atest.1} converges to zero, and the convergence of $\gamma |\nabla \bm{v}^{k}|^2\nabla \bm{v}^{k}$ is a consequence of the well-known Minty's trick for monotone operators.
 Due to the lack of uniqueness for $(\bm{v}^{\sharp},\varphi^{\sharp}, \mu^{\sharp}, \sigma^{\sharp})$, in order to construct a global weak solution on the whole interval $[0,\infty)$, we shall combine the standard compactness argument with a diagonal extraction process as in \cite[Chapter V, Section 1.3.6]{Bo13} for the three-dimensional Navier--Stokes system.

 Integrating \eqref{menergy} over $[0,t]$ for any $t>0$ gives the energy equality
 \begin{align}
 \widehat{\mathcal{E}}^k(t)+ \int_0^t \widehat{\mathcal{D}}^k(\tau)\,\d \tau=\widehat{\mathcal{E}}^k(0),\quad \forall\, t>0.
 \label{app-BEL2}
 \end{align}
 %
Using the weak/strong convergence results, weak lower-semicontinuity of norms, and Fatou's lemma, we can derive the energy inequality \eqref{menergy-b} from \eqref{app-BEL2} following an argument similar to that in \cite[Section 4.1]{Frigeri2016}. Here, we present only the convergence of $\int_\Omega \sigma^k|\nabla(\ln \sigma^k+ \beta(\varphi^k))|^{2}\, \mathrm{d} x$. The convergence results mentioned in the following should be understood in the sense of a subsequence.
In view of \eqref{dis-1}, we write
\begin{align*}
& \int_\Omega  \sigma^k \big|\nabla(\ln \sigma^k+ \beta(\varphi^k))\big|^{2}\, \mathrm{d} x
= \int_\Omega \left| 2\nabla \sqrt{\sigma^k} + \sqrt{\sigma^k} \nabla \beta(\varphi^k) \right|^2 \, \mathrm{d}x
\notag\\
&\quad =
\int_\Omega 4\big|\nabla \sqrt{\sigma^k}\big|^2 \, \d x
+ \int_\Omega |\beta'(\varphi^k)|^2 \sigma^k |\nabla \varphi^k|^2 \, \d x
+ \int_{\Omega} 2 \nabla \sigma^k \cdot \nabla \beta(\varphi^k) \, \d x.
\end{align*}
Since $\sigma^k$ is uniformly bounded in $L^2(0,T;H^1(\Omega))$ and $\partial_t \sigma^k$ is uniformly bounded in $L^2(0,T;(H^1(\Omega))')$ for any $T>0$, we infer from the Aubin--Lions lemma that
\begin{align*}
&\sigma^k \rightarrow \sigma^{\sharp} \quad \text{strongly in } L^2(0,T; L^2(\Omega)) \text{ and  a.e. in } \Omega \times (0,T).
\end{align*}
The pointwise convergence of $\sigma^k$ combined with \eqref{dis-4} yields
\begin{align*}
&\sqrt{\sigma^k} \rightharpoonup \sqrt{\sigma^{\sharp}}\quad \text{weakly in } L^2(0,T; H^1(\Omega)),
\end{align*}
which enables us to conclude (cf. \cite[Proof of Lemma 4.1]{W2016})
\begin{align}
\int_0^T \!\! \int_\Omega  \big|\nabla \sqrt{\sigma^{\sharp}} \big|^2  \, \d x \d t
\leq \liminf_{k \to \infty}
\int_0^T \!\! \int_\Omega \big|\nabla \sqrt{\sigma^{k}} \big|^2 \, \d x \d t.
\label{Dis-Sig-1}
\end{align}
Since $\varphi^k$ is uniformly bounded in $L^4(0,T; H^2(\Omega))$ and $\partial_t\varphi^k$ is uniformly bounded in $L^2(0,T; (H^1(\Omega))')$, we infer from the Aubin--Lions lemma that
$$
\varphi^k \to \varphi^{\sharp} \quad \text{strongly in } L^4(0,T; W^{1,4}(\Omega))\text{ and  a.e. in } \Omega \times (0,T),
$$
and thus
$$
|\nabla \varphi^k|^2 \to |\nabla \varphi^{\sharp}|^2 \quad \text{strongly in } L^2(0,T; L^2(\Omega)).
$$
Since $\beta'$, $\beta''$ are bounded, we can further deduce that
$$
\beta'(\varphi^k) \rightarrow \beta'(\varphi^{\sharp}) \quad \text{strongly in } L^4(0,T; L^4(\Omega)),
$$
as well as
$$
|\beta'( \varphi^k)|^2 \to |\beta'( \varphi^{\sharp})|^2 \quad \text{strongly in } L^4(0,T; L^4(\Omega)).
$$
These observations combined with the strong convergence of $\sigma^k$ in $L^2(0,T; L^2(\Omega))$ and the boundedness of $\sigma^{\sharp}$ in $L^4(0,T;L^4(\Omega))$ yield
\begin{align}
&\left|\int_0^T\!\! \int_\Omega |\beta'(\varphi^k)|^2 \sigma^k |\nabla \varphi^k|^2 \, \d x \d t
- \int_0^T \!\! \int_\Omega |\beta'(\varphi^{\sharp})|^2 \sigma^{\sharp} |\nabla \varphi^{\sharp}|^2 \, \d x\d t\right|
 \notag \\
&\quad \leq \left|\int_0^T\!\! \int_\Omega |\beta'(\varphi^k)|^2 (\sigma^k-\sigma^{\sharp} )  |\nabla \varphi^k|^2 \, \d x \d t \right|
+
\left|\int_0^T\!\! \int_\Omega (|\beta'(\varphi^k)|^2 - |\beta'(\varphi^{\sharp})|^2) \sigma^{\sharp} |\nabla \varphi^k|^2 \, \d x \d t\right|
\notag \\
&\qquad
+ \left| \int_0^T \!\! \int_\Omega |\beta'(\varphi^{\sharp})|^2 \sigma^{\sharp} (|\nabla \varphi^k|^2-|\nabla \varphi^{\sharp}|^2) \, \d x\d t\right|
\to 0.
\label{Dis-Sig-2}
\end{align}
On the other hand, from the weak convergence of $\sigma^k$ in $L^2(0,T; H^1(\Omega))$ we infer that
\begin{align}
&\left|\int_0^T\!\! \int_{\Omega} 2 \nabla \sigma^k \cdot \nabla \beta(\varphi^k) \, \d x\d t- \int_0^T\!\! \int_{\Omega} 2 \nabla \sigma^{\sharp} \cdot \nabla \beta(\varphi^{\sharp}) \, \d x\d t\right|
\notag \\
&\quad \leq \left|\int_0^T\!\! \int_{\Omega} 2 (\nabla \sigma^k -\nabla \sigma^{\sharp} )\cdot \nabla\varphi^{\sharp}  \beta'(\varphi^{\sharp}) \, \d x\d t \right|
+ \left|\int_0^T\!\! \int_{\Omega} 2 \nabla \sigma^k \cdot \nabla \varphi^k(\beta'(\varphi^k)- \beta'(\varphi^{\sharp}))  \, \d x\d t\right|
\notag \\
&\qquad
+ \left|\int_0^T\!\! \int_{\Omega} 2 \nabla \sigma^k \cdot (\nabla \varphi^k- \nabla \varphi^{\sharp}) \beta'(\varphi^{\sharp}) \, \d x\d t\right|
\to 0.
\label{Dis-Sig-3}
\end{align}
From \eqref{Dis-Sig-1}--\eqref{Dis-Sig-3}, we can conclude that
$$
\int_0^T \!\! \int_\Omega \left| 2\nabla \sqrt{\sigma^{\sharp}} + \sqrt{\sigma^{\sharp}} \nabla \beta(\varphi^{\sharp}) \right|^2 \, \mathrm{d}x\d t
\leq
\liminf_{k \to \infty} \int_0^T\!\!  \int_\Omega \left| 2\nabla \sqrt{\sigma^k} + \sqrt{\sigma^k} \nabla \beta(\varphi^k) \right|^2 \, \mathrm{d}x\d t,
$$
for any $T>0$. The remainder of the limiting procedure is standard, and we omit the details.

The proof of Proposition \ref{p2} is complete.
\end{proof}

\subsection{Completion of the Proof of Theorem \ref{WEAK-SOL}}
\label{proof-thm1}
We shall pass to the limit as $\gamma, \epsilon \to 0$ and $n\to \infty$ simultaneously. For this purpose, we take $\gamma=\epsilon=\frac{1}{n}$ (for sufficiently large integers $n$) and simply denote the approximate solutions $(\bm{v}^{\gamma,\epsilon,n},\varphi^{\gamma,\epsilon,n}, \mu^{\gamma,\epsilon,n}, \sigma^{\gamma,\epsilon,n})$ obtained in Proposition \ref{p2} by $(\bm{v}^{n},\varphi^{n},\mu^{n}, \sigma^{n})$.
\medskip

\textbf{Part 1.} Let us first consider the case with $\sigma_0\in L^2(\Omega)$.

For $\sigma_0\in L^2(\Omega)$, we can find a family of approximations $\{\sigma_{0,n}\}_{n\in \mathbb{Z}^+}$, with the following properties
$$
\sigma_{0,n} \in C^{\infty}_0(\Omega),  \quad \sigma_{0,n}\ge 0 \ \text { in } \Omega,\quad \sigma_{0,n}\not\equiv 0,\quad
\sigma_{0,n} \rightarrow \sigma_0\ \text { in }\ L^2(\Omega) \ \text { as } n\to \infty.
$$
It easily follows that (for sufficiently large $n$)
\begin{align}
-\mathrm{e}^{-1}|\Omega|\leq \int_\Omega \sigma_{0,n}\ln \sigma_{0,n}\,\d x\leq \|\sigma_{0,n}\|^2 \leq \|\sigma_{0}\|^2 +1.
\label{bd-ini-sigL1}
\end{align}
Hence, from the energy inequality \eqref{menergy-b} and using arguments similar to those in Section \ref{ue}, we can maintain uniform estimates \eqref{energy-a}, \eqref{energy-b}, \eqref{mu}, \eqref{sigma-L2}, \eqref{sigma-H1}, \eqref{phi-L4H2}, \eqref{phimt}, \eqref{vmt2d}, \eqref{sigma_t} for $(\bm{v}^{n},\varphi^{n},\mu^{n}, \sigma^{n})$, except \eqref{Psip0b} (depending on $\epsilon$). To recover \eqref{vmt2d}, we note that at this stage, since we have already taken the limit as $k\to \infty$, it is unnecessary to handle a highly nonlinear term like in \eqref{es-proj}, that is,
$$
\frac12\big(\widehat{\rho}'(\varphi^k)(\bm{I}-\bm{P}_{Z_{k}}) \big((\bm{v}^k\cdot\nabla\varphi^k)-\div \big(m(\varphi^k)\nabla \mu^k\big)\big)\vv^k,\bm{P}_{\bm{Y}_{k}}\bm{\zeta}\big),
$$
for which the regularization involving the $p$-Laplacian term was essentially used (depending on $\gamma$). On the other hand, the uniform estimate for $\|\Psi_\epsilon'(\varphi^n)\|_{L^2(0,T; L^2(\Omega))}$ follows a direct comparison in \eqref{test4.d-reg}. The remaining part of the limiting procedure as $n\to \infty$ can be performed analogously as in \cite[Sections 4.2, 4.3]{Frigeri2016} (with a corrigendum in \cite[Remark 4.3]{Frigeri2021}) for the Navier--Stokes part, in \cite{DD95,MT16} for the Cahn--Hilliard part, and in \cite{Lan2016} for the $\sigma^n$-equation. Denote the limit functions by $(\vv,\varphi,\mu,\sigma)$. In particular, we have
 $$
 \Psi_{1/n}(\varphi^n)\to \Psi(\varphi),\ \ \Psi'_{1/n}(\varphi^n)\to \Psi'(\varphi)\quad \text{a.e. in}\ \Omega\times(0,T),
 $$
 and
 $$ \varphi\in L^\infty(\Omega\times(0,T)),\qquad
 |\varphi(x,t)|<1\quad  \text{for a.a.}\ (x,t)\in  \Omega\times (0,T),
 $$
 for any $T>0$, where the latter is a consequence of the singular nature of $\Psi$ at $\pm 1$. This physical bound for $\varphi$ combined the definition of $\widehat{\rho}$ yields that
 $$
 \widehat{\rho}(\varphi)=\rho(\varphi).
 $$
 As a result, the artificial term $\widehat{R}$ in the modified momentum balance equation (cf. \eqref{reg.1}, \eqref{reg.1a}) vanishes after passing to the limit as $n\to \infty$, that is, $\widehat{R}=0$. In addition, from the estimate
 \begin{align}
\left(\frac{1}{n}\right)^\frac14\|\nabla \bm{v}^n\|_{L^{4}(0, T ; \bm{L}^4(\Omega))} \leq C,
\label{p4-estimate}
 \end{align}
 where $C>0$ is independent of $n$, we see that for all $\bm{\zeta}\in L^4(0,T;\bm{L}^4(\Omega))$, it holds
 $$
 \left|\int_0^T\left(\frac{1}{n}\right)\big(|\nabla \bm{v}^n|^2\nabla \bm{v}^n,\bm{\zeta}\big)\,\d t\right|
 \leq C\left(\frac{1}{n}\right)^\frac14\|\bm{\zeta}\|_{L^4(0,T;\bm{L}^4(\Omega))}\to 0,\quad \text{as}\ n\to \infty.
 $$
 Then, a direct comparison in \eqref{uu-weak2} further yields $\partial_t\bm{P}(\rho(\varphi)\vv)\in L_{\uloc}^{s}([0,\infty);(D(\bm{S}))')$, for any $s\in [1,2)$. Next, using the facts $\mu\in L^2_{\uloc}([0,\infty);H^1(\Omega))$ and $\beta'(\varphi)\sigma\in L^2_{\uloc}([0,\infty);H^1(\Omega))$, we can apply the classical result for the Cahn--Hilliard equation with a singular potential to conclude (see, e.g., \cite{GiGrWu2018})
\begin{align}
\sup_{t\geq 0}\|\varphi\|_{L^2(t,t+1;W^{2,q}(\Omega))}
+\sup_{t\geq 0}\|\Psi'(\varphi)\|_{L^2(t,t+1;L^q(\Omega))}
\leq C,\quad \forall\, q\in [2,\infty).
\label{es-vhi-W2q}
\end{align}
The construction of a global weak solution on the whole interval $[0,\infty)$ follows a diagonal extraction process as in \cite{Bo13} due to the lack of uniqueness.
Finally, using the weak/strong convergence results, weak lower-semicontinuity of norms and Fatou's lemma, we can obtain the energy inequality \eqref{menergy-weak-fi} from \eqref{menergy-b}, cf. the proof of Proposition \ref{p2}.

The proof of Theorem \ref{WEAK-SOL}-(2) is complete.
 \medskip

\textbf{Part 2.} Now we study the more involved case with only $\sigma_0\ln \sigma_0\in L^1(\Omega)$.

The uniform estimates obtained in Step 5 of Section \ref{ue} are no longer valid for $(\varphi^{n},\sigma^{n})$. Thus, we lose the estimates corresponding to \eqref{sigma-L2}, \eqref{sigma-H1}, \eqref{phi-L4H2}, and the estimates of time derivatives as in \eqref{vmt2d}, \eqref{sigma_t}.

On the other hand, an examination of the argument in Step 4 of Section \ref{ue} allows us to maintain some weaker estimates. For instance, we still have
\begin{align}
 \sup_{t\geq 0}\|\sigma^n\|_{L^2(t,t+1;L^2(\Omega))}\leq C, \quad
 \sup_{t\geq 0}\|(\sigma^n)^\frac12\|_{L^2(t,t+1;H^1(\Omega))}\leq C, \quad
 \sup_{t\geq 0}\| \varphi^n\|_{L^2(t,t+1; H^2(\Omega))}\leq C,
 \label{lown-phi-sig}
\end{align}
which are uniform with respect to $n$. Then, a comparison in \eqref{test3.c-reg} yields (cf. the treatment for \eqref{rhovt-1})
\be
 \sup_{t\geq 0}\| \partial_{t} \bm{P}(\widehat{\rho}(\varphi^n)\bm{v}^{n})\|_{ L^{1}(t,t+1;(D(\bm{S}))')}
 \le C.
 \label{vmt2dn}
\ee
Here, the main modification is due to \eqref{es-Rhat} (keeping in mind that the term as in \eqref{es-proj} simply vanishes after taking $k\to \infty$), because we can only use \eqref{phi-h-2} for $\varphi^n$ instead of \eqref{phi-L4H2}. On the other hand, since
\begin{align*}
\|\nabla (\widehat{\rho}(\varphi^n)\vv^n)\|\leq  \|\widehat{\rho}(\varphi^n)\|_{L^\infty(\Omega)}\|\nabla \vv^n\| +
\|\widehat{\rho}'(\varphi^n)\|_{L^\infty(\Omega)}\|\nabla \varphi^n\|_{\bm{L}^4(\Omega)} \|\vv^n\|_{\bm{L}^4(\Omega)},
\end{align*}
we can conclude that $\widehat{\rho}(\varphi^n)\vv^n$ and thus $\bm{P}(\widehat{\rho}(\varphi^n)\vv^n)$ are uniformly bounded in $L^2(t,t+1;\bm{H}^1(\Omega))$ for all $t\geq 0$. This fact, combined with \eqref{vmt2dn} and the Aubin--Lions--Simon lemma, is sufficient for the strong compactness of $\bm{P}(\widehat{\rho}(\varphi^n)\bm{v}^{n})$ and then $\bm{v}^n$ (cf. \cite[Section 4.1]{Frigeri2021}). For the convenience of the readers, we sketch the proof here. First, applying the Aubin--Lions--Simon lemma, we get
\begin{equation*}
    \bm{P}(\widehat{\rho}(\varphi^n)\bm{v}^{n})\to \bm{w} \quad \text{strongly in} \ \ L^2(0,T;\bm{L}^2(\Omega)),
\end{equation*}
for any fixed $T>0$, where $\bm{w}\in L^2(0,T;\bm{H}^1_{0,\sigma}(\Omega))$. Next, from the easy facts
\begin{align*}
& \bm{v}^n \rightharpoonup \bm{v}\quad \text{weakly in}\ \ L^4(0,T;\bm{L}^4(\Omega)),\\
& \varphi^n\to \varphi \quad \text{strongly in}\ \ L^2(0,T;H^1(\Omega))\ \ \text{and a.e. in } \Omega\times (0,T),\\
& \widehat{\rho}(\varphi^n) \to \widehat{\rho}(\varphi) \quad \text{strongly in}\ \ L^4(0,T;L^4(\Omega)),
\end{align*}
we can verify that
\begin{align*}
    & \bm{P}(\widehat{\rho}(\varphi^n)\bm{v}^{n})\rightharpoonup \bm{P}(\widehat{\rho}(\varphi)\bm{v}) \quad \text{weakly in} \ \ L^2(0,T;\bm{L}^2(\Omega)),\\
    & \sqrt{\widehat{\rho}(\varphi^n)}\bm{v}^{n}\rightharpoonup  \sqrt{\widehat{\rho}(\varphi)}\bm{v} \quad \text{weakly in} \ \ L^2(0,T;\bm{L}^2(\Omega)).
\end{align*}
Hence, $\bm{w}=\bm{P}(\widehat{\rho}(\varphi)\bm{v})$ holds due to the uniqueness of the limit. Then using the same argument as in \cite[Section 5.1]{ADG2013}, we find
\begin{align*}
\int_0^T\int_\Omega \widehat{\rho}(\varphi^n)|\bm{v}^{n}|^2\,\mathrm{d}x\mathrm{d}t
& =\int_0^T\int_\Omega
\bm{P}(\widehat{\rho}(\varphi^n)\bm{v}^{n})\cdot \bm{v}^{n}\,\mathrm{d}x\mathrm{d}t
\\
& \to \int_0^T\int_\Omega
\bm{P}(\widehat{\rho}(\varphi)\bm{v})\cdot \bm{v}\,\mathrm{d}x\mathrm{d}t = \int_0^T\int_\Omega \widehat{\rho}(\varphi)|\bm{v}|^2\,\mathrm{d}x\mathrm{d}t,
\end{align*}
which combined with the uniform convexity of   $L^2(0,T;\bm{L}^2(\Omega))$
yields
$$
\sqrt{\widehat{\rho}(\varphi^n)}\bm{v}^{n}\to  \sqrt{\widehat{\rho}(\varphi)}\bm{v} \quad \text{strongly in} \ \ L^2(0,T;\bm{L}^2(\Omega)).
$$
Combining the convergence results above and the uniform positivity property of $\widehat{\rho}(\varphi^n)$ (see \eqref{apprho2}), we can further conclude
$$
\bm{v}^{n}= \frac{1}{\sqrt{\widehat{\rho}(\varphi^n)}} \Big(\sqrt{\widehat{\rho}(\varphi^n)}\bm{v}^{n}\Big) \to \frac{1}{\sqrt{\widehat{\rho}(\varphi)}} \Big(\sqrt{\widehat{\rho}(\varphi)}\bm{v}\Big)=\bm{v}\quad \text{strongly in}\ \ L^2(0,T;\bm{L}^2(\Omega)).
$$
We also note that \eqref{vmt2dn} is not enough to guarantee a weak limit for $\partial_{t} \bm{P}(\widehat{\rho}(\varphi^n)\bm{v}^{n})$. This motivates a weaker formulation of the momentum balance equation for the fluid velocity field $\vv$ (see \eqref{uu-weak}).

In the following, we say more words about the convergence involving $\sigma^n$. A comparison in \eqref{test2.b-reg} yields
\begin{align}
 \|\partial_{t}\sigma^{n}\|_{(H^2_N(\Omega))'}
& \leq C \|\bm{v}^n\|_{\bm{L}^4(\Omega)} \|\sigma^n\|+ C\|\sigma^n\|
+ C\|\beta'(\varphi^n)\|_{L^\infty(\Omega)} \| \sigma^n\| \| \nabla \varphi^n\|_{\bm{L}^4(\Omega)}
\notag \\
& \leq C \big(1+ \|\bm{v}^n\|^\frac12 \|\bm{v}^n\|_{\bm{H}^1(\Omega)}^\frac12+ \|\nabla \varphi^n\|^\frac12 \|\nabla \varphi^n\|_{\bm{H}^1(\Omega)}^\frac12 \big) \|\sigma^n\|
\notag
\end{align}
which implies
\begin{align}
 \sup_{t\geq 0}\|\partial_{t}\sigma^{n}\|_{L^\frac{4}{3}(t,t+1;(H^2_N(\Omega))')}\leq C,
 \label{sigmatn}
\end{align}
where $C>0$ is independent of $n$. To show the strong convergence of $\sigma^n$, we apply a compactness theorem due to Dubinskii (see, e.g., \cite{Dubin1965,BS2012}):
\begin{lemma}\label{Dubin}
Let $T>0$. Suppose that $B_0$ is a semi-normed set
that is compactly embedded into a Banach space $B$, which is continuously embedded into
another Banach space $B_1$. Then, for any $1\leq p, q<\infty$, the following embedding is compact:
$$
\big\{f\ \big|\ f\in L^p(0,T;B_0),\ \partial_t f\in L^q(0,T;B_1)\big\}\ \hookrightarrow \ L^p (0,T; B).
$$
\end{lemma}
Define
$$
B=L^1(\Omega),\quad B_0=\big\{f\in B\ \big|\ f\geq 0 \ \  \text{a.e. in}\ \Omega,\ f^\frac12\in H^1(\Omega)\big\},\quad B_1=(H_N^2(\Omega))',
$$
where $B_0$ is a semi-normed space with the semi-norm
$$
[f]_{B_0}=\|f\|_{L^1(\Omega)}+\|\nabla f^\frac12\|^2.
$$
The continuity of the embedding $B\hookrightarrow B_1$ is obvious, and the compactness of the embedding $B_0\hookrightarrow B$ has been verified in \cite{BLS2017}. It follows from \eqref{lown-phi-sig} and \eqref{sigmatn} that
$$
\sigma^n\ \text{is uniformly bounded in}\  \ \big\{f\ \big|\ f\in L^1(0,T;B_0),\ \partial_t f\in L^\frac{4}{3}(0,T;B_1)\big\}.
$$
Thus, we can apply Lemma \ref{Dubin} to conclude
$$
\sigma^n \to \sigma \quad\text{strongly in}\ L^1(0,T; L^1(\Omega)) \cong L^1(\Omega\times(0,T)),
$$
as $n\to \infty$ (up to a subsequence). This yields
$\sigma^n \to \sigma$ almost everywhere in $\Omega\times(0,T)$. Using \eqref{lown-phi-sig} and interpolation, we also have
\begin{align}
\sigma^n \to \sigma \quad\text{strongly in}\ L^{2-r}(\Omega\times(0,T)),\quad \forall\,r\in (0,1).
\label{conv-sig-Lr}
\end{align}
In addition, from the facts $\sigma\ln \sigma \in  L^\infty(0,\infty; L^1(\Omega))$, $\partial_t \sigma\in L^\frac{4}{3}_{\mathrm{uloc}}([0,\infty);(H^2_N(\Omega))')$, we can apply the argument in \cite[Section 12]{BLS2017} to obtain $\sigma\in BC_\mathrm{w}([0,\infty);L^1(\Omega))$.

Keeping these modifications in mind, we can take the limit as $n\to \infty$ to establish the existence of a global finite energy solution $(\vv,\varphi, \mu, \sigma)$ on $[0,\infty)$ again through a diagonal extraction process as in \cite{Bo13}. Finally, using the weak/strong convergence results, weak lower-semicontinuity of norms, and Fatou's lemma, we can recover the energy inequality \eqref{menergy-weak-fi} from \eqref{menergy-b}.
However, a different treatment for the dissipation term $\int_0^t\!\int_\Omega \big| 2\nabla \sqrt{\sigma^n} + \sqrt{\sigma^n} \nabla \beta(\varphi^n) \big|^2 \,\d x\d \tau$ is required here, compared to the proof of Proposition \ref{p2}.
For any $T>0$, we infer from \eqref{menergy-b} that $\int_0^T\!\int_\Omega \big| 2\nabla \sqrt{\sigma^n} + \sqrt{\sigma^n} \nabla \beta(\varphi^n) \big|^2 \,\d x\d t$ is uniformly bounded with respect to $n$. This implies
$$
2\nabla \sqrt{\sigma^n} + \sqrt{\sigma^n} \nabla \beta(\varphi^n) \rightharpoonup \bm{g}
\quad \text{weakly in}\ L^2(0,T;\bm{L}^2(\Omega)),
$$
for some $\bm{g}\in L^2(0,T;\bm{L}^2(\Omega))$, and as a consequence,
\begin{align}
\int_0^T\!\!\int_\Omega |\bm{g}|^2\,\d x\d t\leq
\liminf_{n\to \infty} \int_0^T \!\!\int_\Omega \big| 2\nabla \sqrt{\sigma^n} + \sqrt{\sigma^n} \nabla \beta(\varphi^n) \big|^2 \,\d x\d t.
\label{inf-sigma-n}
\end{align}
Next, we observe that
\begin{align*}
\int_0^T\!\!\int_\Omega \big|\sqrt{\sigma^n}-\sqrt{\sigma}\big|^3\,\d x\d t &\leq \int_0^T\!\!\int_\Omega |\sigma^n- \sigma |\big|\sqrt{\sigma^n}-\sqrt{\sigma}\big| \,\d x\d t
\\
&\leq \left(\int_0^T\!\!\int_\Omega |\sigma^n- \sigma |^\frac32\,\d x\d t\right)^\frac23\left(\int_0^T\!\!\int_\Omega \big|\sqrt{\sigma^n}-\sqrt{\sigma}\big|^3\,\d x\d t\right)^\frac13.
\end{align*}
This together with \eqref{conv-sig-Lr} (taking $r=1/2$) yields
\begin{align}
\sqrt{\sigma^n}\to \sqrt{\sigma}\quad \text{strongly in } L^3(0,T;L^3(\Omega)).
\label{conv-sig-sL3}
\end{align}
In addition, from the strong convergence $\varphi^n\to \varphi$ in $L^2(0,T;H^1(\Omega))\cap L^4(0,T;L^4(\Omega))$, (H4) on $\beta$ and the boundedness of $\nabla \varphi^n$ in $L^4(0,T;\bm{L}^4(\Omega))$, we can deduce that
\begin{align*}
&\int_0^T\!\!\int_\Omega |\nabla \beta(\varphi^n)-\nabla \beta(\varphi)|^2\,\d x\d t
\\
&\quad \leq 2\int_0^T\!\!\int_\Omega |\beta'(\varphi^n)-\beta'(\varphi)|^2|\nabla \varphi^n|^2\,\d x\d t
+ 2\int_0^T\!\!\int_\Omega |\beta'(\varphi)|^2|\nabla \varphi^n-\nabla \varphi|^2\,\d x\d t
\to 0.
\end{align*}
Thus, we get
\begin{align*}
\sqrt{\sigma^n} \nabla \beta(\varphi^n) \to \sqrt{\sigma} \nabla \beta(\varphi) \quad
\text{strongly in } L^\frac65(0,T;L^\frac65(\Omega)).
\end{align*}
On the other hand, it follows from \eqref{lown-phi-sig} and \eqref{conv-sig-sL3} that
$$
\sqrt{\sigma^n} \rightharpoonup \sqrt{\sigma}
\quad \text{weakly in } L^2(0,T;H^1(\Omega)).
$$
As a consequence, we can identify $\bm{g}= 2\nabla \sqrt{\sigma} + \sqrt{\sigma} \nabla \beta(\varphi)$,
which combined with \eqref{inf-sigma-n} yields
\begin{align}
\int_0^T\!\!\int_\Omega \big|2\nabla \sqrt{\sigma} + \sqrt{\sigma} \nabla \beta(\varphi)\big|^2\,\d x\d t\leq
\liminf_{n\to \infty} \int_0^T \!\!\int_\Omega \big| 2\nabla \sqrt{\sigma^n} + \sqrt{\sigma^n} \nabla \beta(\varphi^n) \big|^2 \,\d x\d t.
\notag
\end{align}
The proof of Theorem \ref{WEAK-SOL}-(1) is complete.
\hfill $\square$

\subsection{Proof of Corollary \ref{finite-to-weak}}
We recall that the approximate solution $\sigma^{k}$ considered in Section \ref{ue} satisfies
$$
\sup_{t\geq 0}\int_t^{t+1} \|\sigma^k(\tau) \|^2 \, \d \tau\leq C,
\quad
\sup_{t\geq 0}\int_t^{t+1} \|\nabla \mu^k(\tau) \|^2 \, \d \tau\leq C,
\quad \forall \, k \geq \widehat{k}.
$$
Thus, an application of the uniform Gronwall lemma (see \cite[Chapter III, Lemma 1.1]{Temam}) to \eqref{sL2-4} implies that
$$
\|\sigma^k(t+\tau)\|^2 \leq \frac{C}{\tau},\quad \forall\, t \geq 0,\quad \forall\, \tau\in (0,1],
$$
where $C>0$ depends on $\mathcal{E}_0$, $\Omega$ and coefficients of the system, but it is independent of $k$, $\gamma$, $\epsilon$, $n$ (and also $\|\sigma_{0,n}\|$), $t$ and $\tau$. Based on this fact, we find
\begin{align}
&\sup_{t\geq \tau} \|\sigma^k\|_{L^2(t,t+1; H^1(\Omega))}\leq C(\tau),
\quad
\sup_{t\geq \tau}\|\varphi^k\|_{L^4(t,t+1; H^2(\Omega))}\leq C(\tau),
\notag
\end{align}
where the positive constant $C(\tau)$ depends on $\tau$, $\mathcal{E}_0$, $\Omega$ and the coefficients of the system, but it is independent of $k$, $\gamma$, $\epsilon$, $n$ (and also $\|\sigma_{0,n}\|$), $t$. Moreover, $C(\tau)$ explodes as $\tau\to 0^+$.
Taking $k\to \infty$, the same properties hold for the approximate solutions obtained in Proposition \ref{p2}.
Since $\tau\in (0,1]$ is arbitrary, repeating the procedure in Section \ref{proof-thm1}, we can obtain a global finite energy solution defined on $[0,\infty)$ that becomes a global weak solution for $t>0$.

The proof of Corollary \ref{finite-to-weak} is complete.
 \hfill $\square$

\section{Regularity Results for Auxiliary Decoupled Systems}
\label{sec:decoup}
\setcounter{equation}{0}

In this section, we investigate some auxiliary problems related to the coupled system \eqref{NSCH}--\eqref{NSCH-bic}. They will play a crucial role in the
sequel in establishing the global strong well-posedness of problem \eqref{NSCH}--\eqref{NSCH-bic} and the propagation of regularity for global weak solutions.

\subsection{Diffusion equation with divergence-free drift}
\label{sec:reg-sigma}

Let $\vv$, $\varphi$ be two given functions with suitable regularity properties. Consider the following diffusion equation with convection:
\begin{alignat}{3}
&\partial_t\sigma +\vv\cdot\nabla\sigma -\Delta\sigma
=  \mathrm{div}\big(\beta'(\varphi)\sigma \nabla \varphi\big),
\quad \text{in}\ \Omega\times (0,\infty),
\label{re-di-sigma}
\end{alignat}
equipped with the following boundary and initial conditions
\begin{alignat}{3}
&\partial_{\boldsymbol{n}} \sigma =0  \quad \text{on}\ \partial\Omega\times(0,\infty),
\qquad \sigma|_{t=0}=\sigma_{0} \quad \text{in}\ \Omega.
\label{re-di-sigma-bdini}
\end{alignat}

First, we prove the existence and uniqueness of a global weak solution to problem
\eqref{re-di-sigma}--\eqref{re-di-sigma-bdini}.
\begin{proposition}
\label{sigma-weak}
Let $\Omega$ be a bounded domain in $\mathbb{R}^2$ with a $C^2$ boundary.
Assume that (H4) is satisfied and
$$
\vv \in L^{\infty}(0,\infty;\L^2_{\sigma}(\Omega))\cap L^2(0,\infty;\H^1_{0,\sigma}(\Omega)),\quad
\varphi \in L^\infty(0,\infty;H^1(\Omega))\cap L^2_{\mathrm{uloc}}([0,\infty);H^2_N(\Omega)).
$$
Then, for any initial datum $\sigma_0 \in L^2(\Omega)$ satisfying $\sigma_0 \geq 0$ almost everywhere in $\Omega$, problem \eqref{re-di-sigma}--\eqref{re-di-sigma-bdini} admits a unique global weak  solution $\sigma$ in $\Omega\times [0,\infty)$ such that
\begin{equation}
\label{REG1-w}
\sigma \in L^{\infty}(0,T;L^2(\Omega)) \cap L^2(0,T; H^1(\Omega))\cap H^1(0,T;(H^1(\Omega))'),
\end{equation}
for any $T>0$,  with $\sigma(x,t)\geq 0$ almost everywhere in $\Omega \times (0,\infty)$.
Moreover, it holds
\begin{align}
&\l \partial_t \sigma, \xi\r_{(H^1(\Omega))',H^1(\Omega)}
-(\sigma\vv, \nabla \xi)+ (\nabla \sigma, \nabla\xi)
+ \left( \beta'(\varphi)\sigma \nabla \varphi, \nabla \xi\right)=0,
\quad \forall\, \xi \in H^1(\Omega),
\label{sig-weak2-w}
\end{align}
almost everywhere in $(0,\infty)$, and  $\sigma|_{t=0}=\sigma_{0}$ almost everywhere in $\Omega$.
\end{proposition}
\begin{proof}
The existence of a nonnegative global weak solution follows from the construction  of smooth approximate solutions (see \cite[Lemma 3.1]{GHW1}), the \textit{a priori} estimates (see below) combined with the compactness method.
Testing \eqref{sig-weak2-w} with $1$, we get
\begin{align}
\int_\Omega\sigma(t)\,\mathrm{d}x=\int_\Omega \sigma_0\,\mathrm{d}x,\quad \forall\, t\geq 0.
\label{conv-sig}
\end{align}
Next, a modification of the argument in Section 4.1 (see Step 5) yields that
\begin{equation}
\label{sL2-2-b}
  \ddt \|\sigma \|^2 + \|\nabla \sigma \|^2
\leq C \left(1+\|\varphi\|_{H^2(\Omega)}^2 \right) \|\sigma\|^2.
\end{equation}
By Gronwall's lemma, we have
\begin{align}
\|\sigma\|_{L^\infty(0,T;L^2(\Omega))}\leq K_2,\quad \|\sigma\|_{L^2(0,T;H^1(\Omega))}\leq K_2,
\label{es-sigL2}
\end{align}
for any $T>0$, where $K_2>0$ depends on $\|\sigma_0\|$, $\|\varphi\|_{L^2(0,T;H^2(\Omega))}$, $T$. Moreover, similar to \eqref{sigma_t}, from the assumptions on $\vv$ and $\varphi$, by a comparison in \eqref{sig-weak2-w}, we can deduce that $\partial_t\sigma \in L^2(0,T;(H^1(\Omega))')$.

Now we prove the uniqueness. Let $\sigma_1$ and $\sigma_2$ be two weak solutions to problem \eqref{re-di-sigma}--\eqref{re-di-sigma-bdini}, corresponding to the initial data $\sigma_{0,1}$ and $\sigma_{0,2}$, respectively. Define the difference
$\widetilde{\sigma}=\sigma_1-\sigma_2$, which solves
\begin{align}
	\label{sigma-weak1}
	\l \partial_t \widetilde{\sigma},\xi\r_{(H^1(\Omega))',H^1(\Omega)}
- ( \vv \widetilde{\sigma},\nabla \xi )
	+(\nabla\widetilde{\sigma},\nabla \xi) + (\beta'(\varphi)\widetilde{\sigma}\nabla \varphi, \nabla \xi )=0,
\quad \forall\, \xi\in H^1(\Omega),
\end{align}
almost everywhere in $(0,T)$ for any fixed $T>0$. Taking $\xi= \widetilde{\sigma}$ in \eqref{sigma-weak1}, after integration by parts, we get
\begin{align}
\frac{1}{2}\ddt \|\widetilde{\sigma}\|^2+\|\nabla \widetilde{\sigma}\|^2
&= - (\beta'(\varphi)\widetilde{\sigma}\nabla \varphi,\nabla \widetilde{\sigma})
\notag\\
&\leq \|\beta'(\varphi)\|_{L^\infty(\Omega)}
\|\nabla \varphi\|_{\bm{L}^4(\Omega)}
\|\widetilde{\sigma}\|_{L^4(\Omega)}\|\nabla \widetilde{\sigma}\|
\notag\\
&\leq C\|\nabla \varphi\|^\frac12\|\varphi\|_{H^2(\Omega)}^\frac12\|\widetilde{\sigma}\|^\frac12
\|\nabla \widetilde{\sigma}\|^\frac32
\notag\\
&\leq \frac12 \|\nabla \widetilde{\sigma}\|^2 +C \|\varphi\|_{H^2(\Omega)}^2\|\widetilde{\sigma}\|^2.
\label{sig-uni}
\end{align}
An application of Gronwall's lemma yields the continuous dependence estimate
$$
\|\sigma_1(t)-\sigma_2(t)\|^2
\leq \|\sigma_{0,1}-\sigma_{0,2}\|^2 \mathrm{exp}\left(C\int_0^t\|\varphi(s)\|_{H^2(\Omega)}^2\,\mathrm{d}s\right), \quad \forall\, t\in [0,T],
$$
which guarantees the uniqueness of the weak solution.

The proof of Proposition \ref{sigma-weak} is complete.
\end{proof}

The following proposition yields the strong well-posedness of problem
\eqref{re-di-sigma}--\eqref{re-di-sigma-bdini}.
\begin{proposition}
\label{sigma-strong}
Let $\Omega$ be a bounded domain in $\mathbb{R}^2$ with a $C^2$ boundary. Assume that (H4) is satisfied and
\begin{align*}
&\vv \in L^{\infty}(0,\infty;\L^2_{\sigma}(\Omega))\cap L^2(0,\infty;\H^1_{0,\sigma}(\Omega)),
\\
&\varphi \in L^\infty(0,\infty;H^1(\Omega))\cap L^2_{\mathrm{uloc}}([0,\infty);W^{2,q}(\Omega)\cap H^2_N(\Omega)),
\end{align*}
for some $q>2$. Then for any initial datum $\sigma_0 \in H^1(\Omega)$ satisfying $\sigma_0 \geq 0$ almost everywhere in $\Omega$, problem
\eqref{re-di-sigma}--\eqref{re-di-sigma-bdini} admits a unique global strong solution $\sigma$ in $\Omega\times [0,\infty)$ such that
\begin{equation}
\label{REG1}
\begin{split}
\sigma \in L^{\infty}(0,T;H^1(\Omega)) \cap L^2(0,T; H^2(\Omega))\cap H^1(0,T;L^2(\Omega)),
\end{split}
\end{equation}
for any $T>0$, with $\sigma(x,t)\geq 0$ almost everywhere in $\Omega \times (0,\infty)$,
which satisfies \eqref{re-di-sigma} almost everywhere in $\Omega\times(0,\infty)$. Moreover, we have $\partial_{\boldsymbol{n}} \sigma =0$ almost everywhere on $\partial\Omega\times(0,\infty)$ and  $\sigma|_{t=0}=\sigma_{0}$ almost everywhere in $\Omega$.
\end{proposition}
\begin{proof}
We only derive the necessary higher-order \textit{a priori} estimates.
Multiplying \eqref{re-di-sigma} by $-\Delta\sigma$ and integrating over $\Omega$, we obtain
\begin{equation}
\begin{aligned}
\frac{1}{2}\frac{\mathrm{d}}{\mathrm{d} t}\|\nabla \sigma\|^2+\left\|\Delta\sigma\right\|^2
&=\int_{\Omega}\vv \cdot \nabla \sigma\Delta\sigma\,\mathrm{d}x
- \int_{\Omega}\mathrm{div}\big(\beta'(\varphi)\sigma \nabla \varphi\big) \Delta\sigma \,\mathrm{d}x.
\end{aligned}
\notag
\end{equation}
Using H\"older's inequality, Young's inequality, the Sobolev embedding theorem, the Gagliardo--Nirenberg inequality, the Poincar\'e--Wirtinger inequality and \eqref{conv-sig}, we have
\begin{equation*}
\begin{aligned}
\left|\int_{\Omega}\vv \cdot \nabla \sigma \Delta\sigma \,\mathrm{d}x\right|
&\le \|\vv \|_{\L^4(\Omega)}\|\nabla \sigma\|_{\bm{L}^4(\Omega)}\|\Delta\sigma\|\\
&\le \frac18\|\Delta\sigma\|^2 + C\|\vv \|\|\vv \|_{\bm{H}^1(\Omega)}
\|\nabla \sigma\| (\|\Delta\sigma\| + \overline{\sigma_0}) \\
&\le \frac14\|\Delta\sigma\|^2 + C\|\nabla \vv \|^2\|\nabla \sigma\|^2+C,
\end{aligned}
\end{equation*}
and
\begin{equation*}
\begin{aligned}
&\left|\int_{\Omega}\mathrm{div}\big(\beta'(\varphi)\sigma \nabla \varphi\big) \Delta\sigma \,\mathrm{d}x\right|
\\
&\quad\le \|\beta'(\varphi)\|_{L^\infty(\Omega)}\|\nabla \sigma\|\|\nabla \varphi\|_{\bm{L}^\infty(\Omega)}\|\Delta\sigma\|
+ \|\beta''(\varphi)\|_{L^\infty(\Omega)}\|\sigma\|_{L^\frac{2q}{q-2}(\Omega)} \|\nabla \varphi\|^2_{\bm{L}^{2q}(\Omega)}\|\Delta\sigma\|
\\
&\qquad +\|\beta'(\varphi)\|_{L^\infty(\Omega)}\|\sigma\|_{L^\frac{2q}{q-2}(\Omega)} \|\Delta \varphi\|_{L^q(\Omega)}\|\Delta\sigma\| \\
&\quad \leq C \|\nabla \sigma\|\|\varphi\|_{W^{2,q}(\Omega)}\|\Delta \sigma\|
+ C
\| \sigma \|_{H^1(\Omega)}
(\| \nabla \varphi\| +1)
 \| \varphi\|_{W^{2,q}(\Omega)} \| \Delta \sigma\|
\\
&\quad\le \frac14\|\Delta\sigma\|^2
+ C\|\varphi\|_{W^{2,q}(\Omega)}^2(\|\nabla \sigma\|^2+1).
\end{aligned}
\end{equation*}
From the above estimates, we can derive the following differential inequality
\begin{equation}
\begin{aligned}
\frac{\mathrm{d}}{\mathrm{d} t} \|\nabla \sigma\|^2 + \|\Delta\sigma\|^2
&\le C\big(\|\nabla \vv \|^2+\|\varphi\|_{W^{2,q}(\Omega)}^2\big) \|\nabla \sigma\|^2
+ C\big(\|\varphi\|_{W^{2,q}(\Omega)}^2+1\big).
\end{aligned}
\label{es-sigH0}
\end{equation}
Applying Gronwall's lemma and the Poincar\'{e}--Wirtinger inequality, we obtain
\begin{align}
\|\sigma\|_{L^\infty(0,T;H^1(\Omega))}\leq K_3,\quad \|\sigma\|_{L^2(0,T;H^2(\Omega))}\leq K_3,
\label{es-sigH1}
\end{align}
for any $T>0$, where the constant $K_3>0$ depends on $\|\sigma_0\|_{H^1(\Omega)}$,
$\|\vv\|_{L^\infty(0,\infty;\L^2_{\sigma}(\Omega))}$, $\|\vv\|_{L^2(0,\infty;\H^1_{0,\sigma}(\Omega))}$,
$\|\varphi \|_{L^\infty(0,T;H^1(\Omega))}$, $\|\varphi\|_{L^2(0,T;W^{2,q}(\Omega))}$, and $T$. Finally, by a comparison in the equation \eqref{re-di-sigma}, it is straightforward to check that $\partial_t\sigma \in L^2(0,T;L^2(\Omega))$.

The proof of Proposition \ref{sigma-strong} is complete.
\end{proof}

Finally, we derive an $L^\infty$-estimate of $\sigma$ under some additional assumptions on $\sigma_0$ and $\varphi$.

\begin{corollary}
\label{non-blowup}
Let the assumptions in Proposition \ref{sigma-strong} be satisfied. Assume in addition,
$$\sigma_0\in L^\infty(\Omega)\quad \text{and}\quad \varphi\in L^\infty(0,\infty;W^{2,q}(\Omega))\quad \text{for some }q>2.$$
Then, we have
\begin{align}
\|\sigma\|_{L^\infty(0,\infty;L^\infty(\Omega))}\leq C,
\label{sig-Linf}
\end{align}
where $C>0$ depends on $\|\sigma_0\|_{L^\infty(\Omega)}$, $\|\varphi\|_{L^\infty(0,\infty;W^{2,q}(\Omega))}$ and $\Omega$.
\end{corollary}
\begin{proof}
For any $p\geq 2$, multiplying \eqref{re-di-sigma} by $\sigma^{p-1}$ and integrating over $\Omega$, exploiting the incompressibility condition and the no-slip boundary condition for $\vv$, we find
\begin{equation}
\frac{1}{p} \ddt \int_\Omega \sigma^p \, \d x +\frac{4(p-1)}{p^2}  \int_{\Omega}|\nabla \sigma^\frac{p}{2}|^2 \, \d x
= -(p-1) \int_{\Omega}\beta'(\varphi) \sigma^{p-1} \nabla \varphi \cdot \nabla \sigma \, \d x.
\notag
\end{equation}
Observing that
\begin{align*}
& \left|(p-1) \int_{\Omega} \beta'(\varphi)  \sigma^{p-1} \nabla \varphi \cdot \nabla \sigma \, \d x \right|
\notag \\
& \quad \leq \frac{2(p-1)}{p} \|\beta'(\varphi)\|_{L^\infty(\Omega)}\|\nabla \varphi\|_{\bm{L}^\infty(\Omega)}
\big\| \sigma^{\frac{p}{2}}\big\|
\big\|\nabla \sigma^\frac{p}{2}\big\|
\\
&\quad \leq \frac{2(p-1)}{p^2}  \int_{\Omega}|\nabla \sigma^\frac{p}{2}|^2 \, \d x
+ \frac{p-1}{2} \|\beta'(\varphi)\|_{L^\infty(\Omega)}^2\|\nabla \varphi\|_{\bm{L}^\infty(\Omega)}^2
\int_\Omega  \sigma^p\,\mathrm{d}x,
\end{align*}
and using the Sobolev embedding theorem $W^{2,q}(\Omega)\hookrightarrow W^{1,\infty}(\Omega)$ for $q>2$, we get
\begin{align}
 \ddt \int_\Omega \sigma^p \, \d x
+\frac{2(p-1)}{p}  \int_{\Omega}|\nabla \sigma^\frac{p}{2}|^2 \, \d x
& \leq Cp(p-1)\int_\Omega  \sigma^p\,\mathrm{d}x,
\label{sig-Linf1}
\end{align}
where the constant $C>0$ is independent of the exponent $p$. This allows us to apply the Moser--Alikakos iteration technique \cite{Ali79}. By the same argument as that for \cite[Lemma 3.2]{Tao11}, we can deduce from differential inequality \eqref{sig-Linf1} that the $L^\infty$-estimate \eqref{sig-Linf} holds. The details are left to the interested reader.
\end{proof}

\subsection{Cahn--Hilliard equation with divergence-free drift and conservative forcing}
\label{sec:reg-CH}

Let $\vv$ and $\sigma$ be two given functions with suitable regularity properties. We consider the following Cahn--Hilliard equation with divergence-free drift and conservative forcing
\begin{equation}
\begin{cases}
\label{CH}
\partial_t \varphi+ \vv \cdot \nabla \varphi = \div(m(\varphi) \nabla \mu), \\
\mu= -\Delta \varphi+ \Psi'(\varphi) +  \beta'(\varphi) \sigma,
\end{cases}
\quad \text{ in } \Omega\times (0,\infty),
\end{equation}
subject to the boundary and initial conditions
\begin{equation}
\label{bcic}
\begin{cases}
\partial_\n \varphi=\partial_\n \mu=0 \quad &\text{ on } \partial \Omega \times (0,\infty),\\
\varphi|_{t=0}=\varphi_0\quad &\text{ in } \Omega.
\end{cases}
\end{equation}

First, we establish the existence and uniqueness of a global weak solution to problem  \eqref{CH}--\eqref{bcic}. This is a generalization of \cite[Theorem 6]{Abels2009} (see also \cite[Remark 2.2]{AGG2023}) and \cite[Theorem 1.2 - (A), (B)]{CGGG}.

\begin{proposition}
\label{well-pos}
Let $\Omega$ be a bounded domain in $\mathbb{R}^2$ with a $C^2$ boundary. Assume that (H1), (H3) and (H4) are satisfied, and
$$
\vv \in  L^2(0,\infty;\H^1_{0,\sigma}(\Omega))\quad \text{and}\quad
\sigma \in L^\infty(0,\infty,L^2(\Omega))\cap L_{\uloc}^2([0,\infty); H^1(\Omega)).
$$
Then, for any initial datum $\varphi_0\in H^1(\Omega)$ with $\| \varphi_0\|_{L^\infty(\Omega)}\leq 1$ and $\left|\overline{\varphi_0}\right|<1$, problem \eqref{CH}--\eqref{bcic} admits a unique global weak solution $(\varphi,\mu)$ in $\Omega\times (0,\infty)$ such that
\begin{enumerate}
\item The weak solution satisfies
\begin{align*}
&\varphi \in  L^\infty(0,T ; H^1(\Omega))\cap L^4(0,T;H^2(\Omega))\cap L^2(0,T;W^{2,q}(\Omega)),\\
&\varphi \in L^{\infty}(\Omega\times (0,\infty))\ \text{ such that }  |\varphi(x,t)|<1  \ \text{a.e. in }\Omega\times (0,\infty),\\
&\partial_t \varphi \in L^2(0,T; V_0^{-1}),\\
& \mu \in L^2(0,T;H^1(\Omega)),
\quad\Psi_0'(\varphi)\in L^2(0,T;L^q(\Omega)),
\end{align*}
 for any $q\in [2,\infty)$ and $T>0$.
 \smallskip
\item The weak solution solves \eqref{CH} in a variational sense as follows:
\begin{align}
\label{e2}
\l \partial_t \varphi,\xi\r_{(H^1(\Omega))',H^1(\Omega)}
+ ( \vv\cdot \nabla \varphi,\xi)
+ (m(\varphi)\nabla \mu, \nabla \xi) =0,
\quad \forall \, \xi \in H^1(\Omega), \ \text{a.e. in } (0,\infty),
\end{align}
with
$$\mu=-\Delta \varphi + \Psi'(\varphi) + \beta'(\varphi) \sigma\quad \text{a.e. in}\ \ \Omega\times (0,\infty).
$$
Besides, $\partial_\n \varphi=0$ holds almost everywhere on $\partial\Omega\times(0,\infty)$ and
$\varphi|_{t=0}=\varphi_0$ holds almost everywhere in $\Omega$.
\smallskip
\item The weak solution satisfies the following energy equality
\begin{align}
& E_{\mathrm{free}}(\varphi(t))
+ \int_0^t\!\!\int_\Omega
m(\varphi)| \nabla \mu|^2 \,\d x \d s
\notag \\
&\quad =  E_{\mathrm{free}}(\varphi_0) - \int_{0}^t \big(\vv \cdot \nabla \varphi, \mu-\beta'(\varphi) \sigma\big) \, \d s
+  \int_{0}^t \big(m(\varphi)\nabla \mu, \nabla (\beta'(\varphi)\sigma)\big) \, \d s,
\label{EI}
\end{align}
for every $t\geq 0$, where
$$
E_{\mathrm{free}}(\varphi)=\frac12\|\nabla \varphi\|^2+\int_\Omega \Psi(\varphi)\,\mathrm{d}x.
$$
\end{enumerate}
\end{proposition}

\begin{proof}
The existence of a global weak solution to problem \eqref{CH}--\eqref{bcic} can be proven by regularizing the singular potential and applying a suitable Faedo--Galerkin approximation scheme analogous to that in Section \ref{modif-sys} (or, alternatively, by exploiting the viscous regularization as in \cite{AGG2023}). We refer to, e.g.,  \cite[Lemma 2.4]{Sch07} for further details of the approximating procedure, where the simplified case with $\vv=\mathbf{0}$, $\sigma=0$ was treated. In the following, we only derive necessary formal {\it a priori} estimates.

Integrating \eqref{CH}$_1$ over $\Omega$, using the incompressibility and the no-slip boundary condition of the velocity field $\vv$, we get
\begin{equation}
\label{cons-mass}
\int_{\Omega} \varphi(t) \, \d x= \int_{\Omega} \varphi_{0} \, \d x, \quad \forall \, t \geq 0.
\end{equation}
Multiplying \eqref{CH}$_1$ by
$\mu - \beta'(\varphi) \sigma$, integrating over $\Omega$ and exploiting the definition of $\mu$, we find
\begin{align}
&
\ddt  \int_{\Omega} \Big( \frac12 |\nabla \varphi|^2 + \Psi(\varphi)\Big) \, \d x
+ \int_{\Omega}  m(\varphi) |\nabla \mu|^2\,\d x
\notag\\
&\quad = - \int_{\Omega} (\vv \cdot \nabla \varphi) (\mu-\beta'(\varphi) \sigma) \, \d x
+  \int_{\Omega} m(\varphi) \nabla \mu \cdot \nabla (\beta'(\varphi) \sigma)  \, \d x,
\label{CH-mu}
\end{align}
for almost every $t>0$. To proceed, we assume the {\it a priori} estimate
\begin{align}
\|\varphi\|_{L^\infty(0,\infty;L^\infty(\Omega))}\leq 1.
\label{phi-Lif}
\end{align}
which can be guaranteed by the singularity of $\Psi$, see Remark \ref{rem-Linf} below for further comments.
By the Gagliardo--Nirenberg inequality, we have
\begin{align}
\|\nabla (\beta'(\varphi) \sigma)\|
& \leq \|\beta'(\varphi)\|_{L^\infty(\Omega)}\|\nabla \sigma\|
+ \|\beta''(\varphi)\|_{L^\infty(\Omega)}\|\nabla \varphi\|_{\bm{L}^4(\Omega)}
\|\sigma\|_{L^4(\Omega)}
\notag\\
&\leq C\|\nabla \sigma\|+ C \|\nabla \varphi\|^\frac12
\|\Delta \varphi\|^\frac12
\|\sigma\|^\frac12(\|\nabla \sigma\|+\|\sigma\|)^\frac12.
\label{es-beta-sigH1}
\end{align}
Then, exploiting \eqref{phi-Lif}, \eqref{es-beta-sigH1}, (H4), the assumptions on $(\vv, \sigma)$ and  Young's inequality, we see that
\begin{align}
& \left|
\int_{\Omega} (\vv \cdot \nabla \varphi) (\mu - \beta'(\varphi) \sigma)\, \d x \right|
\notag \\
&\quad \leq \left| \int_{\Omega} (\vv \cdot \nabla  \mu)  \varphi \, \d x\right|
+ \left|\int_\Omega \varphi \vv \cdot \nabla (\beta'(\varphi) \sigma) \, \d x\right|
\notag\\
&\quad \leq \| \varphi\|_{L^\infty(\Omega)}\|\vv\| \| \nabla  \mu\|
+ C\|\varphi\|_{L^\infty(\Omega)}\|\vv\| \|\nabla  \sigma\|
\notag\\
&\qquad
+  C \|\varphi\|_{L^\infty(\Omega)}\|\vv\|\|\nabla \varphi\|^\frac12
\|\Delta \varphi\|^\frac12
\|\sigma\|^\frac12(\|\nabla \sigma\|+\|\sigma\|)^\frac12
\notag\\
&\quad \leq \frac{m_*}{8}\|\nabla \mu\|^2+ C\|\vv\|^2 + C\|\nabla \varphi\|
\|\Delta \varphi\|
\|\sigma\|(\|\nabla \sigma\|+\|\sigma\|) + C\|\nabla \sigma\|^2
\notag\\
&\quad \leq  \|\Delta \varphi\|^2 + \frac{m_*}{8}\|\nabla \mu\|^2+ C(\|\vv\|^2+\|\nabla \sigma\|^2)
 +C  (\|\sigma\|^2 \|\nabla \sigma\|^2 +\|\sigma\|^4) \|\nabla \varphi\|^2.
\label{es-vvmu}
\end{align}
Using \eqref{es-beta-sigH1}, by a similar argument for \eqref{es-vvmu}, we get
\begin{align}
& \left|\int_{\Omega} m(\varphi) \nabla \mu \cdot \nabla (\beta'(\varphi) \sigma)  \, \d x\right|
 \leq \|m(\varphi)\|_{L^\infty(\Omega)}\|\nabla \mu\| \|\nabla (\beta'(\varphi) \sigma)\|
\notag\\
&\quad \leq \|\Delta\varphi\|^2 + \frac{m_*}{8}\|\nabla \mu\|^2 +C\|\nabla \sigma\|^2
+C  (\|\sigma\|^2 \|\nabla \sigma\|^2 + \|\sigma\|^4) \|\nabla \varphi\|^2.
\label{es-mmu}
\end{align}
In order to estimate $\|\Delta\varphi\|$, let us consider the following elliptic problem
\begin{equation}
\begin{cases}
-\Delta \varphi + \Psi_0(\varphi)= \mu+\theta_0\varphi-\beta'(\varphi)\sigma,&\quad \text{in}\ \Omega\times(0,T),\\
\partial_{\bm{n}}\varphi=0,&\quad \text{on}\ \partial\Omega\times(0,T).
\end{cases}
\label{vphi-ellip}
\end{equation}
Testing \eqref{vphi-ellip}$_1$ by $-\Delta \varphi$ yields that
\begin{align}
\|\Delta \varphi\|^2
+\int_{\Omega}\Psi_0''(\varphi)|\nabla \varphi|^2 \, \d x
& = \int_{\Omega} \nabla \mu \cdot \nabla \varphi \, \d x
+ \theta_0 \int_{\Omega}  |\nabla \varphi|^2 \, \d x
+  \int_\Omega \nabla (\beta'(\varphi)\sigma) \cdot \nabla \varphi \, \d x
\notag \\
&\leq \|\nabla \mu\|\|\nabla \varphi\|+ |\theta_0| \|\nabla \varphi\|^2
+ \|\beta'(\varphi)\|_{L^\infty(\Omega)}\|\nabla \sigma\|\|\nabla \varphi\|
\notag\\
&\quad + \|\beta''(\varphi)\|_{L^\infty(\Omega)}\|\nabla \varphi\|_{\bm{L}^4(\Omega)}\|\sigma\|_{L^4(\Omega)}\|\nabla \varphi\|
\notag\\
&\leq \frac{m_*}{16} \|\nabla \mu\|^2+ C\|\nabla \varphi\|^2+ C\|\nabla \sigma\|^2
\notag\\
&\quad + C\|\nabla \varphi\|^\frac32 \|\Delta \varphi\|^\frac12
\|\sigma\|^\frac12 (\|\nabla \sigma\|+\|\sigma\|)^\frac12
\notag\\
&\leq \frac12\|\Delta\varphi\|^2+ \frac{m_*}{16}\|\nabla \mu\|^2+ C\|\nabla \sigma\|^2\notag\\
&\quad + C(1+\|\sigma\|^4 +  \|\sigma\|^2\|\nabla \sigma\|^2)\|\nabla \varphi\|^2.
\label{Dephi-L2}
\end{align}
Combining \eqref{es-vvmu}, \eqref{es-mmu}, \eqref{Dephi-L2} and the assumption $\sigma \in L^\infty(0,\infty,L^2(\Omega))$, we deduce from \eqref{CH-mu} and (H1) that
\begin{align}
&\ddt   \int_{\Omega} \Big( \frac12 |\nabla \varphi|^2 + \Psi(\varphi)\Big) \, \d x
+ \frac12 \int_{\Omega}  m(\varphi) |\nabla \mu|^2 \, \d x
\notag\\
&\quad \leq  C(1+\|\nabla \sigma\|^2)  \int_{\Omega} \Big( \frac12 |\nabla \varphi|^2 + \Psi(\varphi)\Big) \, \d x  + C(1+ \| \vv\|^2 + \| \nabla \sigma\|^2).
\label{EE-CH}
\end{align}
An application of Gronwall's lemma yields
\begin{equation}
\label{EE-2}
\begin{split}
&\sup_{t \in [0,T]} \left(\frac12 \|\nabla \varphi(t)\|^2+\int_\Omega\Psi(\varphi(t))\,\mathrm{d}x\right)
+ \frac{m_*}{2}\int_0^T \|\nabla \mu(s)\|^2 \, \d s
 \leq C,
\end{split}
\end{equation}
for any $T>0$, where the positive constant $C$ depends on $E_{\text{free}}(\varphi_0)$, $\| \vv\|_{L^2(0,\infty;\bm{H}^1_{0,\sigma(\Omega)})}$,
$\| \nabla \sigma \|_{L^2(0,T;\bm{L}^2(\Omega))}$, $\|  \sigma \|_{L^\infty(0,\infty;L^2(\Omega))}$, $\theta$, $\theta_0$, $\Omega$ and $T$. Combining \eqref{cons-mass} and \eqref{EE-2}, we obtain
\begin{equation}
\label{EE-3}
\| \varphi \|_{L^\infty(0,T;H^1(\Omega))}\leq C_0, \quad
\| \nabla \mu\|_{L^2(0,T;\bm{L}^2(\Omega))}\leq C_0, \quad \forall \, T>0,
\end{equation}
where the positive constant $C_0$ depends on $E_{\text{free}}(\varphi_0)$, $|\overline{\varphi_0}|$, $\theta$, $\theta_0$, $\Omega $, $\| \vv\|_{L^2(0,\infty;\bm{H}^1_{0,\sigma(\Omega)})}$,
$\| \nabla \sigma \|_{L^2(0,T;\bm{L}^2(\Omega))}$, $\|  \sigma \|_{L^\infty(0,\infty;L^2(\Omega))}$, and $T$.

Recall the well-known inequality (see, for instance, \cite{MZ04})
\begin{equation}
\label{F'-L1}
\int_{\Omega} \left|\Psi_0'(\varphi) \right| \, \d x \leq C_1 \int_{\Omega}\Psi_0'(\varphi) \left( \varphi-\overline{\varphi_{0}} \right) \, \d x+ C_2,
\end{equation}
where $C_1>0$ depends only on  $\overline{\varphi_{0}}\in (-1,1)$ and $C_2>0$ depends only on $\theta$ and $\overline{\varphi_{0}}$. Multiplying \eqref{CH}$_2$ by $\varphi - \overline{\varphi_{0}}$ (cf. \eqref{cons-mass}), we find
\begin{align*}
& \int_{\Omega} |\nabla \varphi|^2 \, \d x
+
\int_{\Omega}\Psi_0'(\varphi) \left( \varphi -\overline{\varphi_{0}} \right) \, \d x
\\
&\quad = \int_{\Omega} (\mu-\overline{\mu}) \left( \varphi -\overline{\varphi_{0}} \right) \, \d x + \theta_0 \int_{\Omega} \varphi \left( \varphi -\overline{\varphi_{0}} \right) \, \d x
-  \int_{\Omega} \beta'(\varphi)\sigma (\varphi-\overline{\varphi_0}) \, \d x.
\end{align*}
Then by the Poincar\'{e}--Wirtinger inequality, \eqref{F'-L1} and (H3), we reach
\begin{equation}
\label{F'-L1e}
 \int_{\Omega} \left|\Psi_0'(\varphi) \right| \, \d x  \leq
C \|\nabla \varphi\|\left(\| \nabla \mu\|+\|\varphi\| + \| \sigma\|\right) +C,
\end{equation}
where $C$ depends only on $\theta_0$, $C_1$, $C_2$ and $\Omega$. Since $\overline{\mu}= \overline{\Psi_0'(\varphi)}- \theta_0 \overline{\varphi}+  \overline{\beta'(\varphi)\sigma}$, we infer from \eqref{cons-mass} and \eqref{F'-L1e} that
$$
|\overline{\mu}|\leq C \|\nabla \varphi\|\left(\| \nabla \mu\|+\|\varphi\| + \| \sigma\|\right)+C.
$$
Applying the Poincar\'{e}--Wirtinger inequality again, we have
\begin{equation}
\label{mu-H1e}
\| \mu\|_{H^1(\Omega)}\leq
 C \|\nabla \varphi\|\left(\| \nabla \mu\|+\|\varphi\| + \| \sigma\|\right)+C\|\nabla \mu\|+C,
\end{equation}
which together with \eqref{EE-3} yields that  $\|\mu\|_{L^2(0,T;H^1(\Omega))}$ is bounded for any fixed $T>0$.

Noticing that $\|\beta'(\varphi)\sigma\|_{L^q(\Omega)}\leq \|\beta'(\varphi)\|_{L^\infty(\Omega)}\|\sigma\|_{L^q(\Omega)}$ for any $q\geq 2$. Applying \cite[Lemma A.1]{CG2020} to problem \eqref{vphi-ellip}, we infer from \eqref{EE-3}, \eqref{mu-H1e} and the elliptic regularity theory that
\begin{equation}
\label{phi2-p}
\| \varphi\|_{W^{2,q}(\Omega)}+ \|\Psi_0'(\varphi)\|_{L^q(\Omega)}\leq C\left( 1+ \| \nabla\mu\|  +  \|  \sigma\|_{L^q(\Omega)} \right),
\end{equation}
for any $q\geq 2$.
The positive constant $C$ in \eqref{phi2-p} depends on $q$, $\overline{\varphi_0}$, $\theta$, $\theta_0$, $\Omega$, $C_0$, $C_1, C_2$.
Thus, we can deduce from \eqref{mu-H1e}, \eqref{phi2-p} that
\begin{align}
\label{phi-w2p}
\|\varphi\|_{L^2(0,T;W^{2,q}(\Omega))}\leq C_3,\quad \|\Psi_0'(\varphi)\|_{L^2(0,T;L^q(\Omega))}\leq C_3,\quad \forall\, T>0.
\end{align}
Besides, exploiting the proof of \eqref{Dephi-L2} and using the estimate \eqref{EE-3}, we find
\begin{align}
\|\Delta \varphi\|^2
&\leq C\|\nabla \mu\|
  + C(1+\|\nabla \sigma\|)
  + C\|\Delta \varphi\|^\frac12(1+\|\nabla \sigma\|)^\frac12
\notag\\
&\leq  \frac12 \|\Delta\varphi\|^2+ C \left( 1+ \| \nabla \mu\|+ \| \nabla \sigma\|\right),
\label{Dephi-L2b}
\end{align}
which implies
\begin{align}
\label{phi-l4H2}
\|\varphi\|_{L^4(0,T;H^{2}(\Omega))}\leq C_4.
\end{align}
In addition, a comparison argument in \eqref{CH}$_1$ yields
\begin{equation}
\label{phit-est}
\| \partial_t \varphi\|_{(H^1(\Omega))'}
\leq \|\varphi\|_{L^4(\Omega)} \| \vv \|_{\bm{L}^4(\Omega)}
+ \|m(\varphi)\|_{L^\infty(\Omega)}\|\nabla \mu \|,
\end{equation}
so that
\begin{align}
\label{phit-2}
\| \partial_t \varphi\|_{L^2(0,T;(H^1(\Omega))')}\leq C_5.
\end{align}

Next, we prove the uniqueness of the weak solution. Let $(\varphi_1,\mu_1)$ and $(\varphi_2,\mu_2)$ be two weak solutions to problem \eqref{CH}--\eqref{bcic}, corresponding to the initial data $\varphi_{0,1}$ and $\varphi_{0,2}$ satisfying $\overline{\varphi_{0,1}}=\overline{\varphi_{0,2}}$.
We define the differences $\widetilde{\varphi}=\varphi_1-\varphi_2$, $\widetilde{\mu}=\mu_1-\mu_2$, which solve
\begin{align}
&\l \partial_t \widetilde{\varphi},\xi\r_{(H^1(\Omega))',H^1(\Omega)}
+ (m(\varphi_1)\nabla \widetilde{\mu}, \nabla \xi)
\notag\\
&\quad = - ( \vv\cdot \nabla \widetilde{\varphi},\xi)
-((m(\varphi_1)-m(\varphi_2))\nabla \mu_2, \nabla \xi),
\quad  \forall \, \xi \in H^1(\Omega), \ \text{a.e. in } (0,\infty),
\label{bphi-d}
\end{align}
where
\begin{align}
&\widetilde{\mu}=-\Delta \widetilde{\varphi} + \Psi'(\varphi_1)-\Psi'(\varphi_2) + (\beta'(\varphi_1)-\beta'(\varphi_2))\sigma.
\label{bmu-d}
\end{align}
Taking  $\xi= \mathcal{G}_{\varphi_1}\widetilde{\varphi}$ in \eqref{bphi-d}, we find
\begin{align}
&\l \partial_t \widetilde{\varphi},\mathcal{G}_{\varphi_1}\widetilde{\varphi} \r_{(H^1(\Omega))',H^1(\Omega)}
+ \|\nabla \widetilde{\varphi}\|^2 + (\Psi_0'(\varphi_1)-\Psi_0'(\varphi_2),\widetilde{\varphi})
\notag \\
&\quad = \theta_0 \|\widetilde{\varphi}\|^2
+ (\widetilde{\varphi}\vv,\nabla \mathcal{G}_{\varphi_1}\widetilde{\varphi}) -((m(\varphi_1)-m(\varphi_2))\nabla \mu_2, \nabla \mathcal{G}_{\varphi_1}\widetilde{\varphi})
\notag\\
&\qquad -((\beta'(\varphi_1)-\beta'(\varphi_2))\sigma,\widetilde{\varphi}).
\label{varp-dif}
\end{align}
The four terms on the right-hand side of \eqref{varp-dif} can be estimated as follows
\begin{align}
\theta_0 \|\widetilde{\varphi}\|^2 \leq \frac{1}{12}\|\nabla \widetilde{\varphi}\|^2+ C\|\nabla \mathcal{G}_{\varphi_1}\widetilde{\varphi}\|^2,
\notag
\end{align}
\begin{align}
\left|(\widetilde{\varphi}\vv,\nabla \mathcal{G}_{\varphi_1}\widetilde{\varphi})\right|
&\leq \|\widetilde{\varphi}\|_{L^4(\Omega)}\|\vv\|_{\bm{L}^4(\Omega)}\|\nabla \mathcal{G}_{\varphi_1}\widetilde{\varphi}\|
\notag \\
&\leq \frac{1}{12}\|\nabla \widetilde{\varphi}\|^2 + C\|\nabla \vv\|^2 \|\nabla \mathcal{G}_{\varphi_1}\widetilde{\varphi}\|^2,
\notag
\end{align}
\begin{align}
&\left|((m(\varphi_1)-m(\varphi_2))\nabla \mu_2, \nabla \mathcal{G}_{\varphi_1}\widetilde{\varphi})\right|
 \notag \\
&\quad \leq \frac{1}{12}\|\nabla \widetilde{\varphi}\|^2
+C\left(\|\nabla \mu_2\|^2+ \|\varphi_1\|_{H^2(\Omega)}^4 \right) \|\nabla \mathcal{G}_{\varphi_1}\widetilde{\varphi}\|^2,
\notag
\end{align}
\begin{align}
 \left|((\beta'(\varphi_1)-\beta'(\varphi_2))\sigma,\widetilde{\varphi})\right|
 & \leq C\|\sigma\|_{L^4(\Omega)}\|\widetilde{\varphi}\| \|\widetilde{\varphi}\|_{L^4(\Omega)}
 \notag\\
 &\leq C\|\sigma\|_{L^4(\Omega)}
 \|\nabla \widetilde{\varphi}\|^\frac54 \|\nabla \mathcal{G}_{\varphi_1}\widetilde{\varphi}\|^\frac34
 \notag\\
 &\leq \frac{1}{12}\|\nabla \widetilde{\varphi}\|^2 + C\left(1+\|\sigma\|_{H^1(\Omega)}^2 \right) \|\nabla \mathcal{G}_{\varphi_1}\widetilde{\varphi}\|^2,
 \notag
\end{align}
where the third estimate follows from \cite[(3.51)]{CGGG}. We now handle the term $\l \partial_t \widetilde{\varphi},\mathcal{G}_{\varphi_1}\widetilde{\varphi} \r_{(H^1(\Omega))',H^1(\Omega)}$ on the left-hand side. Recalling \cite[(3.16)]{CGGG}, we have
\begin{align}
& \l \partial_t \widetilde{\varphi},\mathcal{G}_{\varphi_1}\widetilde{\varphi} \r_{(H^1(\Omega))',H^1(\Omega)}
= \int_\Omega (\mathcal{G}_{\varphi_1} \partial_t \widetilde{\varphi})  \,  \widetilde{\varphi}  \, \d x
\notag \\
&\qquad = \ddt \frac{1}{2} (\mathcal{G}_{\varphi_1} \widetilde{\varphi}, \widetilde{\varphi})
+\frac12 \int_\Omega \nabla \mathcal{N} \partial_t \varphi_1 \cdot m''(\varphi_1) \nabla \varphi_1 \left| \nabla \mathcal{G}_{\varphi_1} \widetilde{\varphi} \right|^2 \, \d x
\notag \\
&\qquad\quad  +  \int_\Omega \nabla \mathcal{N} \partial_t \varphi_1 \cdot  m'(\varphi_1) \left( D^2  \mathcal{G}_{\varphi_1} \widetilde{\varphi} \nabla \mathcal{G}_{\varphi_1} \widetilde{\varphi} \right) \, \d x,
\label{IP-time}
\end{align}
almost everywhere in $(0,\infty)$. Here, $D^2 f$ denotes the Hessian of a given scalar function $f$. We note that
$$
(\mathcal{G}_{\varphi_1} \widetilde{\varphi}, \widetilde{\varphi})^\frac12
= \big\|\sqrt{m(\varphi_1)}\nabla \mathcal{G}_{\varphi_1} \widetilde{\varphi}\big\|
$$
is a norm in $ V^{-1}_{(0)}$, which is equivalent to
$ \|\nabla \mathcal{N} \widetilde{\varphi}\|$.
The second and third terms on the right-hand side of \eqref{IP-time} have been estimated in \cite[(3.47), (3.50)]{CGGG} such that
\begin{align}
& \left|\frac12 \int_\Omega \nabla \mathcal{N} \partial_t \varphi_1 \cdot m''(\varphi_1) \nabla \varphi_1 \left| \nabla \mathcal{G}_{\varphi_1} \widetilde{\varphi} \right|^2 \, \d x\right|
\notag\\
&\quad\leq \frac{1}{12} \| \nabla \widetilde{\varphi}\|^2
 + C \left( \|\partial_t \varphi_1\|_{V^{-1}_{(0)}}^2+
 \|\varphi_1\|_{H^2(\Omega)}^4 \right)\|\nabla \mathcal{G}_{\varphi_1} \widetilde{\varphi}\|^2,
 \notag
\end{align}
and
\begin{align}
&\left|\int_\Omega \nabla \mathcal{N} \partial_t \varphi_1 \cdot  m'(\varphi_1) \left( D^2  \mathcal{G}_{\varphi_1} \widetilde{\varphi} \nabla \mathcal{G}_{\varphi_1} \widetilde{\varphi} \right) \, \d x\right|
\notag \\
&\quad \leq \frac{1}{12} \|\nabla \widetilde{\varphi}\|^2
 + C \left( \|\partial_t \varphi_1\|_{ V^{-1}_{(0)}}^2 +
 \| \varphi_1\|_{H^2(\Omega)}^4\right)
\|\nabla \mathcal{G}_{\varphi_1} \widetilde{\varphi}\|^2,
\notag
\end{align}
with some minor adjustments in the coefficients here due to Young's inequality.  Recalling that $\| \nabla \mathcal{G}_{\varphi_1} \widetilde{\varphi}\|$ and $\|\sqrt{m(\varphi_1)} \nabla \mathcal{G}_{\varphi_1} \widetilde{\varphi}\|$ are equivalent norms, combining the above estimates and using the convexity of $\Psi_0$, we can deduce from \eqref{varp-dif} that
\begin{equation}
\ddt  \big\|\sqrt{m(\varphi_1)} \nabla \mathcal{G}_{\varphi_1} \widetilde{\varphi}\big\|^2
+  \|\nabla \widetilde{\varphi}\|^2
\leq h_1(t) \big\|\sqrt{m(\varphi_1)} \nabla \mathcal{G}_{\varphi_1} \widetilde{\varphi}\big\|^2,
\notag
\end{equation}
where
$$
h_1(\cdot)= C \Big(1+
\|\nabla \mu_2\|^2
+ \|\partial_t \varphi_1\|_{V^{-1}_{(0)}}^2 +
 \|\varphi_1\|_{H^2(\Omega)}^4 + \|\nabla \vv\|^2+ \|\sigma\|^2_{H^1(\Omega)}\Big) \in L^1(0,T), \quad \forall \, T>0.
$$
An application of Gronwall's lemma entails that
\begin{equation}
\big\|\sqrt{m(\varphi_1)} \nabla  \mathcal{G}_{\varphi_1} ( \varphi_1(t)-\varphi_2(t))\big\|^2
\leq \big\|\sqrt{m(\varphi_1)} \nabla  \mathcal{G}_{\varphi_1}  ( \varphi_{0,1}-\varphi_{0,2})\big\|^2
\mathrm{e}^{\int_0^t h(s)\, \d s}, \quad \forall \, t \geq 0.
\notag
\end{equation}
Hence, if $\varphi_{0,1}=\varphi_{0,2}$, then we get the uniqueness of a weak solution.

Finally, we note that the regularity of $(\varphi, \mu)$ allows us to test \eqref{e2} with $\xi = \mu-\beta'(\varphi)\sigma$. Integrating the resultant on $[0,t]$ gives the energy identity \eqref{EI}.

The proof of Proposition \ref{well-pos} is complete.
\end{proof}
\begin{remark}\rm
\label{rem-Linf}
The $L^\infty$-estimate \eqref{phi-Lif} is not available at the level of Galerkin approximation (along with a regularization of the singular potential $\Psi$). However, we can first prove the existence of a (unique) weak solution $(\varphi^n, \mu^n)$ under sufficiently smooth given data $(\vv^n,\sigma^n)$ that approximate $(\vv, \sigma)$ (cf. the estimate \eqref{es-vvmu} where \eqref{phi-Lif} was essentially used). The approximate solution $\varphi^n$ satisfies \eqref{phi-Lif} due to the singular nature of $\Psi$. Then the associated estimates independent of $n$ can be recovered from the energy identity \eqref{EI} for $\varphi^n$. Afterwards, we can pass to the limit as $n\to \infty$. By the compactness argument, we can obtain the existence of a global weak solution to the original problem.
\end{remark}

The following result presents the existence and uniqueness of a global strong solution under additional assumptions on $(\vv,\sigma)$.

\begin{proposition}
\label{CH-strong}
Let $\Omega$ be a bounded domain in $\mathbb{R}^2$ with a $C^3$ boundary. Assume that (H1), (H3) and (H4) are satisfied.
Let
\begin{align*}
&\vv \in L^\infty(0, \infty;\L^2_\sigma(\Omega))
\cap L^2(0,\infty;\H^1_{0,\sigma}(\Omega)),\\
&\sigma \in  L^\infty(0,T; L^q(\Omega))\cap L^2(0,T; H^1(\Omega))\cap H^1(0,T;L^2(\Omega))
\end{align*}
for some $q>2$ and any $T>0$, with the additional property $\sigma(\cdot,0)\in H^1(\Omega)$.
Then, for any initial datum $\varphi_0 \in H^2_N(\Omega)$ with
$\| \varphi_0\|_{L^\infty(\Omega)}\leq 1$, $\left|\overline{\varphi_0}\right|<1$ and $\mu_0:=-\Delta \varphi_0+\Psi'(\varphi_0) \in H^1(\Omega)$,  problem  \eqref{CH}--\eqref{bcic} admits a unique global strong solution in $\Omega\times [0,\infty)$, which satisfies
\begin{equation}
\label{REG}
\begin{split}
&\varphi \in L^\infty(0,T;W^{2,q}(\Omega)),
\\
&\partial_t \varphi \in L^2(0,T;H^1(\Omega)) \cap  L^\infty(0,T; (H^1(\Omega))'),
\\
&\varphi \in L^{\infty}(\Omega\times (0,\infty)) \ \text{ with } \ |\varphi(x,t)|<1
\ \text{a.e. in } \Omega\times (0,\infty),
\\
&\mu \in L^{\infty}(0,T; H^1(\Omega))\cap L^4(0,T;H^2(\Omega)),
\quad\Psi_0'(\varphi) \in L^\infty(0,T; L^q(\Omega)),
\end{split}
\end{equation}
for any $T>0$.
The strong solution satisfies \eqref{CH} almost everywhere in $\Omega \times (0,\infty)$, the boundary conditions $\partial_\n \varphi=\partial_\n \mu=0$ almost everywhere on
$\partial\Omega\times(0,\infty)$, and the initial condition $\varphi|_{t=0}=\varphi_0$ almost everywhere in $\Omega$.
%
%
In addition, if $\sigma \in L^\infty(0,T; H^1(\Omega))$, then we have $\varphi\in L^\infty(0,T;H^3(\Omega))$ and there exists a constant $\widetilde{\delta}_1=\widetilde{\delta}_1(T)\in (0,1)$ such that
\be
\|\varphi(t)\|_{L^{\infty}(\Omega)} \leq 1-\widetilde{\delta}_1,\quad \forall\, t \in[0,T],\label{sep1-ch}
\ee
where $\widetilde{\delta}_1$ depends on  $\|\varphi_0\|_{H^2(\Omega)}$, $\|\nabla (-\Delta \varphi_0+\Psi'(\varphi_0))\|$, $\|\nabla \vv\|_{L^2(0,T;\L^2(\Omega))}$, $\|\sigma\|_{L^\infty(0,T;H^1(\Omega))}$, $\overline{\varphi_0}$, coefficients of the system, $\Omega$ and $T$.
\end{proposition}
\begin{proof}
We extend the arguments devised in \cite[Theorem 2.4]{AGG2023} (with a constant mobility and a divergence-free drift, without the conservative forcing term) and \cite[Section 4]{CGGG} (with a variable mobility, without the divergence-free drift and the conservative forcing term).

Multiplying \eqref{CH}$_1$ by $\partial_t \mu$ and integrating over $\Omega$, we have
$$
\frac12 \ddt \int_\Omega m(\varphi) |\nabla \mu |^2\,\mathrm{d}x
+ \int_{\Omega} \partial_t \varphi \, \partial_t \mu \,\d x + \int_{\Omega} (\vv \cdot \nabla \varphi) \, \partial_t \mu \, \d x
=\frac12\int_\Omega m'(\varphi)\partial_t\varphi|\nabla \mu|^2\,\mathrm{d}x.
$$
By definition of $\mu$, we observe that
\begin{align}
\int_{\Omega} \partial_t \varphi \, \partial_t \mu \,\d x
&= \int_{\Omega} |\nabla \partial_t \varphi|^2 \, \d x + \int_{\Omega}\Psi_0''(\varphi) |\partial_t \varphi |^2 \, \d x- \theta_0 \int_{\Omega}  |\partial_t \varphi |^2 \, \d x
+ \int_\Omega \partial_t (\beta'(\varphi)\sigma)  \partial_t \varphi \,\mathrm{d}x.
\notag
\end{align}
Exploiting the incompressibility and the no-slip boundary condition of the velocity field, after integration by parts, we find
\begin{align*}
\int_{\Omega} (\vv \cdot \nabla \varphi) \, \partial_t \mu \, \d x
&= \int_{\Omega} (\vv \cdot \nabla \varphi)  \left(-\Delta \partial_t \varphi +\Psi_0''(\varphi)\partial_t \varphi -\theta_0 \partial_t  \varphi + \partial_t (\beta'(\varphi)\sigma)
\right) \, \d x\\
&=  \int_{\Omega} \nabla \left( \vv \cdot \nabla \varphi\right) \cdot \nabla \partial_t \varphi \, \d x
- \underbrace{\int_{\partial \Omega} (\vv \cdot \nabla \varphi) (\nabla \partial_t \varphi \cdot \n) \, \d S}_{=0}
\\
&\quad  + \int_{\Omega} \vv \cdot \left(\Psi_0''(\varphi) \nabla \varphi \right) \partial_t \varphi \, \d x
- \theta_0 \int_{\Omega} \vv \cdot \nabla \varphi \, \partial_t \varphi \, \d x
\\
&\quad
+ \int_\Omega \partial_t (\beta'(\varphi)\sigma)  \vv \cdot \nabla \varphi  \,\mathrm{d}x
\\
&=
 \int_{\Omega} \left( \nabla \vv^T \nabla \varphi \right) \cdot \nabla \partial_t \varphi \, \d x
 +
\int_{\Omega} \left( \nabla^2 \varphi \, \vv \right) \cdot \nabla \partial_t \varphi \, \d x
\\
&\quad -
\int_{\Omega} (\vv \cdot \nabla \partial_t \varphi) \Psi_0'(\varphi) \, \d x
+ \theta_0 \int_{\Omega} (\vv \cdot \nabla \partial_t \varphi) \varphi \, \d x
+  \int_\Omega \partial_t (\beta'(\varphi)\sigma)  \vv \cdot \nabla \varphi  \,\mathrm{d}x.
\end{align*}
Then, we arrive at
\begin{align}
& \frac12 \ddt \int_\Omega m(\varphi) |\nabla \mu |^2\,\mathrm{d}x
+ \int_{\Omega} |\nabla \partial_t \varphi|^2 \, \d x + \int_{\Omega}\Psi_0''(\varphi) |\partial_t \varphi |^2 \, \d x
\notag \\
&\quad =\theta_0 \int_{\Omega}  |\partial_t \varphi |^2 \, \d x
- \int_{\Omega} \left( \nabla \vv^T \nabla \varphi \right) \cdot \nabla \partial_t \varphi \, \d x -\int_{\Omega} \left( \nabla^2 \varphi \, \vv \right) \cdot \nabla \partial_t \varphi \, \d x
\notag \\
&\qquad +
\int_{\Omega} (\vv \cdot \nabla \partial_t \varphi) \Psi_0'(\varphi) \, \d x
- \theta_0 \int_{\Omega} (\vv \cdot \nabla \partial_t \varphi) \varphi \, \d x
-  \int_\Omega \partial_t (\beta'(\varphi)\sigma)(\partial_t \varphi+ \vv \cdot \nabla \varphi)\,\mathrm{d}x
\notag \\
&\qquad + \frac12\int_\Omega m'(\varphi)\partial_t\varphi|\nabla \mu|^2\,\mathrm{d}x.
\label{diff-eq}
\end{align}
Combining \eqref{mu-H1e}, \eqref{phi2-p} with the assumption $\sigma \in L^\infty(0,T;L^q(\Omega))$ for
some given $q>2$ and any $T>0$, we observe that
\begin{equation}
\label{phi2p}
\| \varphi\|_{W^{2,q}(\Omega)}+ \|\Psi_0'(\varphi)\|_{L^q(\Omega)}
\leq  C_q \left( 1+\| \nabla \mu \|\right), \quad \text{for a.a. } t \in [0,T].
\end{equation}
Now we estimate the terms on the right-hand side of \eqref{diff-eq}. Recalling that $\partial_t \varphi$ is mean-free, thanks to the Poincar\'{e}--Wirtinger inequality, Young's inequality, \eqref{EE-3} and \eqref{phit-est}, we find
\begin{align*}
|\theta_0| \int_{\Omega}  |\partial_t \varphi |^2 \, \d x
&\leq C \| \nabla \partial_t \varphi\| \| \partial_t \varphi\|_{(H^1(\Omega))'}\\
&\leq \frac{1}{16} \| \nabla \partial_t \varphi\|^2
+ C  \| \nabla \mu\|^2+ C\| \nabla \vv\|^2.
\end{align*}
Exploiting \eqref{phi2p} and the Sobolev embedding theorem (with $q>2$), we infer that (cf.  \cite{AGG2023})
\begin{align*}
\left| \int_{\Omega} \left( \nabla \vv^T \nabla \varphi \right) \cdot \nabla \partial_t \varphi \, \d x \right|
&\leq \| \nabla \vv\| \| \nabla \varphi\|_{\bm{L}^\infty(\Omega)} \| \nabla \partial_t \varphi\|\\
& \leq \frac{1}{16} \| \nabla \partial_t \varphi\|^2 +
C \| \nabla \vv\|^2 \| \varphi \|_{W^{2,q}(\Omega)}^2\\
& \leq \frac{1}{16} \| \nabla \partial_t \varphi\|^2 +
C \| \nabla \vv\|^2  \| \nabla \mu\|^2 +C \| \nabla \vv\|^2,
\end{align*}
\begin{align*}
\left|
\int_{\Omega} \left( \nabla^2 \varphi \, \vv \right) \cdot \nabla \partial_t \varphi \, \d x \right|
&\leq \| \varphi\|_{W^{2,q}(\Omega)} \| \vv\|_{\bm{L}^{\frac{2q}{q-2}}(\Omega)} \| \nabla \partial_t \varphi\| \\
& \leq  C\left( 1+  \| \nabla \mu \|  \right) \| \nabla \vv\|  \| \nabla \partial_t \varphi\| \\
&\leq \frac{1}{16} \| \nabla \partial_t \varphi\| ^2  +
C \| \nabla \vv\| ^2   \| \nabla \mu \| ^2+ C \| \nabla \vv\| ^2,
\end{align*}
\begin{align*}
\left| \int_{\Omega} (\vv \cdot \nabla \partial_t \varphi) \Psi_0'(\varphi) \, \d x \right|
&\leq \| \vv\|_{\bm{L}^\frac{2q}{q-2}(\Omega)} \| \nabla \partial_t \varphi\|  \|\Psi_0'(\varphi)\|_{L^q(\Omega)}
\\
& \le\frac{1}{16} \| \nabla \partial_t \varphi\|^2
+ C \| \nabla \vv\|^2  \| \nabla \mu \|^2+ C \| \nabla \vv\|^2.
\end{align*}
Next, we obtain from \eqref{EE-3}, \eqref{phit-est} and \eqref{phi2p} that
\begin{align}
&\left|  \int_\Omega \partial_t (\beta'(\varphi)\sigma)(\partial_t \varphi+ \vv \cdot \nabla \varphi)\,\mathrm{d}x\right|
\notag \\
&\quad \leq C\|\beta'(\varphi)\|_{L^\infty(\Omega)}\|\partial_t \sigma \|
 (\|\partial_t \varphi\| +\| \vv \cdot \nabla \varphi\|)
\notag\\
&\qquad +C\|\beta''(\varphi)\|_{L^\infty(\Omega)} \|\partial_t\varphi\|_{L^\frac{2q}{q-2}(\Omega)} \|\sigma\|_{L^q(\Omega)}
(\|\partial_t \varphi\| +\| \vv \cdot \nabla \varphi\|)
\notag\\
&\quad \leq \frac{1}{32}\|\nabla \partial_t \varphi\|^2+ C(1+\|\sigma\|_{L^q(\Omega)}^2)\|\partial_t\varphi\|^2
+ C\|\partial_t\sigma\|^2
\notag\\
&\qquad + C(1+\|\sigma\|_{L^q(\Omega)}^2) \|\vv\|_{\bm{L}^4(\Omega)}^2\|\nabla \varphi\|_{\bm{L}^4(\Omega)}^2
\notag\\
&\quad \leq \frac{1}{16}\|\nabla \partial_t \varphi\|^2
+ C \left( \| \vv\|_{\bm{L}^4(\Omega)}^2+ \| \nabla \mu\|^2 \right)
+ C\|\partial_t \sigma\|^2
+ C\|\nabla \vv\|^2\|\varphi\|_{W^{1,4}(\Omega)}^2
\notag\\
&\quad \leq \frac{1}{16}\|\nabla \partial_t \varphi\|^2
+ C\|\nabla \vv\|^2 \|\nabla \mu\|^2
+ C(\|\nabla \vv\|^2+\|\nabla \mu\|^2+\|\partial_t \sigma\|^2),
\notag
\end{align}
while using \eqref{phi-Lif}, we get
\begin{align*}
\left| \theta_0 \int_{\Omega} (\vv \cdot \nabla \partial_t \varphi) \varphi \, \d x\right|
&\leq C \| \vv\| \| \nabla \partial_t \varphi\| \| \varphi\|_{L^\infty(\Omega)} \\
&\leq \frac{1}{16} \| \nabla \partial_t \varphi\|^2 +
C \| \vv\|^2.
\notag
\end{align*}
The positive constants $C$ in all the above estimates may depend on the parameters of the system, such as $\theta_0$, $\theta$, $\Omega$, $\overline{\varphi_0}$, and $T>0$ (note that we have used \eqref{EE-3} and  $\|\sigma\|_{L^\infty(0,T;L^q(\Omega))}$).
Finally, we treat the last term on the right-hand side of \eqref{diff-eq}. Observe that
\begin{align}
\left| \frac12\int_\Omega m'(\varphi)\partial_t\varphi|\nabla \mu|^2\,\mathrm{d}x\right|
& \leq C\|\partial_t\varphi\|\|\nabla \mu\|_{\bm{L}^4(\Omega)}^2
\notag\\
&\leq C \| \nabla \partial_t \varphi\|^\frac12 \| \partial_t \varphi\|_{(H^1(\Omega))'}^\frac12
\|\nabla \mu\| \|\nabla \mu\|_{H^1(\Omega)}.
\label{es-lasta}
\end{align}
Thus, it remains to estimate $\|\nabla \mu\|_{H^1(\Omega)}$.
Using the idea in \cite{CGGG}, we rewrite \eqref{CH}$_1$ as
\begin{align}
\mu-\overline{\mu}= - \mathcal{G}_\varphi (\partial_t\varphi+\bm{v}\cdot\nabla \varphi),
\label{CH-b}
\end{align}
where $\overline{\mu}= \overline{\Psi_0'(\varphi)}- \theta_0 \overline{\varphi}+  \overline{\beta'(\varphi)\sigma}$. Applying the elliptic estimate \eqref{es-Ga-H2}, we find
\begin{align}
\|\mu-\overline{\mu}\|_{H^2(\Omega)}
&
\leq C\big(\|\nabla \varphi\|\|\varphi\|_{H^2(\Omega)}\|\nabla \mathcal{G}_\varphi (\partial_t\varphi+\bm{v}\cdot\nabla \varphi)\|+ \| \partial_t\varphi+\bm{v}\cdot\nabla \varphi\|\big)
\notag\\
&\leq C\|\nabla \varphi\| \|\varphi\|_{H^2(\Omega)}\|\nabla\mu\|
+ C\|\partial_t\varphi\|_{(H^1(\Omega))'}^\frac12\|\nabla \partial_t\varphi\|^\frac12 \notag\\
&\quad +C\|\bm{v}\|_{\bm{L}^4(\Omega)}\|\nabla \varphi\|_{\bm{L}^4(\Omega)}.
\label{mu-H2a}
\end{align}
Using \eqref{EE-3}, \eqref{phit-est}, \eqref{mu-H2a}, and the assumption $\vv \in L^\infty(0, \infty;\L^2_\sigma(\Omega))$,
we infer from \eqref{es-lasta} that
\begin{align}
& \left| \frac12\int_\Omega m'(\varphi)\partial_t\varphi|\nabla \mu|^2\,\mathrm{d}x\right|
\notag \\
& \quad \leq C\|\nabla \partial_t\varphi\|^\frac12 (\|\bm{v}\|_{\bm{L}^4(\Omega)}+\|\nabla \mu\|)^\frac12 \|\nabla \mu\|^2 \|\varphi\|_{H^2(\Omega)}  \notag\\
&\qquad +C \|\nabla \partial_t\varphi\|  (\|\bm{v}\|_{\bm{L}^4(\Omega)}+\|\nabla \mu\|)  \|\nabla \mu\| \notag \\
&\qquad +C \|\nabla \partial_t\varphi\|^\frac12 (\|\nabla \bm{v}\|^\frac12 +\|\nabla \mu\|)^\frac12 \|\nabla \mu\| \|\nabla \bm{v}\|^\frac12 \|\varphi\|_{H^2(\Omega)}
\notag\\
&\quad \leq \frac{1}{16} \|\nabla \partial_t\varphi\|^2 + C(1+\|\nabla \bm{v}\|^2+ \|\varphi\|_{H^2(\Omega)}^4+\|\nabla \mu\|^2)\|\nabla \mu\|^2 + C\|\varphi\|_{H^2(\Omega)}^4.
\end{align}
Collecting the above estimates, we can deduce from \eqref{Dephi-L2b} and \eqref{diff-eq} that
\begin{equation}
 \label{DI-alpha}
\begin{split}
 &  \ddt \big\|\sqrt{m(\varphi)}\nabla \mu\big\|^2
 +  \|\nabla \partial_t \varphi\|^2
   \leq  h_2(t)  \big\|\sqrt{m(\varphi)}\nabla \mu\big\|^2 + h_3(t),
\end{split}
\end{equation}
for almost every $t \in [0,T]$, with $T>0$ being arbitrary, where
\begin{align*}
& h_2(\cdot) = C\big(1+ \|\sqrt{m(\varphi)}\nabla \mu\|^2+\|\nabla \vv\|^2+\|\nabla \sigma\|^2\big)\in L^1(0,T),
\\
& h_3(\cdot) = C\big(1+\|\nabla \bm{v}\|^2 +\|\nabla \sigma\|^2 +\|\partial_t\sigma\|^2\big)\in L^1(0,T).
\end{align*}
An application of Gronwall's lemma entails that
\begin{equation}
\label{nmu-H1}
\begin{split}
\sup_{0\leq t \leq T} \|\nabla \mu(t)\|^2
&\leq \frac{1}{m_*}
  \left(m^*\|\nabla \mu(0)\|^2 + \int_0^T  h_3(s)\,\mathrm{d}s\right)e^{\int_0^T h_2(s)\, \d s},
\end{split}
\end{equation}
where we have used (H3). Noticing that
\begin{align}
\| \nabla \mu(0)\|_{L^2(\Omega)}
& = \| \nabla ( \mu_0 + \beta'(\varphi_0) \sigma(0))\|
\notag\\
&\leq  \| \nabla \mu_0\|
+   \|\beta'(\varphi_0)\|_{L^\infty(\Omega)}\|\nabla  \sigma(0) \|
 +\|\beta''(\varphi_0)\|_{L^\infty(\Omega)}\|\nabla \varphi_0\|_{\bm{L}^4(\Omega)}\|\sigma(0)\|_{L^4(\Omega)}
\notag\\
&\leq \| \nabla \mu_0\| + C(\|\varphi_0\|_{H^2(\Omega)}+1)\|\sigma(0)\|_{H^1(\Omega)}
\notag\\
&\leq \| \nabla \mu_0\| + C(1+\|\nabla \mu_0\|+\|\nabla \sigma(0)\|)^\frac12\| \sigma(0)\|_{H^1(\Omega)},
\notag
\end{align}
where the constant $C>0$ depends on $\|\sigma_0\|$ and $\|\varphi_0\|_{H^1(\Omega)}$.
Therefore, we find
\begin{align}
\sup_{0\leq t \leq T}
\| \nabla \mu(t)\|_{L^2(\Omega)}^2
&\leq C
\left( \| \nabla \mu_0\|^2 + C(1+\|\nabla \mu_0\|+\|\nabla \sigma(0)\|) \| \sigma(0)\|_{H^1(\Omega)}^2 \right)e^{\int_0^T h_2(s)\, \d s}
\notag \\
& \quad
+  C e^{\int_0^T h_2(s)\, \d s}\int_0^T  h_3(s)\,\mathrm{d}s
\notag \\
&=: \mathcal{F}(T).
\label{nmu-H1-2}
\end{align}
Integrating \eqref{DI-alpha} on $[0,T]$, and exploiting \eqref{nmu-H1-2}, we further obtain
\begin{align}
\int_0^T \| \nabla \partial_t \varphi(s)\|^2 \, \d s
&\leq  C
\left( \| \nabla \mu_0\|^2 + C(1+\|\nabla \mu_0\|+\|\nabla \sigma(0)\|) \| \sigma(0)\|_{H^1(\Omega)}^2 \right)
\notag \\
&\quad +C \mathcal{F}(T) \int_0^T h_2(s)\, \d s + \int_0^T  h_3(s)\,\d s.
\label{E2-alpha}
\end{align}
Thus, we can conclude
\begin{equation}
\label{EST1}
\| \nabla \mu\|_{L^\infty(0,T; \bm{L}^2(\Omega))} \leq C,
\quad
\| \nabla \partial_t \varphi\|_{L^2(0,T;\bm{L}^2(\Omega))}\leq C,
\end{equation}
for any $T>0$. Then it follows from \eqref{mu-H1e}, \eqref{phi2p} and \eqref{EST1}  that
\begin{align}
& \|  \mu\|_{L^\infty(0,T; H^1(\Omega))} \leq C,
 \label{EST2-a}\\
& \| \varphi\|_{L^\infty(0,T; W^{2,q}(\Omega))}\leq C,
\quad
\|\Psi_0'(\varphi)\|_{L^\infty(0,T;L^q(\Omega))} \leq C.
\label{EST2-b}
\end{align}
By a comparison in \eqref{phit-est}, we also obtain
\begin{align}
\|\partial_t \varphi\|_{L^\infty(0,T; (H^1(\Omega))')}\leq C.
\label{EST2-c}
\end{align}
As a consequence, from \eqref{mu-H2a} and \eqref{EST2-a}--\eqref{EST2-c}, we can deduce that
$$
\int_0^T \|\mu(t)-\overline{\mu(t)}\|_{H^2(\Omega)}^4\,\mathrm{d}t
\leq CT+ C \int_0^T\big(\|\nabla \partial_t(t)\varphi\|^2 + \|\nabla \bm{v}(t)\|^2\big)\,\mathrm{d}t\leq C,
$$
that is, $\mu\in L^4(0,T;H^2(\Omega))$ for any $T>0$.

Consider the elliptic problem \eqref{vphi-ellip}. Using \eqref{es-beta-sigH1}, \eqref{Dephi-L2b}, \eqref{EST2-a}, \eqref{EST2-b} and (H5), we can apply the same argument as in \cite[Section 3]{GGG2023} to conclude \eqref{sep1-ch}.
This strict separation property holds for all $t\in [0,T]$, because of the continuity $\varphi\in C([0,T]; C(\overline{\Omega}))$ due to \eqref{EST1}, \eqref{EST2-b}, the Aubin--Lions--Simon lemma and the Sobolev embedding theorem.
Moreover, it enables us to gain more regularity of the strong solution. Since $\Psi\in C^2([-1+\delta_1,\, 1-\delta_1])$, we can deduce from \eqref{EST2-b} that $\Psi' (\varphi )\in L^\infty (0 ,T ;H^1(\Omega ))$. Applying the classical elliptic estimate to the elliptic equation $\eqref{CH}_2$, we further get
\begin{equation}
	\varphi \in L^\infty (0 ,T ;H^3 (\Omega )).
	\label{phih3}
\end{equation}

The proof of Proposition \ref{CH-strong} is complete.
\end{proof}

\subsection{Inhomogeneous Stokes system with forcing}
Let $\uu$, $\varphi$, $\mu$, $\sigma$ be four given functions with suitable regularity properties. We consider the following Stokes system with a forcing term (cf. Remark \ref{rem:NS-2})
\begin{equation}
 \label{NS-for}
\begin{cases}
\rho(\varphi)\partial_t \vv +  \big(( \rho(\varphi) \uu+ \J) \cdot\nabla \big)\vv - \div \big( 2\nu(\varphi) D \vv \big) + \nabla P = \big(\mu-\beta'(\varphi)  \sigma\big) \nabla \varphi,\\
\qquad \text{with} \ \ \J= -\rho'(\varphi) m(\varphi) \nabla \mu,\quad \rho(\varphi)= \widetilde{\rho}_1 \dfrac{1-\varphi}{2}+ \widetilde{\rho}_2 \dfrac{1+\varphi}{2},
\\
\div \, \vv=0,
\end{cases}
\end{equation}
in $\Omega \times (0,\infty)$, subject to the boundary and initial conditions
\begin{equation}
\label{NS-for-bic}
\begin{cases}
\vv=\mathbf{0},  \quad &\text{on }  \partial \Omega \times (0,\infty),\\
\vv|_{t=0}=\vv_0, \quad &\text{in } \Omega.
\end{cases}
\end{equation}

\begin{proposition}
\label{exe-NS-for}
Let $\Omega$ be a bounded domain in $\mathbb{R}^2$ with a $C^3$ boundary. Suppose that assumptions (H1)--(H4) are satisfied. Let
\begin{align*}
&\uu \in L^\infty(0,\infty;\L^2_{\sigma} (\Omega))\cap L^2(0,\infty;\H^1_{0,\sigma}(\Omega)),
\\
&\sigma \in L^\infty(0,\infty;H^1(\Omega)) \cap L^2_{\mathrm{uloc}}([0,\infty); H^2(\Omega))\cap H^1_{\mathrm{uloc}}([0,\infty);L^2(\Omega)).
\end{align*}
Moreover, we assume that $(\varphi,\mu)$ is determined by Proposition \ref{CH-strong} with the functions $(\uu,\sigma)$ given above as drift and conservative forcing terms.\smallskip

(1) For any initial datum $\vv_0 \in \L^2_{\sigma}(\Omega)$,  problem  \eqref{NS-for}--\eqref{NS-for-bic} admits a global weak solution $\vv$ in $\Omega\times [0,\infty)$, which satisfies
$$
\vv \in L^\infty(0,T; \L^2_{\sigma}(\Omega))
\cap L^2(0,T; \H^1_{0,\sigma}(\Omega)),\quad \bm{P}\partial_t(\rho(\varphi)\vv)\in L^2(0,T; (\H^1_{0,\sigma}(\Omega))'),
$$
for any $T>0$. The solution $\bm{v}$ satisfies
\begin{align}
&\l\partial_t(\rho(\varphi) \vv), \bm{\zeta} \r_{(\H^1_{0,\sigma}(\Omega))',\H^1_{0,\sigma}(\Omega))}
-(  \rho(\varphi) \uu\otimes \vv, \nabla \bm{\zeta} )
-( \vv \otimes \J, \nabla \bm{\zeta} )
+ \big( 2\nu(\varphi) D \vv, D \bm{\zeta} \big)
\notag\\
&\quad
=  \big( (\mu-\beta'(\varphi) \sigma) \nabla \varphi, \bm{\zeta} \big),
\quad \forall\,\bm{\zeta}\in \H^1_{0,\sigma}(\Omega),
\notag
\end{align}
almost everywhere in $\Omega \times (0,\infty)$, the boundary condition $\vv=\mathbf{0}$ almost everywhere on $\partial \Omega \times (0,\infty)$ and the initial condition $\vv|_{t=0}=\vv_0$
 almost everywhere in $\Omega$.
 \smallskip

(2) For any initial datum $\vv_0 \in \H^1_{0,\sigma}(\Omega)$,  problem  \eqref{NS-for}--\eqref{NS-for-bic} admits a unique global strong solution $\vv$ in $\Omega\times [0,\infty)$, which satisfies
$$
\vv \in L^\infty(0,T; \H^1_{0,\sigma}(\Omega))
\cap L^2(0,T; \H^2(\Omega))\cap H^1(0,T; \L_\sigma^2(\Omega)),
$$
for any $T>0$.
The solution $\bm{v}$ satisfies the system \eqref{NS-for} almost everywhere in $\Omega \times (0,\infty)$, the boundary condition $\vv=\mathbf{0}$ almost everywhere on $\partial \Omega \times (0,\infty)$ and the initial condition $\vv|_{t=0}=\vv_0$ almost everywhere in $\Omega$.
\end{proposition}
\begin{proof}
\textbf{Existence.} The proof for the existence part is based on a Faedo--Galerkin approximation scheme similar to that in \cite{Gior2021}. In the following, we only perform formal \textit{a priori} estimates.

Testing the first equation in \eqref{NS-for} by $\vv$,  using \eqref{CH} and integration by parts, we get
\begin{align}
&\frac12\frac{\d}{\d t} \int_\Omega  \rho (\varphi)|\vv|^2\,\d x
+\int_\Omega  2\nu(\varphi) |D\bm{v}|^2\,\d x
=-\int_\Omega \varphi \nabla (\mu -\beta'(\varphi) \sigma)\cdot \vv\, \d x.
\label{BEL-NS-1}
\end{align}
In view of \eqref{es-beta-sigH1}, it holds
\begin{align}
\left|-\int_\Omega \varphi \nabla (\mu -\beta'(\varphi) \sigma)\cdot \vv\, \d x\right|
& \leq \|\varphi\|_{L^\infty(\Omega)}(\|\nabla \mu\|+ \|\nabla (\beta'(\varphi)\sigma)\|) \|\vv\|
\notag\\
& \leq C(\|\nabla \mu\|+\|\varphi\|_{H^2(\Omega)}+ \|\sigma\|_{H^1(\Omega)})\|\vv\|
\notag \\
&\leq \frac{\nu_*}{2}\|\nabla \vv\|^2
+ C(\|\nabla \mu\|^2+\|\varphi\|_{H^2(\Omega)}^2+ \|\sigma\|_{H^1(\Omega)}^2), \notag
\end{align}
where the constant $C>0$ depends on $\|\varphi\|_{L^\infty(0,T;H^1(\Omega))}$, $\|\sigma\|_{L^\infty(0,T;L^2(\Omega))}$.
Thus, an application of Gronwall's lemma to \eqref{BEL-NS-1} yields
\begin{align}
\|\vv(t)\|^2+\int_0^t\|\nabla \vv(s)\|^2\,\d s\leq C,
\quad \forall\, t\in [0,T],
\label{NS-low}
\end{align}
where $C>0$ depends on $\|\vv_0\|$, $T$, $\nu_*$, $\rho_*$, $\rho^*$, $\|\mu\|_{L^\infty(0,T;H^1(\Omega))}$, $\|\varphi\|_{L^\infty(0,T;H^2(\Omega))}$, $\|\sigma\|_{L^\infty(0,T;H^1(\Omega))}$.
Besides, from \eqref{NS-for} and the regularity properties of $(\uu, \varphi, \mu, \sigma)$, we can verify that
\begin{align}
\|\bm{P}\partial_t(\rho(\varphi)\vv)\|_{L^2(0,T;(\H^1_{0,\sigma}(\Omega))')}\leq C.
\label{NS-low2}
\end{align}

We proceed to derive higher-order estimates. Testing \eqref{NS-for} by $\partial_t\vv$, we obtain
\begin{align}
&\frac{\d}{\d t}\int_\Omega \nu(\varphi)|D\vv|^2\,\d x + \int_\Omega \rho(\varphi)|\partial_t \vv|^2\,\d x
\notag \\
&\quad = -\int_\Omega ( \rho(\varphi)\uu \cdot\nabla)\vv\cdot \partial_t \vv\,\d x
-\int_\Omega (  \J  \cdot\nabla )\vv\cdot \partial_t \vv\,\d x
\notag\\
&\qquad
+ 2\int_\Omega \nu'(\varphi)\partial_t\varphi|D\vv|^2\,\d x
-\int_\Omega \varphi \nabla (\mu -\beta'(\varphi) \sigma)\cdot \partial_t \vv\, \d x.
\label{NS-high1}
\end{align}
Using the Ladyzhenskaya inequality and Korn's inequality, we have
\begin{align}
\left|\int_\Omega ( \rho(\varphi)\uu \cdot\nabla)\vv\cdot \partial_t \vv\,\d x \right|
&\leq  \rho^* \|\uu\|_{\L^4(\Omega)}\|\nabla\vv\|_{\L^4(\Omega)}\|\partial_t\vv\|
\notag\\
&\leq C \|\uu\|^\frac12\|\nabla \uu\|^\frac12\|\nabla \vv\|^\frac12\|\vv\|_{\H^2(\Omega)}^\frac12
\|\partial_t\vv\|
\notag\\
&\leq \frac{\rho_*}{8}\|\partial_t\vv\|^2 + \eta \|\vv\|_{\H^2(\Omega)}^2+ C_\eta\|\nabla \uu\|^2\|\nabla \vv\|^2.
\notag
\end{align}
In a similar manner, we get
\begin{align}
\left|\int_\Omega (  \J  \cdot\nabla )\vv\cdot \partial_t \vv\,\d x \right|
&\leq  \|\rho'(\varphi)\|_{L^\infty(\Omega)}\|m(\varphi)\|_{L^\infty(\Omega)}\|\nabla \mu\|_{\L^4(\Omega)}\|\nabla \vv\|_{\L^4(\Omega)}\|\partial_t\vv\|
\notag\\
&\leq C\|\nabla \mu\|^\frac12 \|\mu\|_{H^2(\Omega)}^\frac12\|\nabla \vv\|^\frac12\|\vv\|_{\H^2(\Omega)}^\frac12
\|\partial_t\vv\|
\notag\\
&\leq \frac{\rho_*}{8}\|\partial_t\vv\|^2
+ \eta \|\vv\|_{\H^2(\Omega)}^2+ C_\eta\|\mu\|_{H^2(\Omega)}^2\|\nabla \vv\|^2,
\notag
\end{align}
\begin{align}
\left|2\int_\Omega \nu'(\varphi)\partial_t\varphi|D\vv|^2\,\d x\right|
&\leq \|\nu'(\varphi)\|_{L^\infty(\Omega)}\|\partial_t\varphi\|_{\L^4(\Omega)}\|D\vv\|_{\L^4(\Omega)}\|D\vv\|
\notag\\
&\leq \eta \|\vv\|_{\H^2(\Omega)}^2+ C_\eta\|\nabla \partial_t\varphi\|^2\|\nabla \vv\|^2,
\notag
\end{align}
\begin{align}
   \left| \int_\Omega \varphi \nabla (\mu -\beta'(\varphi) \sigma)\cdot \partial_t \vv\, \d x\right|
   &\leq \|\varphi\|_{L^\infty(\Omega)}(\|\nabla \mu\|+\|\nabla (\beta'(\varphi) \sigma)\|)\|\partial_t \vv\|
   \notag\\
   & \leq C(\|\nabla \mu\|+\|\varphi\|_{H^2(\Omega)} + \|\sigma\|_{H^1(\Omega)})\|\partial_t \vv\|
\notag\\
&\leq \frac{\rho_*}{8}\|\partial_t\vv\|^2  + C(\|\nabla \mu\|^2
+\|\varphi\|_{H^2(\Omega)}^2+ \|\sigma\|_{H^1(\Omega)}^2).
\notag
\end{align}
Here, the small constant $\eta\in (0,1)$ will be determined later.
Exploiting the regularity theory of the Stokes equation with concentration-dependent viscosity
(see, e.g., \cite[Lemma 4]{Abels2009}) and using Young's inequality, we have
\begin{align}
\|\vv\|_{\H^2(\Omega)}
&\leq C\|\partial_t\vv\|+ C \|\rho(\varphi)\|_{L^\infty(\Omega)}\|\uu\|_{\L^4(\Omega)}\|\nabla \vv\|_{\L^4(\Omega)}
\notag\\
&\quad + C\|\rho'(\varphi)\|_{L^\infty(\Omega)}\|m(\varphi)\|_{L^\infty(\Omega)}\|\nabla \mu\|_{\L^4(\Omega)}\|\nabla \vv\|_{\L^4(\Omega)}
\notag\\
&\quad + C\|\varphi\|_{L^\infty(\Omega)}(\|\nabla \mu\|+\|\nabla (\beta'(\varphi)\sigma)\|),
\notag
\end{align}
which combined with the Ladyzhenskaya inequality and  Young's inequality yields
\begin{align}
\|\vv\|_{\H^2(\Omega)}^2
&\leq C\|\partial_t\vv\|^2
+ C\|\nabla \uu\|^2\|\nabla \vv\|^2 + C \|\mu\|_{H^2(\Omega)}^2\|\nabla \vv\|^2
\notag \\
&\quad + C(\|\nabla \mu\|^2+\|\varphi\|_{H^2(\Omega)}^2+ \|\sigma\|_{H^1(\Omega)}^2).
\label{NS-es-H2}
\end{align}
Combining the above estimates and taking $\eta>0$ sufficiently small, we can deduce from
\eqref{NS-high1} that
\begin{align}
&\frac{\d}{\d t}\int_\Omega \nu(\varphi)|D\vv|^2\,\d x + \frac{\rho_*}{2}\int_\Omega |\partial_t \vv|^2\,\d x
\notag \\
&\quad  \leq  C(\|\nabla \uu\|^2+ \|\mu\|_{H^2(\Omega)}^2+ \|\nabla \partial_t\varphi\|^2)\int_\Omega \nu(\varphi)|D\vv|^2\,\d x
\notag \\
&\qquad + C(\|\nabla \mu\|^2+\|\varphi\|_{H^2(\Omega)}^2+ \|\sigma\|_{H^1(\Omega)}^2).
\label{NS-high2}
\end{align}
Then, by Gronwall's lemma and \eqref{NS-es-H2}, we conclude that
\begin{align}
\|\vv(t)\|_{\H^1(\Omega)}^2+\int_0^t \big(\|\partial_t\vv(s)\|^2+ \|\vv(s)\|^2_{\H^2(\Omega)}\big)\,\d s\leq C,\quad \forall\, t\in [0,T],
\label{NS-high3}
\end{align}
where $C>0$ depends on $\|\nabla \vv_0\|$, $T$, $\rho_*$, $\rho^*$, $\nu_*$, $\nu^*$, $\|\mu\|_{L^\infty(0,T;H^1(\Omega))}$,
$\|\mu\|_{L^2(0,T;H^2(\Omega))}$, $\|\varphi\|_{L^\infty(0,T;H^2(\Omega))}$,
$\|\partial_t \varphi\|_{L^2(0,T;H^1(\Omega))}$,
$\|\sigma\|_{L^\infty(0,T;H^1(\Omega))}$.

The above estimates enable us to construct a weak as well as a strong solution on $[0,\infty)$.
On the other hand, uniqueness of the strong solution is a consequence of the following weak-strong uniqueness result.
\smallskip

\textbf{Weak-strong uniqueness.} Assume that $\vv_1$ (resp. $\vv_2$) is a weak (resp. strong) solution to problem \eqref{NS-for}--\eqref{NS-for-bic} with the given data $(\uu,\varphi,\mu,\sigma)$, and both solutions satisfy the same initial data $\vv_0$. Let $T>0$. It is straightforward to check that $\vv_1$ satisfies the energy inequality
\begin{align}
&\frac12  \int_\Omega  \rho (\varphi(t))|\vv_1(t)|^2\,\d x
+ \int_0^t \int_\Omega  2\nu(\varphi(s)) |D\bm{v}_1(s)|^2\,\d x\mathrm{d}s
\notag\\
&\quad \leq \frac12  \int_\Omega  \rho (\varphi(0))|\vv_0|^2\,\d x+  \int_0^t\int_\Omega  (\mu(s) -\beta'(\varphi(s)) \sigma(s))\nabla\varphi (s)\cdot \vv_1(s)\, \d x\mathrm{d}s,
\notag
\end{align}
while $\vv_2$ satisfies the energy equality
\begin{align}
&\frac12  \int_\Omega  \rho (\varphi(t))|\vv_2(t)|^2\,\d x
+ \int_0^t \int_\Omega  2\nu(\varphi(s)) |D\bm{v}_2(s)|^2\,\d x\mathrm{d}s
\notag\\
&\quad
= \frac12  \int_\Omega  \rho (\varphi(0))|\vv_0|^2\,\d x+ \int_0^t\int_\Omega   (\mu(s) -\beta'(\varphi(s)) \sigma(s))\nabla\varphi(s) \cdot \vv_2(s)\, \d x\mathrm{d}s,
\notag
\end{align}
for all $t\in [0,T]$.
On the other hand, the regularity properties of $\vv_1$ and $\vv_2$ allow us to take $\bm{\zeta}=\vv_2$ in the weak formulation of $\vv_1$ and test the equation for $\vv_2$ by $\vv_1$. Taking these facts into account, following the same argument as in the proof of \cite[Theorem 4.1]{AGG2023}, we arrive at
\begin{align}
&\int_\Omega \frac12\rho(\varphi(t))|\vv_1(t)-\vv_2(t)|^2\,\d x\d s
+ \int_0^t\int_\Omega 2\nu(\varphi(s))|D(\vv_1(s)-\vv_2(s))|^2\,\d x\d s
\leq 0, \quad \forall\, t\in [0,T].
\notag
\end{align}
Hence, it easily follows that $\vv_1(t)=\vv_2(t)$ on $[0,T]$.

The proof of Proposition \ref{exe-NS-for} is complete.
\end{proof}

\section{Existence and Uniqueness of a Global Strong Solution}
\label{glo-strong}
\setcounter{equation}{0}
This section is devoted to the proof of Theorem \ref{Reg-SOL} on the existence and uniqueness of a global strong solution.

\subsection{Existence}
\label{exe-glo-strong}
Under the assumption of Theorem \ref{Reg-SOL}, problem \eqref{NSCH}--\eqref{NSCH-bic} admits a global weak solution $(\vv^*, \varphi^*, \mu^*, \sigma^*)$ in $\Omega\times [0,\infty)$ thanks to Theorem \ref{WEAK-SOL}-(2). Our aim is to prove that this weak solution is actually a strong one.

The proof is based on a bootstrap-type argument.
\smallskip

\textbf{Step 1.} Applying Proposition \ref{sigma-strong}, we find that problem \eqref{re-di-sigma}--\eqref{re-di-sigma-bdini} with $\vv=\vv^*$, $\varphi=\varphi^*$ and the same initial datum $\sigma_0$ admits a unique global strong solution in $\Omega\times [0,\infty)$, which is denoted by $\sigma^\natural$. Since $\sigma^*$ is a global weak solution to the same problem, then by the uniqueness of weak solutions due to Proposition \ref{sigma-weak}, we have $\sigma^*=\sigma^\natural$ on $[0,\infty)$, noticing that $T>0$ is arbitrary therein.

Next, from the facts
$$
\sigma^* \in L^\infty(0,\infty; L^2(\Omega)) \cap L^2_{\uloc}([0,\infty);H^1(\Omega)),\quad \varphi^* \in L_{\uloc}^2([0,\infty),W^{2,q}(\Omega)),
$$
for any $q\in [2,\infty)$, we can apply \eqref{es-sigH0}, the uniform Gronwall lemma combined with Gronwall's lemma to obtain the improved estimate (cf. \eqref{es-sigH1})
\begin{align}
\|\sigma^\natural\|_{L^\infty(0,\infty;H^1(\Omega))}\leq C,\quad \sup_{t\geq 0}\|\sigma^\natural\|_{L^2(t,t+1;H^2(\Omega))}\leq C,
\label{es-sigH1-i}
\end{align}
which further yields $\sigma^\natural\in H^1_{\uloc}([0,\infty);L^2(\Omega))$.
Here, the property $\sigma^* \in L^\infty(0,\infty; L^2(\Omega))$ is crucial, as it guarantees that the constant $C$ does not depend on time.
\smallskip

\textbf{Step 2.} Applying Proposition \ref{CH-strong}, we find that problem \eqref{CH}--\eqref{bcic} with $\vv=\vv^*$, $\sigma=\sigma^*$ (keeping in mind that now  $\sigma^*=\sigma^\natural$, i.e., a strong one) and the initial datum $\varphi_0$ admits a unique global strong solution in $\Omega\times [0,\infty)$, which is denoted by $(\varphi^\natural,\mu^\natural)$.
Since $(\varphi^*,\mu^*)$ is a global weak solution to the same problem, by the uniqueness of weak solutions due to Proposition \ref{well-pos}, we have $(\varphi^*,\mu^*)=(\varphi^\natural,\mu^\natural)$ on $[0,\infty)$.

With the aid of the uniform in time estimates for $\bm{v}^*$, $\varphi^*$, $\mu^*$ and the estimate \eqref{es-sigH1-i} obtained in the previous step, we can use \eqref{DI-alpha}, in which the positive constant $C$ is now independent of time, the uniform Gronwall lemma combined with Gronwall's lemma to obtain the following improved estimates (cf. \eqref{EST1})
\begin{equation}
\label{nmu-H1-i}
\begin{split}
\|\nabla \mu^\natural\|_{L^\infty(0,\infty;\bm{L}^2(\Omega))}\leq C, \quad \sup_{t\geq 0}\|\nabla \partial_t\varphi^\natural\|_{L^2(t,t+1;\bm{L}^2(\Omega))}\leq C,
\end{split}
\end{equation}
which further imply
\begin{align}
& \| \mu^\natural\|_{L^\infty(0,\infty;H^1(\Omega))}\leq C,
\quad \sup_{t\geq 0} \|  \mu^\natural\|_{L^4(t,t+1; H^2(\Omega))} \leq C,
\label{impro-1} \\
& \| \varphi^\natural\|_{L^\infty(0,\infty; W^{2,q}(\Omega))}\leq C,
\quad \|\Psi'(\varphi^\natural)\|_{L^\infty(0,\infty;L^q(\Omega))}\leq C,
\label{impro-2}
\end{align}
for any $q\in [2,\infty)$. With the above uniform-in-time estimates, we can find a constant $\delta_1\in (0,1)$ independent of time such that (cf. \eqref{sep1-ch})
\be
\|\varphi^\natural(t)\|_{L^{\infty}(\Omega)} \leq 1- \delta_1,\quad \forall\, t \geq 0.
\notag
\ee
As a consequence, this gives (cf. \eqref{phih3})
\begin{equation}
\|\Psi' (\varphi^\natural )\|_{L^\infty (0 ,\infty ;H^1(\Omega ))}\leq C, \quad
\|\varphi^\natural\|_{L^\infty (0,\infty ;H^3(\Omega ))}\leq C.
\label{phih3-i}
\end{equation}
Moreover, if $\sigma_0\in L^\infty(\Omega)$, using \eqref{impro-2} and Corollary \ref{non-blowup}, we can further obtain
\begin{align}
\|\sigma^\natural(t)\|_{L^\infty(\Omega)}\leq C,\quad \forall\, t\geq 0.
\end{align}

\textbf{Step 3.} From Proposition \ref{exe-NS-for}-(2), we find that problem \eqref{NS-for}--\eqref{NS-for-bic} with $\uu=\vv^*$, $\varphi=\varphi^*$, $\mu=\mu^*$, $\sigma=\sigma^*$ (note that $(\varphi^*, \mu^*, \sigma^*)= (\varphi^\natural, \mu^\natural, \sigma^\natural)$ is actually a strong solution satisfying the improved estimates obtained above) and the initial datum $\vv_0$ admits a unique global strong solution in $\Omega\times [0,\infty)$, which is denoted by $\vv^\natural$. Since $\vv^*$ is a global weak solution to the same problem, by the weak-strong uniqueness result shown in the proof of Proposition \ref{exe-NS-for}, we find $\vv^*=\vv^\natural$ on $[0,\infty)$. Moreover, using the estimates \eqref{nmu-H1-i}, \eqref{impro-1}, \eqref{impro-2}, we can apply \eqref{NS-high2}, in which the positive constant $C$ is now independent of time, the uniform Gronwall lemma combined with Gronwall's lemma to obtain (cf. \eqref{NS-high3})
\begin{align}
\begin{split}
&\|\vv^\natural(t)\|_{L^\infty(0,\infty;\H^1(\Omega))}\leq C,
\\
&\sup_{t\geq 0} \,(\|\partial_t\vv^\natural(t)\|_{L^2(t,t+1;\bm{L}^2(\Omega))}
+ \|\vv^\natural(t)\|_{L^2(t,t+1;\H^2(\Omega))})\leq C.
\end{split}
\label{vv-hi-str}
\end{align}

In summary, we have shown that $(\vv^*,\varphi^*, \mu^*, \sigma^*)= (\vv^\natural,\varphi^\natural, \mu^\natural, \sigma^\natural)$ on $[0,\infty)$.
By the construction of $(\vv^\natural,\varphi^\natural, \mu^\natural, \sigma^\natural)$, $(\vv^*,\varphi^*, \mu^*, \sigma^*)$ is indeed a global strong solution in the sense of Definition \ref{def-strong}. Hence, this establishes the existence of a global strong solution to problem \eqref{NSCH}--\eqref{NSCH-bic} on the whole interval $[0,\infty)$.
\hfill $\square$

\subsection{Uniqueness}
Let $(\vv_i,\varphi_i,\mu_i,\sigma_i)$, $i=1,2$, be two global strong solutions to problem \eqref{NSCH}--\eqref{NSCH-bic} subject to the initial data $(\vv_{0,i},\varphi_{0,i},\sigma_{0,i})$. We denote the differences
\begin{align*}
& \bm{V}=\vv_1-\vv_2,\quad P=P_1-P_2,
\quad \Phi=\varphi_1-\varphi_2,
\quad \Upsilon= \mu_1-\mu_2,
\quad \Sigma=\sigma_1-\sigma_2,\\
&\bm{V}_0=\vv_{0,1}-\vv_{0,2},
\quad \Phi_0=\varphi_{0,1}-\varphi_{0,2},
\quad \Sigma_0=\sigma_{0,1}-\sigma_{0,2}.
\end{align*}
Then it holds
\begin{equation}
 \label{di-NSCH}
\begin{cases}
\rho(\varphi_1) \partial_t \bm{V}
+ (\rho(\varphi_1)-\rho(\varphi_2))\partial_t\vv_2
+ \big(\rho(\varphi_1)(\vv_1\cdot\nabla)\vv_1-\rho(\varphi_2)(\vv_2\cdot\nabla)\vv_2\big)
\\
\qquad
 - \big(
\rho'(\varphi_1)m(\varphi_1)(\nabla \mu_1\cdot\nabla) \vv_1-
\rho'(\varphi_2)m(\varphi_2)(\nabla \mu_2\cdot\nabla) \vv_2
\big)
- \div \big( 2\nu(\varphi_1) D \bm{V} \big)
\\
\qquad - \div \big( 2(\nu(\varphi_1)-\nu(\varphi_2)) D \vv_2 \big)
+ \nabla \widetilde{P}
\\
\quad = -\div(\nabla \varphi_1\otimes \nabla \Phi
+ \nabla \Phi\otimes \nabla \varphi_2)
-\big(\sigma_1  \nabla \beta(\varphi_1) - \sigma_2  \nabla \beta(\varphi_2)\big),
\\
\div \, \bm{V} =0,
\\
\partial_t \Phi + \vv_1\cdot \nabla \Phi
+ \bm{V} \cdot \nabla \varphi_2
=
\div\big(m(\varphi_1)\nabla \mu_1\big)-\div\big(m(\varphi_2)\nabla \mu_2\big),
\\
\Upsilon = - \Delta \Phi+ \Psi'(\varphi_1)-\Psi'(\varphi_2) + \beta'(\varphi_1) \sigma_1-\beta'(\varphi_2) \sigma_2,
\\
\partial_t \Sigma + \vv_1 \cdot \nabla \Sigma
+ \bm{V} \cdot \nabla \sigma_2 - \Delta \Sigma
= \div \left( \beta'(\varphi_1)\sigma_1 \nabla \varphi_1\right) -\div \left( \beta'(\varphi_2)\sigma_2 \nabla \varphi_2\right),
\end{cases}
\end{equation}
almost everywhere in $\Omega \times (0,\infty)$, subject to the boundary and initial conditions
\begin{equation}
\label{di-NSCH-bic}
\begin{cases}
\bm{V}=\mathbf{0}, \quad \partial_\n \Phi=\partial_\n \Upsilon=\partial_\n \Sigma=0 \quad &\text{a.e. on }  \partial \Omega \times (0,\infty),\\
\bm{V}|_{t=0}=\bm{V}_0, \quad \Phi|_{t=0}=\Phi_0, \quad \Sigma|_{t=0}=\Sigma_0 \quad
&\text{a.e. in } \Omega.
\end{cases}
\end{equation}
In \eqref{di-NSCH}$_1$, we have denoted the modified pressure by
$$
\widetilde{P}=P - \left(\frac{1}{2}|\nabla \varphi_1|^2 +\Psi(\varphi_1)\right)+ \left(\frac{1}{2}|\nabla \varphi_2|^2 +\Psi(\varphi_2)\right).
$$
Testing \eqref{di-NSCH}$_1$ by $\bm{V}$, we have
\begin{align}
& \frac12\frac{\d}{\d t} \int_\Omega \rho(\varphi_1) |\bm{V}|^2\,\d x
+ 2\int_\Omega \nu(\varphi_1)|D \bm{V}|^2\,\d x
\notag\\
&\quad = - \int_\Omega (\rho(\varphi_1)-\rho(\varphi_2))\partial_t\vv_2\cdot \bm{V}\,\d x
- \int_\Omega \rho(\varphi_1)(\bm{V}\cdot\nabla)\vv_2\cdot\bm{V}\,\d x
\notag\\
&\qquad -\int_\Omega (\rho(\varphi_1)-\rho(\varphi_2))(\vv_2\cdot\nabla)\vv_2\cdot\bm{V}\, \d x
\notag\\
&\qquad +\int_\Omega \big((\rho(\varphi_1)m(\varphi_1)\nabla \mu_1-\rho(\varphi_2)m(\varphi_2)\nabla \mu_2)\cdot\nabla\big)\vv_2\cdot \bm{V}\,\d x
\notag \\
&\qquad -2\int_\Omega(\nu(\varphi_1)-\nu(\varphi_2))D\vv_2: \nabla \bm{V}\,\d x
+ \int_\Omega (\nabla \varphi_1\otimes \nabla \Phi+ \nabla \Phi\otimes \nabla \varphi_2):\nabla \bm{V}\, \d x
\notag \\
&\qquad -\int_\Omega \big(  \sigma_1  \nabla \beta(\varphi_1)- \sigma_2  \nabla \beta(\varphi_2)) \cdot \bm{V}\, \d x
\notag\\
&\quad =:\sum_{i=1}^{7}Z_i.
\label{di-NS}
\end{align}
The terms $Z_i$, $i=1,2,3,5,6$ can be estimated exactly the same as in \cite[Section 6]{Gior2021}, thus we obtain
\begin{align}
& |Z_1|+|Z_2|+|Z_3|+|Z_5|+|Z_6|
\notag\\
&\quad \leq \frac{\nu_*}{6}\|D\bm{V}\|^2 +
C\big(\|\partial_t\vv_2\|^2+ \|\vv_2\|^2_{\H^2(\Omega)}\big)\big(\|\bm{V}\|^2+ \|\Phi\|_{H^1(\Omega)}^2\big).
\notag
\end{align}
Concerning $Z_4$, we find
\begin{align*}
Z_4&= \int_\Omega \rho(\varphi_1)m(\varphi_1)(\nabla\Upsilon \cdot\nabla)\vv_2\cdot \bm{V}\,\d x
+ \int_\Omega \rho(\varphi_1)(m(\varphi_1)-m(\varphi_2))(\nabla \mu_2\cdot\nabla)\vv_2\cdot\bm{V}\,\d x
\notag\\
&\quad + \int_\Omega (\rho(\varphi_1)-\rho(\varphi_2))m(\varphi_2)(\nabla \mu_2\cdot\nabla)\vv_2\cdot\bm{V}\,\d x
\notag\\
&=: Z_{4}^{(1)}+ Z_{4}^{(2)}+Z_{4}^{(3)}.
\notag
\end{align*}
We note that
\begin{align}
\|\nabla \Upsilon\|
&\leq \|\nabla \Delta\Phi \|+ \|\nabla (\Psi'(\varphi_1)-\Psi'(\varphi_2))\|+ \|\nabla (\beta'(\varphi_1) \sigma_1-\beta'(\varphi_2) \sigma_2)\|
\notag\\
&\leq \|\nabla \Delta \Phi \| + \|\Psi''(\varphi_1)\|_{L^\infty(\Omega)}\|\nabla \Phi\|
\notag\\
&\quad +\left\|\int_0^1\Psi'''(s\varphi_1+(1-s)\varphi_2)\Phi\,\d s\right\|_{\L^4(\Omega)}\|\nabla \varphi_2\|_{\L^4(\Omega)}
\notag\\
&\quad +\|\beta'(\varphi_1)\|_{L^\infty(\Omega)}\|\nabla \Sigma\|
+ \left\|\int_0^1\beta''(s\varphi_1+(1-s)\varphi_2)\Phi\,\d s\right\|_{\L^4(\Omega)}\|\nabla \sigma_2\|_{\L^4(\Omega)}
\notag\\
&\quad + \|\beta''(\varphi_1)\|_{L^\infty(\Omega)}\|\nabla \varphi_1\|_{\L^4(\Omega)}\|\Sigma\|_{L^4(\Omega)}
+ \|\beta''(\varphi_1)\|_{L^\infty(\Omega)}\|\nabla \Phi\|\|\sigma_2\|_{L^\infty(\Omega)}
\notag\\
&\quad + \left\|\int_0^1\beta'''(s\varphi_1+(1-s)\varphi_2)\Phi\,\d s\right\|_{\L^4(\Omega)} \|\nabla \varphi_2\|_{\L^4(\Omega)}\|\sigma_2\|_{L^\infty(\Omega)}
\notag\\
&\leq \|\nabla \Delta\Phi \| + C\Big(1+\|\sigma_2\|_{H^2(\Omega)}^\frac12\Big)\|\Phi\|_{H^1(\Omega)} + C_*\|\nabla\Sigma\| + C\|\Sigma\|,
\label{dif-mu}
\end{align}
where
$$C_*=\max_{r\in [-1,1]}|\beta'(r)|+1,$$
and the positive constant $C$ depends on the strict separation property of $\varphi_1$, $\varphi_2$ on $[0,T]$, $\|\varphi_1\|_{L^\infty(0,T;H^3(\Omega))}$, $\|\varphi_2\|_{L^\infty(0,T;H^3(\Omega))}$, $\|\sigma_2\|_{L^\infty(0,T;H^1(\Omega))}$.
Hence, it holds
\begin{align*}
 |Z_4^{(1)}|
 &\leq \|\rho(\varphi_1)\|_{L^\infty(\Omega)}\|m(\varphi_1)\|_{L^\infty(\Omega)} \|\nabla\Upsilon\| \|\nabla\vv_2\|_{\L^4(\Omega)} \|\bm{V}\|_{\L^4(\Omega)}
 \notag\\
 &\leq  C\|\nabla \Upsilon\| \|\vv_2\|_{\H^2(\Omega)}^\frac12 \|\bm{V}\|^\frac12 \|\nabla \bm{V}\|^\frac12
 \notag\\
 &\leq \frac{\nu_*}{6}\|D\bm{V}\|^2
 + \frac{m_*}{12} \|\nabla \Delta \Phi\|^2
 + \frac{1}{2}\|\nabla \Sigma\|^2
 +C \|\vv_2\|_{\H^2(\Omega)}^2\|\bm{V}\|^2
 \notag\\
 &\quad +C\big(1+\|\sigma_2\|_{H^2(\Omega)}\big) \|\Phi\|_{H^1(\Omega)}^2
 + C\|\Sigma\|^2,
\end{align*}
where $C>0$ also depends on $\|\vv_2\|_{L^\infty(0,T;\H^1(\Omega))}$.
Next, we infer from Agmon's inequality that
\begin{align*}
  |Z_4^{(2)}| &\leq  \|\rho(\varphi_1)\|_{L^\infty(\Omega)} \left\|\int_0^1m'(s\varphi_1+(1-s)\varphi_2)\Phi\,\d s\right\|_{L^\infty(\Omega)}\|\nabla \mu_2\|\|\nabla \vv_2\|_{\L^4(\Omega)}\|\bm{V}\|_{\L^4(\Omega)}
  \notag\\
  &\leq C \big(\|\Phi\|_{H^1(\Omega)}^\frac34 \|\nabla \Delta \Phi\|^\frac14
  + \|\Phi\|_{H^1(\Omega)}\big) \|\vv_2\|_{\H^2(\Omega)}^\frac12 \|\bm{V}\|^\frac12 \|\nabla\bm{V}\|^\frac12
  \notag\\
  &\leq \frac{\nu_*}{6}\|D\bm{V}\|^2
  + \frac{m_*}{12} \|\nabla \Delta \Phi\|^2
  + C \|\vv_2\|_{\H^2(\Omega)}^2\|\bm{V}\|^2+ C\|\Phi\|_{H^1(\Omega)}^2,
  \notag
\end{align*}
and
\begin{align*}
|Z_4^{(3)}|&\leq   \|m(\varphi_2)\|_{L^\infty(\Omega)} \left\|\int_0^1\rho'(s\varphi_1+(1-s)\varphi_2)\Phi\,\d s\right\|_{L^\infty(\Omega)}\|\nabla \mu_2\|\|\nabla \vv_2\|_{\L^4(\Omega)}\|\bm{V}\|_{\L^4(\Omega)}
  \notag\\
   &\leq \frac{\nu_*}{6}\|D\bm{V}\|^2
   + \frac{m_*}{12} \|\nabla \Delta \Phi\|^2
  + C \|\vv_2\|_{\H^2(\Omega)}^2 \|\bm{V}\|^2
  + C\|\Phi\|_{H^1(\Omega)}^2.
  \notag
\end{align*}
For $Z_7$, we deduce that
\begin{align*}
|Z_7|&\leq \|\beta'(\varphi_1)\|_{L^\infty(\Omega)}
\|\nabla \varphi_1\|_{\L^4(\Omega)}\|\Sigma\|_{L^4(\Omega)} \|\bm{V}\|
+ \|\beta'(\varphi_1)\|_{L^\infty(\Omega)} \|\nabla \Phi\| \|\sigma_2\|_{L^\infty(\Omega)} \|\bm{V}\|
\notag\\
&\quad + \left\|\int_0^1\beta''(s\varphi_1+(1-s)\varphi_2)\Phi\,\d s\right\|_{L^4(\Omega)}
\|\nabla \varphi_2\| \|\sigma_2\|_{L^\infty(\Omega)} \|\bm{V}\|
\notag\\
&\leq \frac{1}{2}\|\nabla \Sigma\|^2
+ C\|\bm{V}\|^2
+ C \|\sigma_2\|_{H^2(\Omega)} \|\Phi\|_{H^1(\Omega)}^2
 + C\|\Sigma\|^2.
\end{align*}
Testing the third equation in \eqref{di-NSCH} by $-\Delta \Phi$, after integration by parts, we get
\begin{align}
&  \frac12\frac{\d}{\d t} \|\nabla \Phi\|^2 +\int_\Omega m(\varphi_1)|\nabla \Delta \Phi|^2\,\d x
\notag\\
&\quad = -\int_\Omega \nabla (\vv_1\cdot \nabla \Phi)\cdot \nabla  \Phi\, \d x
- \int_\Omega \nabla (\bm{V}\cdot\nabla \varphi_2)\cdot \nabla \Phi\, \d x
\notag\\
&\qquad
+  \int_\Omega  m(\varphi_1) \nabla (\Psi'(\varphi_1)-\Psi'(\varphi_2))
\cdot \nabla \Delta \Phi\, \d x
\notag\\
&\qquad
+ \int_\Omega m(\varphi_1) \nabla (\beta'(\varphi_1) \sigma_1-\beta'(\varphi_2) \sigma_2)\cdot \nabla \Delta \Phi\,\d x
\notag\\
&\qquad
 + \int_\Omega (m(\varphi_1)-m(\varphi_2))\nabla \mu_2\cdot \nabla \Delta \Phi\, \d x
\notag\\
&\quad =:\sum_{i=8}^{12} Z_i.
\label{di-vphi}
\end{align}
We recall that $Z_8$, $Z_9$ can be estimated in the same way as in \cite[Section 6]{Gior2021} such that
\begin{align}
|Z_8|+|Z_9|
&\leq
\frac{\nu_*}{6}\|D\bm{V}\|^2
+ \frac{m_*}{12} \|\nabla \Delta \Phi\|^2
+ C\|\Phi\|_{H^1(\Omega)}^2.
\notag
\end{align}
Using a similar argument for \eqref{dif-mu} together with Young's inequality, we have
\begin{align}
|Z_{10}|+|Z_{11}|
&
\leq \|m(\varphi_1)\|_{L^\infty(\Omega)} \big(\|\nabla (\Psi'(\varphi_1)-\Psi'(\varphi_2))\|
  + \|\nabla (\beta'(\varphi_1) \sigma_1-\beta'(\varphi_2) \sigma_2)\|\big)\|\nabla \Delta \Phi\|
\notag\\
&\quad \leq \frac{m_*}{12} \|\nabla \Delta \Phi\|^2
+ C_{**}\|\nabla \Sigma\|^2
+ C\big(1+\|\sigma_2\|_{H^2(\Omega)}\big) \|\Phi\|_{H^1(\Omega)}^2 + C\|\Sigma\|^2,
\notag
\end{align}
where
$$
C_{**}=\frac{8(m^*C_*)^2}{m_*}.
$$
Concerning $Z_{12}$, it holds
\begin{align}
 |Z_{12}|&\leq   \left\|\int_0^1 m'(s\varphi_1+(1-s)\varphi_2)\Phi\,\d s\right\|_{L^\infty(\Omega)}\|\nabla \mu_2\|\|\nabla \Delta \Phi\|
 \notag\\
 &\leq C\|\Phi\|_{H^2(\Omega)}\|\nabla \Delta \Phi\|
 \notag\\
 &\leq C\|\Phi\|_{H^1(\Omega)}^\frac12\big(\|\Phi\|_{H^1(\Omega)}+ \|\nabla \Delta \Phi\|\big)^\frac12 \|\nabla \Delta \Phi\|
 \notag\\
 &\leq  \frac{m_*}{12} \|\nabla \Delta \Phi\|^2
 + C\|\Phi\|_{H^1(\Omega)}^2.
 \notag
\end{align}
Testing the fifth equation in \eqref{di-NSCH} with $ \Sigma$, after integration by parts, we deduce that
\begin{align}
\frac{1}{2} \frac{\mathrm{d}}{\mathrm{d} t}\|\Sigma\|^2 + \|\nabla \Sigma \|^2
& = - \int_{\Omega} (\bm{V}  \cdot \nabla \sigma_2)  \Sigma \, \mathrm{d} x
-\int_\Omega \big(\sigma_1\nabla \beta(\varphi_1)-\sigma_2\nabla\beta(\varphi_2)\big)\cdot \nabla \Sigma\,\d x
\notag\\
&=: Z_{13}+Z_{14}.
\label{dif-sig}
\end{align}
We observe that
\begin{align}
|Z_{13}|& \le \|\bm{V}\|_{\L^4(\Omega)}\|\nabla\sigma_2\| \|\Sigma\|_{L^4(\Omega)}
\notag\\
&\leq C\|\bm{V}\|^\frac12 \|\nabla \bm{V}\|^\frac12 \|\Sigma\|^\frac12 \|\Sigma\|_{H^1(\Omega)}^\frac12
\notag\\
&\leq \frac{\nu_*}{24(1+C_{**})}\|D\bm{V}\|^2
+ \frac{1}{4}\|\nabla \Sigma\|^2
+ C\|\bm{V}\|^2 +C\|\Sigma\|^2,
\notag
\end{align}
and
\begin{align*}
|Z_{14}|&\leq \|\beta'(\varphi_1)\|_{L^\infty(\Omega)}
\|\nabla \varphi_1\|_{\L^4(\Omega)}\|\Sigma\|_{L^4(\Omega)} \|\nabla \Sigma\|
+ \|\beta'(\varphi_1)\|_{L^\infty(\Omega)}\|\nabla \Phi\| \|\sigma_2\|_{L^\infty(\Omega)}
\|\nabla \Sigma\|
\notag\\
&\quad + \left\|\int_0^1\beta''(s\varphi_1+(1-s)\varphi_2)\Phi\,\d s\right\|_{L^4(\Omega)}
\|\nabla \varphi_2\|\|\sigma_2\|_{L^\infty(\Omega)}\|\nabla \Sigma\|
\notag\\
&\leq \frac{1}{4}\|\nabla \Sigma\|^2
+ C \|\sigma_2\|_{H^2(\Omega)} \|\Phi\|_{H^1(\Omega)}^2
 + C\|\Sigma\|^2.
\end{align*}

Collecting the above estimates, multiplying \eqref{dif-sig} by $4(1+C_{**})$ and adding the resultant with \eqref{di-NS}, \eqref{di-vphi}, we obtain
\begin{align}
&\frac{\d}{\d t} \left(\int_\Omega \rho(\varphi_1) |\bm{V}|^2\,\d x + \|\nabla \Phi\|^2 + |\overline{\Phi}|^2 + 4(1+C_{**})\|\Sigma\|^2\right)
\notag\\
&\qquad + 2\nu_*\|D \bm{V}\|^2
+  m_* \|\nabla \Delta \Phi\|^2
+ 2(1+C_{**})\|\nabla \Sigma\|^2
\notag\\
&\quad \leq C\mathcal{F}(t)\left(\int_\Omega \rho(\varphi_1) |\bm{V}|^2\,\d x
+ \|\nabla \Phi\|^2
+ |\overline{\Phi}|^2
+ 4(1+C_{**})\|\Sigma\|^2\right),
\label{conti-1}
\end{align}
for almost all $t\in (0,T)$, where
$$
\mathcal{F}(t)=1 + \|\partial_t\vv_2(t)\|^2
+ \|\vv_2(t)\|^2_{\H^2(\Omega)} +\|\sigma_2(t)\|_{H^2(\Omega)}^2.
$$
We note that the constant $C>0$ in \eqref{conti-1} depends on the strict separation property of $\varphi_1$, $\varphi_2$ on $[0,T]$, $\|\varphi_1\|_{L^\infty(0,T;H^3(\Omega))}$, $\|\varphi_2\|_{L^\infty(0,T;H^3(\Omega))}$, $\|\sigma_2\|_{L^\infty(0,T;H^1(\Omega))}$, $\rho_*$, $\rho^*$, $\nu_*$, $m_*$, $m^*$, $\beta^*$, $C_*$, $\Omega$ and $T$. In the above derivation, we have also used the mass conservation $\overline{\Phi(t)}= \overline{\Phi_0}$ and the Poincar\'e--Wirtinger inequality. Since $\mathcal{F} \in L^1(0, T)$, if
$(\vv_{0,1}, \varphi_{0,1}, \sigma_{0,1})=(\vv_{0,2}, \varphi_{0,2}, \sigma_{0,2})$, the conclusion $(\vv_1,\varphi_1,\sigma_1)=(\vv_2,\varphi_2,\sigma_2)$ on $[0, T]$ is a direct consequence of Gronwall's lemma. Hence, the global strong solution to problem \eqref{NSCH}--\eqref{NSCH-bic} is unique.

The proof of Theorem \ref{Reg-SOL} is complete.
\hfill $\square$

\section{Propagation of Regularity for Global Weak Solutions}
\label{Propa-reg}
\setcounter{equation}{0}

In this section, we prove Theorem \ref{regw} on the propagation of regularity for global weak solutions.

The proof is based on a bootstrap-type argument in the sprit of Section \ref{exe-glo-strong}. Let $(\vv^*, \varphi^*,\mu^*,\sigma^*)$ be a global weak solution given by Theorem \ref{WEAK-SOL}-(2). Assume that $\tau\in (0,1)$ is an arbitrary but fixed time.
\smallskip

\textbf{Step 1. Global regularity of $\sigma^*$.}
Since $\sigma^* \in  L^2(0,1; H^1(\Omega))$, there exists $\tau_1\in(0,\tau)$ such that $\sigma^*(\tau_1) \in H^1(\Omega)$ with $\sigma^*(\tau_1)\geq 0$ almost everywhere in $\Omega$. At the same time, we have the regularity properties
\begin{align*}
&\vv^* \in L^{\infty}(0,\infty;\L^2_{\sigma}(\Omega))\cap L^2(0,\infty;\H^1_{0,\sigma}(\Omega)),
\\
&\varphi^* \in L^\infty(0,\infty;H^1(\Omega))\cap L^2_{\mathrm{uloc}}([0,\infty);W^{2,q}(\Omega)),
\end{align*}
for any $q\geq 2$.
Hence, taking $\sigma^*(\tau_1)$ as the new initial datum, we infer from Proposition \ref{sigma-strong} and the argument in Section \ref{exe-glo-strong} that problem
\eqref{re-di-sigma}--\eqref{re-di-sigma-bdini} with $\vv=\vv^*$, $\varphi=\varphi^*$ admits a unique global strong solution denoted by $\sigma^\natural$ on $[\tau_1,\infty)$ such that
\begin{equation}
\label{REG1-b}
\begin{split}
\sigma^\natural \in L^{\infty}(\tau_1,\infty;H^1(\Omega)) \cap L^2_{\mathrm{uloc}}([\tau_1,\infty); H^2(\Omega))\cap H^1_{\mathrm{uloc}}([\tau_1,\infty);L^2(\Omega)).
\end{split}
\end{equation}
Since $\sigma^*$ is a global weak solution to the same problem on $[\tau_1,\infty)$, then by the uniqueness of weak solutions due to Proposition \ref{sigma-weak}, we find $\sigma^*=\sigma^\natural$, which establishes the regularization of $\sigma^*$ for $t\geq \tau_1$.
\smallskip

\textbf{Step 2. Global regularity of $(\varphi^*,\mu^*)$.}
Since
\begin{align*}
&\varphi^* \in L^\infty(0,1;H^1(\Omega))\cap L^2(0,1; H_N^2(\Omega)),\quad
\mu^*\in L^2(0,1; H^1(\Omega)),
\\
&\sigma^*\in L^\infty(0,1;L^2(\Omega))\cap L^2(0,1;H^1(\Omega)),
\end{align*}
then applying \eqref{es-beta-sigH1} and \eqref{Dephi-L2} to $(\varphi^*,\sigma^*)$, we find $\|\beta'(\varphi^*)\sigma^*\|_{L^2(0,1;H^1(\Omega))}$ is bounded as well. As a result, it holds $-\Delta \varphi^*+\Psi'(\varphi^*)\in L^2(0,1;H^1(\Omega))$. Hence, there exist $\tau_2\in (\tau_1,\tau)$ such that $\varphi^*(\tau_2) \in H^2_N(\Omega)$ with
	$\| \varphi^*(\tau_2)\|_{L^\infty(\Omega)}\leq 1$, $\big|\overline{\varphi^*(\tau_2)}\big|<1$, $\mu^*_{\tau_2}:=-\Delta \varphi^*(\tau_2)+\Psi'(\varphi^*(\tau_2)) \in H^1(\Omega)$ and $\sigma^*(\tau_2)\in H^1(\Omega)$.

Consider problem \eqref{CH}--\eqref{bcic}, with $\varphi^*(\tau_2)$ being the new initial datum and $\vv=\vv^*$, $\sigma=\sigma^*$. Thanks to the fact $\sigma^*=\sigma^\natural$, the regularity \eqref{REG1-b} and the Sobolev embedding theorem, we can conclude  that
$$
\sigma^* \in L^{\infty}(\tau_2,\infty;L^q(\Omega))\cap H^1_{\mathrm{uloc}}([\tau_2,\infty);L^2(\Omega)),
$$ for any $q\geq 2$.
Applying Proposition \ref{CH-strong} and the argument in Section \ref{exe-glo-strong}, we see that problem \eqref{CH}--\eqref{bcic} admits a unique strong solution denoted by  $(\varphi^\natural,\mu^\natural)$ on $[\tau_2,\infty)$ such that
\begin{equation}
\label{REG-b}
\begin{split}
&\varphi^\natural \in L^\infty(\tau_2,\infty;W^{2,q}(\Omega)), \quad \partial_t \varphi^\natural \in L^2_{\mathrm{uloc}}([\tau_2,\infty);H^1(\Omega)),\\
&\mu^\natural \in L^{\infty}(\tau_2,\infty; H^1(\Omega))\cap L^2_{\mathrm{uloc}}([\tau_2,\infty);H^3(\Omega)),
\quad\Psi'(\varphi^\natural) \in L^\infty(\tau_2,\infty; L^q(\Omega)),
\end{split}
\end{equation}
for any $q\geq 2$. Since $(\varphi^*,\mu^*)$ is a global weak solution to the same problem on $[\tau_2,\infty)$, then by the uniqueness of weak solutions due to Proposition \ref{well-pos}, we find $(\varphi^*,\mu^*)=(\varphi^\natural,\mu^\natural)$ on $[\tau_2,\infty)$, which establishes the regularization of $(\varphi^*,\mu^*)$ for $t\geq \tau_2$.

Next, we establish the strict separation property of $\varphi^*$. At the initial time $\tau_2$, since $\mu^*({\tau_2})\in H^1(\Omega)$ and $\beta'(\varphi^*(\tau_2))\sigma^*(\tau_2)\in H^1(\Omega)$ (owing to the facts $\varphi^*(\tau_2)\in H^2(\Omega)$, $\sigma^*(\tau_2)\in H^1(\Omega)$), by \cite[Lemma 7.4]{GiGrWu2018}, we find $\|\Psi_0''(\varphi^*(\tau_2))\|_{L^q(\Omega)} \leq C(q)$ for any $q\in [2,\infty)$.
Then applying \cite[Lemma 3.2]{HW}, we can conclude that $\Psi_0'(\varphi^*(\tau_2)) \in W^{1,q}(\Omega)$ for any $q\in[2,\infty)$, which further implies $\Psi_0'(\varphi^*(\tau_2))\in L^\infty(\Omega)$. Hence, there exists a constant $\delta(\tau_2)\in (0,1)$ such that
\begin{equation}
\|\varphi^*(\tau_2)\|_{L^{\infty}(\Omega)} \leq 1-\delta(\tau_2).
\label{inisep}
\end{equation}
By the same reasoning, we deduce from the facts $$\varphi^* \in L^\infty(\tau_2,\infty;W^{2,q}(\Omega)),\quad \mu^* \in L^{\infty}(\tau_2,\infty; H^1(\Omega)),\quad \sigma^* \in L^{\infty}(\tau_2,\infty;H^1(\Omega))
$$
that
there exists a constant $\delta_2(\tau_2)\in (0,\delta(\tau_2)]$ such that
\begin{equation}
\|\varphi^*(t)\|_{L^{\infty}(\Omega)} \leq 1-\delta_2(\tau_2),\quad
\forall\, t\geq \tau_2.
\label{T-sep}
\end{equation}
Finally, the strict separation property \eqref{T-sep} combined with the elliptic estimate for \eqref{CH}$_2$ further implies that $
\varphi^*\in L^\infty(\tau_2,\infty;H^3(\Omega))$.
\smallskip

\textbf{Step 3. Global regularity of $\vv^*$.}
Since $\vv^*\in L^2(0,1;\bm{H}^1_{0,\sigma}(\Omega))$, there exists some $\tau_3\in (\tau_2,\tau)$ such that $\vv(\tau_3)\in \bm{H}^1_{0,\sigma}(\Omega)$.  From Proposition \ref{exe-NS-for}-(2), we find that problem \eqref{NS-for}--\eqref{NS-for-bic} with $\uu=\vv^*$, $\varphi=\varphi^*$, $\mu=\mu^*$, $\sigma=\sigma^*$ (note that $(\varphi^*, \mu^*, \sigma^*)= (\varphi^\natural, \mu^\natural, \sigma^\natural)$  is actually a strong solution that satisfies the improved estimates obtained above) and the initial datum $\vv(\tau_3)$ admits a unique global strong solution on $[\tau_3,\infty)$, which is denoted by $\vv^\natural$. Since $\vv^*$ is a global weak solution to the same problem, then by the weak-strong uniqueness result again, we find $\vv^*=\vv^\natural$ on $[\tau_3,\infty)$.
\smallskip

Since $\tau\in (0,1)$ is arbitrary, we can conclude that every global weak solution to problem \eqref{NSCH}--\eqref{NSCH-bic} becomes a global strong one for $t>0$.

The proof of Theorem \ref{regw} is complete.
\hfill $\square$

\appendix

\section{Appendix}
\setcounter{equation}{0}
\subsection{Proof of Lemma \ref{NSSa}}
\label{appen-1}
In the appendix, we sketch the proof of Lemma \ref{NSSa}.

\begin{proof}[Proof of Lemma \ref{NSSa}] In view of Equation \eqref{g4.d}, we find that $\mu^k$ (i.e., $\{c_i^{k}\}_{i=1}^{k}$) can be uniquely determined by $(\varphi^k, \sigma^k)$. Hence, problem \eqref{aatest3.c}--\eqref{aatest3.cini} can be reduced to a system with $2k$ nonlinear ordinary differential equations for the time-dependent coefficients $\{a_{i}^{k}\}_{i=1}^{k}$ (by taking $\bm{\zeta}=\bm{y}_i$, $i=1,\cdots,k$) and $\{b_{i}^{k}\}_{i=1}^{k}$ (by taking $w=z_i$, $i=1,\cdots,k$). Since $\widehat{\rho}$ is bounded with a strictly positive lower bound, then the Gram matrix $\{(\widehat{\rho}(\varphi^k)\yy_i,\yy_j)\}_{i,j=1,\cdots,k}$ in the first term on the left-hand side of \eqref{aatest3.c} is non-degenerate. From the assumptions (H2)--(H4), the construction of $\Psi_\epsilon$  and the regularity property of $\sigma^k$, we can apply the Cauchy--Lipschitz theorem for nonlinear ODE systems to conclude the existence and uniqueness of local solutions $\{a_{i}^{k}\}_{i=1}^k$, $\{b_{i}^{k}\}_{i=1}^{k}\subset  C^1([0,T_k])$, which satisfy the above mentioned ODE system on a certain time interval $[0,T_k]\subset[0,T]$.
Therefore, we obtain a unique local solution $(\bm{v}^k,\varphi^k,\mu^k)$ to problem \eqref{aatest3.c}--\eqref{aatest3.cini} that satisfies
$$
\bm{v}^k\in C^1([0,T_{k}];\bm{Y}_{k}),
\quad
\varphi^k\in C^1([0,T_{k}];Z_{k}),
\quad
\mu^k\in C^1([0,T_{k}];Z_{k}).
$$

We proceed to show that the existence time $T_k$ can be extended to $T$.
Testing \eqref{aatest3.c} by $\bm{v}^{k}$, using the incompressibility condition and integration by parts, we have
\begin{align}
&\frac12 \frac{\mathrm{d}}{\mathrm{d}t}\big(\widehat{\rho}(\varphi^k) \bm{v}^k, \bm{v}^k\big) +\int_{\Omega}2\nu(\varphi^{k} )|D\bm{v}^{k}|^2 \, \mathrm{d}x + \gamma \int_\Omega |\nabla \vv^k|^4 \, \mathrm{d}x
 \notag \\
 &\quad  = \big( \mu^{k}\nabla \varphi^{k},\bm{v}^{k}\big)
-\big(\beta'(\varphi^k)\sigma^{k}\nabla \varphi^{k},\bm{v}^{k}\big),
\label{aenergy}
\end{align}
for any $t\in(0,T_k)$. Next, testing \eqref{g1.a} by $\mu^{k}$, we find
\begin{align}
& \frac{\mathrm{d}}{\mathrm{d}t}
\left(\frac{1}{2}\|\nabla \varphi^{k}\|^2+ \int_\Omega  \Psi_\epsilon(\varphi^{k})\, \mathrm{d}x
  \right)
 + \int_\Omega m(\varphi^k)|\nabla \mu^{k}|^2\,\mathrm{d}x
\notag \\
&\quad  = -\big(\bm{v}^{k} \cdot \nabla \varphi^{k},\mu^{k}\big)
-\big( \partial_t\beta(\varphi^{k}) \sigma^{k}\big),
\label{menergy1}
\end{align}
for any $t\in(0,T_k)$.
 We note that the first terms on the right-hand side of \eqref{aenergy} and \eqref{menergy1} cancel each other.
By H\"{o}lder's inequality and the Poincar\'{e}--Wirtinger inequality, we find
\begin{align}
\left|\big(\beta'(\varphi^k) \sigma^{k} \nabla \varphi^{k},\bm{v}^{k}\big)\right|
&  \le  \| \beta'(\varphi^k) \|_{L^\infty(\Omega)} \| \sigma^{k}\|  \|\nabla\varphi^{k}\|_{\bm{L}^{\infty}(\Omega)}\|\bm{v}^{k}\|
\notag \\
&  \le  C  \| \sigma^{k}\|  \| \varphi^{k}\|_{H^3(\Omega)}  \|\bm{v}^{k}\|
\notag \\
& \le C_k\| \sigma^{k}\| \big( \|\nabla \varphi^{k}\|^2 +\|\bm{v}^{k}\|^2 +1 \big).
\label{me-a}
\end{align}
Concerning the second term on the right-hand side of \eqref{menergy1}, we take
$\xi=\partial_t\varphi^k$ in \eqref{g1.a} and apply the same argument as in \cite[Appendix 2]{GHW1} to obtain
\begin{align}
\|\partial_t\varphi^k\|
 \leq C_k(\|\bm{v}^{k}\| \|\varphi^{k}\| + \|\nabla \mu^{k}\|).
\notag
\end{align}
Then, by Young's inequality and the assumption (H4), we have
\begin{align}
\left|\int_\Omega \partial_t\beta(\varphi^{k}) \sigma^{k}\, \mathrm{d}x\right|
& \le \| \beta'(\varphi^k) \|_{L^\infty(\Omega)}\|\partial_t\varphi^{k}\| \|\sigma^{k}\|
\notag \\
&\leq \frac{m_*}{2}\|\nabla \mu^k\|^2
+C_k  \|\sigma^k\|^2 + C_k \|\bm{v}^{k}\|  \|\varphi^{k}\|\|\sigma^k\|.
\label{me-b}
\end{align}
Adding \eqref{aenergy} with \eqref{menergy1}, we infer from \eqref{me-a}, \eqref{me-b} and the Poincar\'{e}--Wirtinger inequality  that
\begin{align}
& \frac{\mathrm{d}}{\mathrm{d}t}
\left(\frac12\big(\widehat{\rho}(\varphi^k) \bm{v}^k, \bm{v}^k\big)
+\frac{1}{2}\|\nabla \varphi^{k}\|^2
+\int_\Omega \Psi_\epsilon (\varphi^{k})\, \mathrm{d}x
  \right)
+2\nu_*\int_{\Omega} |D\bm{v}^{k}|^2 \, \mathrm{d}x
\notag\\
&\qquad
+\gamma \int_\Omega |\nabla \vv^k|^4 \, \mathrm{d}x
+ \frac{m_*}{2}\|\nabla  \mu^{k}\|^2
\notag \\
&\quad  \le C_{1,k}(1+\|\sigma^{k}\|^2)\left(\frac12\big(\widehat{\rho}(\varphi^k) \bm{v}^k, \bm{v}^k\big)
+\frac{1}{2}\|\nabla \varphi^{k}\|^2
+\int_\Omega \Psi_\epsilon(\varphi^{k})\, \mathrm{d}x
\right)
	+C_{1,k}(1+\|\sigma^{k}\|^2).
\label{menergy2}
\end{align}
Moreover, the initial value of the energy can be controlled as follows
\begin{align}
&\int_{\Omega}\left[\frac{1}{2}\widehat{\rho}(\bm{P}_{Z_{k}}\varphi_{0,\gamma})|\bm{P}_{\bm{Y}_{k}} \bm{v}_{0}|^{2}
+\frac{1}{2}|\nabla \bm{P}_{Z_{k}}\varphi_{0}|^{2}
+  \Psi_\epsilon(\bm{P}_{Z_{k}}\varphi_{0})
 \right] \mathrm{d}x
\notag\\
&\quad \le C\Big(\rho^*,\|  \bm{v}_{0}\|, \|\varphi_{0}\|_{H^1(\Omega)}, \max_{r\in[-1,1]}|\Psi_0(r)|, \Omega\Big)
 =: C_0.
\label{iniC0}
\end{align}
Exploiting the convexity of $\Psi_{0,\epsilon}$, we have
\be
\frac{1}{2}\|\nabla \varphi^{k}\|^2
+\int_\Omega \Psi_\epsilon(\varphi^{k})\, \mathrm{d}x
\ge \frac{1}{2}\|\varphi^{k}\|^2-C_{0,k},
\label{me2}
\ee
for some $C_{0,k}\geq 0$.
Hence, applying Gronwall's lemma to \eqref{menergy2}, we deduce from \eqref{me2} that
\be
\begin{aligned}
&  \rho_*\|\bm{v}^k(t)\|^2 + \|\varphi^{k}(t)\|^2
\le 2 \mathrm{e}^{N_1(t)}\big( C_{0} + N_1(t)\big)
+ 2C_{0,k}=:N_2(t),
\quad \forall\, t\in [0,T_k],
\label{auvm1}
\end{aligned}
\ee
 with
\be
\begin{aligned}
N_1(t)&=t C_{1,k}\Big(1+ \sup_{t\in[0,T]} \|\sigma^{k}(t)\|^2\Big).
\notag
\end{aligned}
\ee
Thanks to \eqref{auvm1}, a further integration in time of \eqref{menergy2} gives
\begin{equation}
\int_0^t \| \nabla  \mu^k(s)\|^2 \, \mathrm{d}s \leq
C_0 + C_{0,k} + N_1(t)N_2(t) + N_1(t)=:N_3(t), \quad \forall \, t \in [0,T_k].\notag
\end{equation}
Based on the above estimates and Lemma \ref{fp},
we can extend the unique local solution $(\bm{v}^k,\varphi^k)$ to the whole interval $[0,T]$,
with the same estimate as \eqref{auvm1}.
Using \eqref{R-app}, we can rewrite \eqref{aatest3.c} as
\begin{align}
	& \big(\widehat{\rho}(\varphi^k) \partial_t \bm{v}^k, \bm{\zeta}\big)
	+ \big(  (( \widehat{\rho}(\varphi^k) \vv^k + \widehat{\J}^k)\cdot\nabla )\vv^k, \bm{ \zeta}\big)
    +\big(  2\nu(\varphi^{k}) D\bm{v}^k, D\bm{\zeta}\big)
    +\gamma(|\nabla \vv^k|^2\nabla\vv^k,\nabla \bm{\zeta}\big)
    \notag \\
&\quad =\big((\mu^{k}-\beta'(\varphi^k)\sigma^{k})\nabla \varphi^{k},\bm {\zeta}\big)- \frac12 \big(\widehat{R}^k\vv^k,\bm{\zeta}\big)
\notag\\
&\qquad - \frac12\big(\widehat{\rho}'(\varphi^k)(\bm{I}-\bm{P}_{Z_{k}}) \big((\bm{v}^k\cdot\nabla\varphi^k)-\div \big(m(\varphi^k)\nabla \mu^k\big)\big)\bm{v}^k,\bm{\zeta}\big).
	\label{aatest3.c-b}
\end{align}
Testing \eqref{aatest3.c-b} by $\partial_t \bm{v}^k$, we find
\begin{align*}
&\int_0^T \| \partial_t \bm{v}^k(s)\|^2 \, \mathrm{d}s
\\
&\quad  \leq \frac{C_k\max\{\rho^*,1\}}{\rho_*} \int_0^T \left(
\| \bm{v}^k(s)\|^4 +
\| \nabla  \mu^k(s)\|^2  \| \bm{v}^k(s)\|^2
+ \| \bm{v}^k(s)\|^2 +\| \bm{v}^k(s)\|^6 \right) \mathrm{d}s
\\
&\qquad  + \frac{C_k\max\{\rho^*,1\}}{\rho_*} \int_0^T \left(
\big(\|\nabla \mu^k\|^2+\| \sigma^k(s)\|^2\big) \| \varphi^k(s)\|^2
+ \| \nabla  \mu^k(s)\|^2 \| \varphi^k(s)\|^2\| \bm{v}^k(s)\|^2 \right) \mathrm{d}s
\\
&\qquad  + \frac{C_k\max\{\rho^*,1\}}{\rho_*} \int_0^T
\left( \| \vv^k(s)\|^2 \| \varphi^k(s)\|^2
+ \| \nabla  \mu^k(s)\|^2  \|\bm{v}^k(s)\|^2 \right) \mathrm{d}s
\\
&\quad \leq C_k \left( T N^2_2(T) + N_3(T) N_2(T) + TN_2(T) + T N^3_2(T)+ T N_2(T) \sup_{t\in [0,T]}\|\sigma^{k}(t)\|^2  \right)
\notag\\
&\qquad + C_k \left( N_3(T) N_2(T)^2 + N_2(T)^2 + N_3(T) N_2(T)\right).
\end{align*}
In a similar manner, we can conclude from \eqref{g1.a} that
\begin{align*}
\int_0^T \| \partial_t \varphi^k(s)\|^2 \, \mathrm{d}s
& \leq C_k \int_0^T \left( \| \bm{v}^k(s)\|^2 \| \varphi^k(s)\|^2
+\| \nabla  \mu^k(s)\|^2 \right) \mathrm{d}s
\\
&\leq C_k \big(  T N^2_2(T) + N_3(T) \big).
\end{align*}
Hence,  $(\bm{v}^k,\varphi^k)$ is bounded in $H^1(0,T;\bm{Y}_{k})\times H^1(0,T;Z_{k})$.

The proof of Lemma \ref{NSSa} is complete.
\end{proof}

\section*{Declarations}
\noindent \textbf{Acknowledgments.}
Part of this work was conducted during A. Giorgini’s visit to Fudan University in 2023 supported by the ``Fudan Fellow'' program. J.-N. He acknowledges the support by the RFS Grant from the Research Grants Council (Project P0047825, PI: Prof. Xianpeng Hu) for her research stay in Research Center for Nonlinear Analysis, The Hong Kong Polytechnic University.
H. Wu is a member of Key Laboratory of Mathematics for Nonlinear Sciences (Fudan University), Ministry of Education of China.
\smallskip

\noindent \textbf{Funding.}
A. Giorgini was supported by the MUR grant Dipartimento di Eccellenza 2023--2027 of Dipartimento di Matematica, Politecnico di Milano and by GNAMPA (Gruppo Nazionale per l’Analisi Matematica, la Probabilit\`{a} e le loro Applicazioni) of INdAM (Istituto Nazionale di Alta Matematica).
J.-N. He was supported by Zhejiang Provincial Natural Science Foundation of China (Grant No. LQ24A010011) and the National Natural Science Foundation of China (Grant No. 12401251).
H. Wu was partially supported by Natural Science Foundation of Shanghai (No. 25ZR1401023).
\smallskip

\noindent \textbf{Competing interests.}
The authors have no relevant financial or non-financial interests to disclose.
\smallskip

\noindent \textbf{Data availability statement.}
Data sharing not applicable to this article as no datasets were generated or analyzed during the current study.


\end{document}